\documentclass[11pt,english]{article}
\usepackage{lmodern}
\usepackage[T1]{fontenc}
\usepackage[latin9]{inputenc}

\usepackage{geometry}
\usepackage{amsmath,amssymb,amsthm}
\usepackage{mathrsfs}
\usepackage{mathtools}
\DeclarePairedDelimiter{\ceil}{\lceil}{\rceil}
\DeclarePairedDelimiter{\floor}{\lfloor}{\rfloor}
\usepackage{physics}
\usepackage{bm}
\usepackage{latexsym}
\usepackage{esint}
\usepackage{authblk}
\usepackage{graphicx}
\usepackage{subfig}
\usepackage{float}
\usepackage{wrapfig}
\usepackage{graphics}
\graphicspath{{./figures/}}

\usepackage[lined,boxed,linesnumbered,commentsnumbered,ruled,vlined]{algorithm2e}
\SetAlFnt{\smaller}
\SetAlCapNameFnt{\footnotesize}
\SetAlCapFnt{\footnotesize}
\SetArgSty{textnormal}
\usepackage{floatpag}

\usepackage{comment}
\usepackage{enumitem}
\usepackage{pdfpages}
\usepackage{eso-pic}
\usepackage{everyshi}
\usepackage{hyperref}
\usepackage{url}
\usepackage{color}
\usepackage{placeins}
\usepackage{lscape}
\usepackage[english]{babel}
\usepackage{lineno}
\usepackage[toc,page]{appendix}
\usepackage{titlesec}
\usepackage[noabbrev,capitalise]{cleveref}
\usepackage{xfrac}
\usepackage{afterpage} 

\makeatletter
\@ifundefined{showcaptionsetup}{}{
 \PassOptionsToPackage{caption=false}{subfig}}
\makeatother

\usepackage{algpseudocode}

\algnewcommand{\algorithmicand}{\textbf{ and }}
\algnewcommand{\algorithmicor}{\textbf{ or }}
\algnewcommand{\OR}{\algorithmicor}
\algnewcommand{\AND}{\algorithmicand}
\algnewcommand{\var}{\texttt}

\theoremstyle{plain}
\newtheorem{thm}{Theorem}
\theoremstyle{plain}
  \newtheorem{lem}[thm]{Lemma}  

\theoremstyle{plain}
        
\newtheorem{cor}[thm]{Corollary}      

\theoremstyle{definition}

\theoremstyle{remark}
\newtheorem{rem}{Remark}          

\theoremstyle{observ}

\theoremstyle{claim}
\newtheorem{claim}[thm]{Claim}

\def\Area{\mathop{\rm Area}}
\def\AR{\mathop{\rm AR}}

\def\diam{\mathop{\rm diam}}
\def\width{\mathop{\rm width}}
\def\height{\mathop{\rm height}}

\def\FW{\mathop{\rm FW}}

\usepackage{pifont}
\usepackage{tikz}

\newcommand{\rectang}{\mathbin{\text{\tikz [x=1ex,y=1ex,line width=0.1ex,line join=round] \draw (0,0) rectangle (1.5,1) ;}}}
\newcommand{\strip}{\mathbin{\text{\tikz [x=1ex,y=1ex,line width=0.25ex,line join=round] \draw (0,0) rectangle (1,1.5) ;}}}

\providecommand{\keywords}[1]
{
  \small	
  \textbf{\textit{Keywords:}} #1
}

\usepackage{titlesec}
\titlespacing*{\section}{0pt}{3ex plus 1ex minus .2ex}{1.5ex plus .2ex}
\titlespacing*{\subsection}{0pt}{2.5ex plus 1ex minus .2ex}{1.25ex plus .2ex}
\titlespacing*{\subsubsection}{0pt}{2.25ex plus 1ex minus .2ex}{1ex plus .2ex}
\titlespacing*{\paragraph}{0pt}{2.5ex plus 1ex minus .2ex}{1em}

\usepackage[labelfont=bf,tableposition=top,figureposition=bottom]{caption}

\DeclareCaptionLabelSeparator*{spaced}{\\[2ex]}
\usepackage{etoolbox}
 
\makeatletter
\patchcmd{\maketitle}{\@fnsymbol}{\@alph}{}{}  
\makeatother

\usepackage{booktabs}

\begin{document}
\title{Approximating Median Points in a Convex Polygon}

\author[1]{Reyhaneh Mohammadi}
\author[2]{Raghuveer Devulapalli}
\author[3]{Mehdi Behroozi \footnote{Corresponding author: m.behroozi@neu.edu; Phone (617) 373-2032; Address: 449 Snell Engineering Center, 360 Huntington Ave, Boston, MA 02115. M.B. gratefully acknowledges the support of Northeastern University for this research.}}
\affil[1,3]{Department of Mechanical and Industrial Engineering, Northeastern University, Boston, MA, USA}
\affil[1,3]{E-mail addresses: mohammadi.re@northeastern.edu; m.behroozi@neu.edu}
\affil[2]{Intel Corporation, Hillsboro, OR, USA}
\affil[2]{E-mail address: raghuveer.devulapalli@intel.com}

\date{}

\maketitle

\begin{abstract}
We develop two simple and efficient approximation algorithms for the continuous $k$-medians problems, where we seek to find the optimal location of $k$ facilities among a continuum of client points in a convex polygon $C$ with $n$ vertices in a way that the total (average) Euclidean distance between clients and their nearest facility is minimized. Both algorithms run in $\mathcal{O}(n + k + k \log n)$ time. Our algorithms produce solutions within a factor of 2.002 of optimality. In addition, our simulation results applied to the convex hulls of the State of Massachusetts and the Town of Brookline, MA show that our algorithms generally perform within a range of 5\% to 22\% of optimality in practice.
\end{abstract}


\section{Introduction}
\label{sec:Introduction}
The $k$-medians problem is a well-known geometric optimization problem in which the aim is to select $k$ median points in a way to minimize the total distance of clients to their nearest median points with respect to some measure. The most studied setting for this problem is considering the client set as a discrete set. Papadimitriou \cite{1} showed that the $k$-medians problem in the Euclidean plane is NP-Complete. The first constant-factor approximation algorithm for the $k$-medians problem in the plane was proposed by Arora \cite{Arora}. Bartal gave a randomized approximation algorithm in \cite{Bartal1996randomized} to approximate distances on a graph with $n$ vertices with distances on a tree with approximation ratio of $\mathcal{O}(\log^2 n)$. This made it possible for many optimization problems, including $k$-medians problem, where the objective function is some function of edge length to be solved on a tree, which is much easier to solve exactly or approximately due to the structure of a tree, while losing only a factor of $\mathcal{O}(\log^2 n)$. The $k$-medians problem on a general graph is NP-hard, but there are optimal algorithms for solving it on a tree with $n$ vertices  in time $\mathcal{O}(k^2 n^2)$ \cite{kariv1979algorithmic} and time $\mathcal{O}(kn^2)$ \cite{tamir1996pn2}. Later, Bartal improved the ratio to $\mathcal{O}(\log n\:\log\log n)$, therefore providing a randomized approximation algorithm for the $k$-medians problem with the same approximation ratio \cite{bartal1998approximating}. Then Charikar et al. \cite{Charikar1} by subsequently derandomizing and refining this algorithm obtained the first deterministic approximation algorithm with ratio $\mathcal{O}(\log k\:\log\log k)$ of optimality. Charikar and Guha \cite{Charikar} proposed a factor $6\frac{2}{3}$ approximation algorithm for this problem and improved the bound. Then, Jain and Vazirani \cite{Jane1} proposed a 3-approximation algorithm for the uncapacitated facility location problem based on the primal-dual schema and used that as a sub-routine for the $k$-medians problem and proposed a 6-approximation algorithm. Then Jain et al. \cite{Jane2} improved the bound for the uncapacitated facility location problem to 1.61 leading to a 4-approximation algorithm for the $k$-medians problem. Li and Svensson \cite{Li} presented an approximation algorithm for the $k$-medians problem and via a 2-step proof showed that the approximation guarantee is $1+\sqrt{3}+\epsilon$. In the first step, they showed that to give an $\alpha$-approximation
algorithm for the $k$-medians problem, it suffices to have a pseudo-approximation algorithm that generates an $\alpha$-approximate solution by opening $k+O(1)$ facilities and in the second step, they designed a pseudo-approximation algorithm
with $\alpha = 1+\sqrt{3}+\epsilon$. This bound was slightly improved to $2.675 + \epsilon$ by Byrka et al. \cite{byrka2017improved}.
\\
In the other setting of this problem, the number of clients are so large that we could consider the client set as a continuous region. This setting makes the corresponding discrete $k$-medians problem intractable. However, assuming a continuous distribution of clients could help us in modeling and solving it in an easier way.
The first attempt for solving this problem was made by Fekete et al. in \cite{Fekete}, which provides polynomial-time algorithms for different versions of the 1-median (Fermat-Weber) problem using $L_1$ norm. They also showed that the multiple-center version of this problem which is the $k$-median problem with $L_1$ distance, is NP-hard for large $k$. Later, Carlsson et al. \cite{Carlsson2014kmedian} developed a simple factor 2.74 approximation algorithm for the continuous $k$-medians problem in a convex polygon using $L_2$ norm. 

In this work, we propose two simple constant-factor approximation algorithms for the continuous $k$-medians problem in a convex polygon $C$ with $n$ vertices under the $L_2$ norm. Our theoretical analysis shows that our algorithms always produces solutions within a factor of 2.002 of optimality in $\mathcal{O}(n + k + k \log n)$ time. To the best of our knowledge, this provides a significant improvement to the previous best results for this problem. We have also proposed a modification to our algorithms in order to improve the average practical performance of the algorithms. Our simulation results applied to the convex hulls of the State of Massachusetts and the Town of Brookline, MA show that in practice, our algorithms generally perform much better than the worst case analysis and their solutions are  within a range of 5\% to 22\% of optimality.

\subsection{Notational Conventions}
\label{subsec:Notations}
We define the Fermat-Weber value of a region $C$ relative to the point $p \in C$ as
\[
\FW(C,p)= \iint_C{\|x-p\|\,dA},
\]
and Fermat-Weber value of region $C$ as
\[
\FW(C)= \min_{p}\iint_C{\|x-p\|\,dA}\,.
\]
We also define
\[
\FW(C,k)= \min_{P:|P|=k}\iint_C{\;\min_{p_i\in P}\|x-p_i\|\,dA}\, , 
\]
to be the Fermat-Weber objective function that we try to minimize, where $\|\cdot\|$ denotes the Euclidean norm and we are looking for a set $P$ of $k$ points. Note that $\FW(C)=\FW(C,1)$. When the region is a rectangle $R$ with dimensions $w$ and $h$, for simplicity of analysis, we show $\FW(R,p),\;\FW(R)$, and $\FW(R,k)$ with $\FW_{\rectang}(w,h,p),\; \FW_{\rectang}(w,h)$, and $\FW_{\rectang}(w,h,k)$, respectively. We also show half of $\FW_{\rectang}(w,h)$ with $\FW_{1/2}(w,h)$.
Let $\square C$ denote the minimum-area bounding box of $C$ and let $\width(C)$ and $\height(C)$ denote the horizontal and vertical dimensions of $\square C$. Let $\Area(C)$ and $\diam(C)$ denote the area and diameter of $C$, respectively. For a rectangle $R$, let $\AR(R)$ be the aspect ratio of $R$ which is defined as  $\AR(R)=\max\left\{\frac{\height(R)}{\width(R)}, \: \frac{\width(R)}{\height(R)}\right\}$. Finally, we use $\rho$ to denote the approximation factor of an algorithm.

\section{General Framework of our Approach}
\label{sec:GeneralFramework}
We first present a general framework for our algorithms for solving the continuous $k$-medians problem in a convex polygon. Our approach is aligned with that of \cite{Carlsson2014kmedian} and can be summarized in as follows: the input to our algorithms is a convex polygon $C$ with $n$ vertices and an integer $k$. We assume without loss of generality that $C$ is aligned so that its diameter coincides with the coordinate $x$-axis. We then enclose $C$ in an axis-aligned minimum bounding box $R=\square C$ of dimensions $w$ and $h$, where $w=\diam\left(C\right)$. By aligning the diameter with the $x$-axis, we must have:
\begin{equation}
\label{eq:BoxAreaRelation} 
\frac{wh}{2}\leq\Area\left(C\right)\leq wh
\end{equation} 
Note that \eqref{eq:BoxAreaRelation} holds for a general minimum bounding box of a convex region too (the lower and upper bounds are tight when $C$ is a triangle and a rectangle, respectively). 

We then somehow divide $\square C$ into rectangular pieces of areas $\Area\left(\square C\right)/k$. The way this partitioning is done will have a significant impact on the performance of the algorithm and will be discussed in Sections \ref{sec:ConstructiveAlg} and \ref{sec:SubdivisionAlg}. Then, the center of each sub-rectangle is chosen as our landmark points. It may happen that some of these points may lie outside the convex polygon  (Figure \ref{fig:alg-c}). Such points are simply relocated back into the polygon (Figure \ref{fig:alg-d}). The way we to do this operation turns out to have a minimal impact on the performance of the algorithm.  This approach is illustrated in Figure \ref{fig:An-equal-area-partition}. Finally, without loss of generality (w.l.o.g.) and for simplicity of our abalysis, we can consider the bounding box $R$ to be of dimensions $w$ and $1$, where $w=\diam(C)\geq1$. It follows that $\frac{w}{2}\leq A\leq w$. The algorithm breaks $R$ into $k$ sub-rectangles of equal area $\frac{w}{k}$.

\begin{figure}[t]
\begin{centering}
\subfloat[\label{fig:alg-a}]{\begin{centering}
\includegraphics[width=0.25\columnwidth]{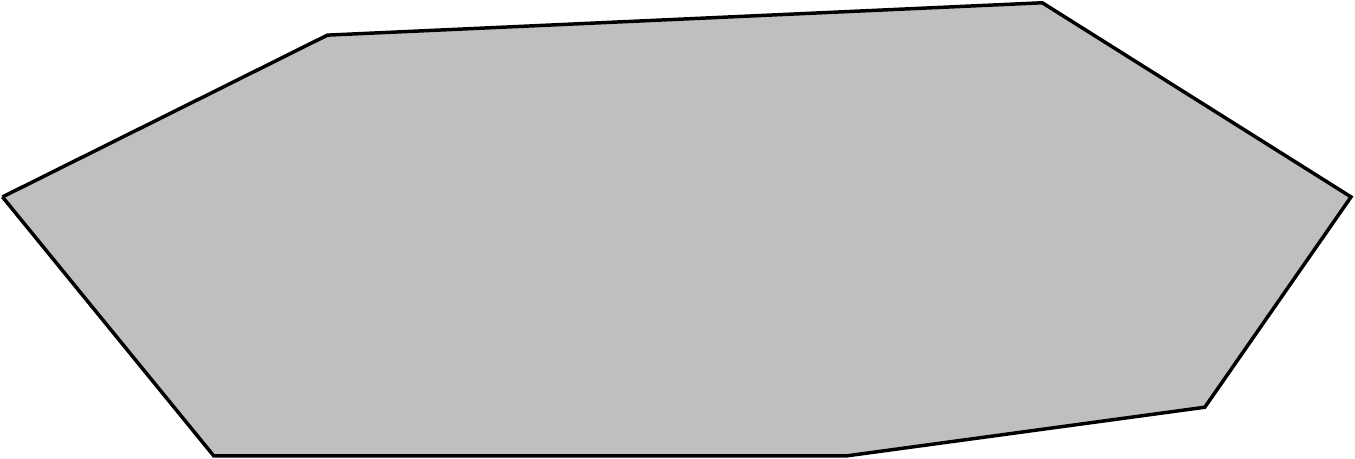}
\par\end{centering}

}\quad{}\subfloat[\label{fig:alg-b}]{\begin{centering}
\includegraphics[width=0.25\columnwidth]{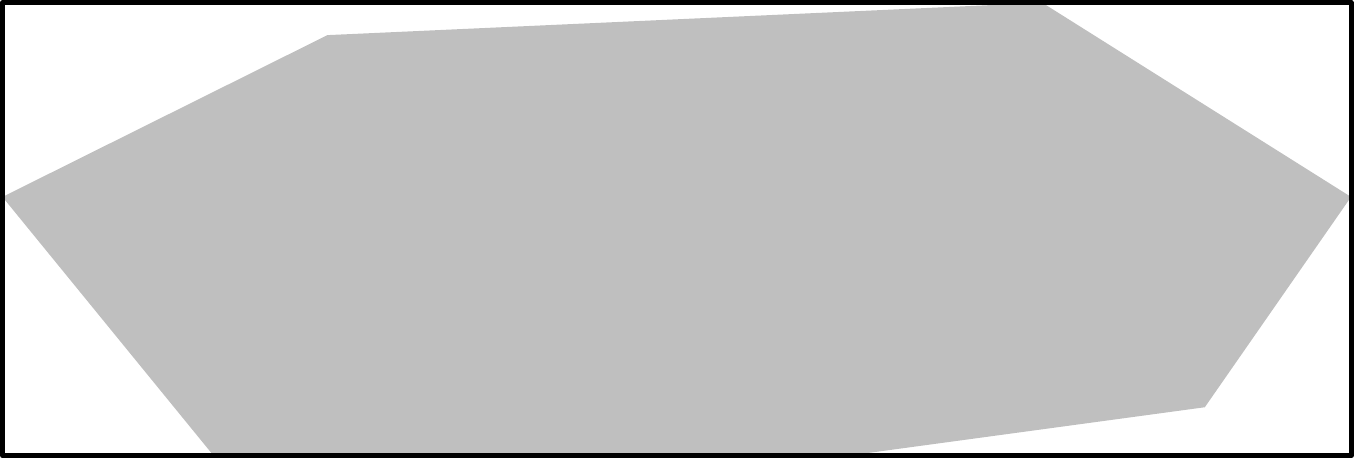}
\par\end{centering}

}\quad{}\subfloat[\label{fig:alg-c}]{\begin{centering}
\includegraphics[width=0.25\columnwidth]{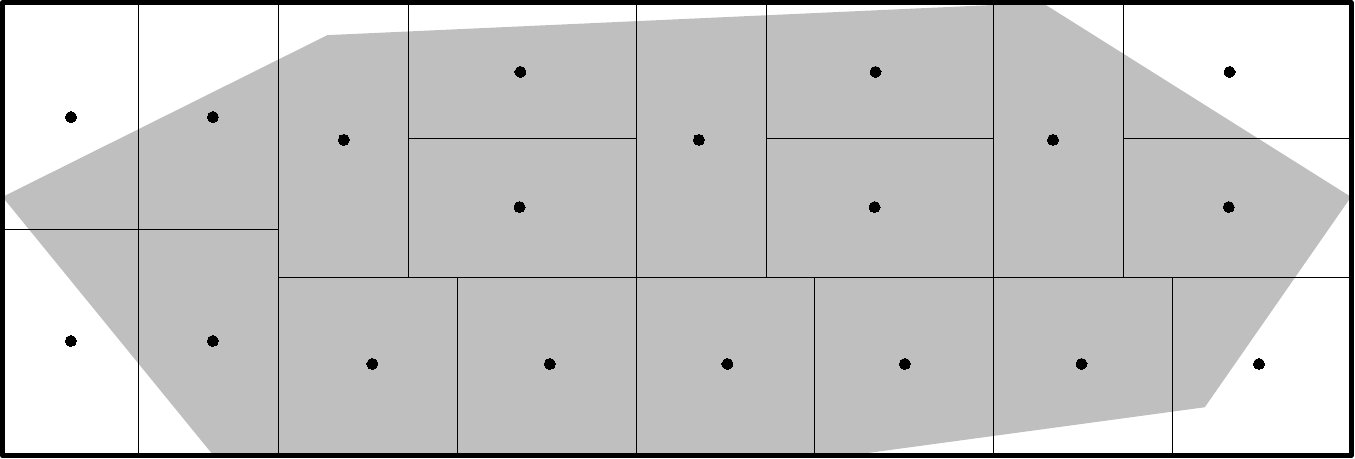}
\par\end{centering}

}
\par\end{centering}

\begin{centering}
\subfloat[\label{fig:alg-d}]{\begin{centering}
\includegraphics[width=0.25\columnwidth]{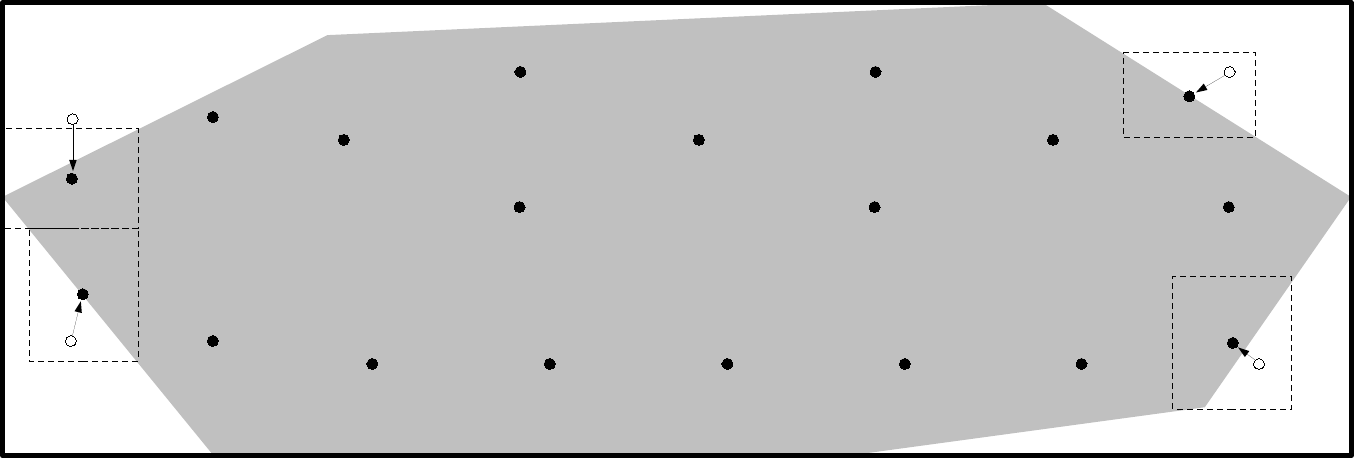}
\par\end{centering}

}\quad{}\subfloat[\label{fig:alg-e}]{\begin{centering}
\includegraphics[width=0.25\columnwidth]{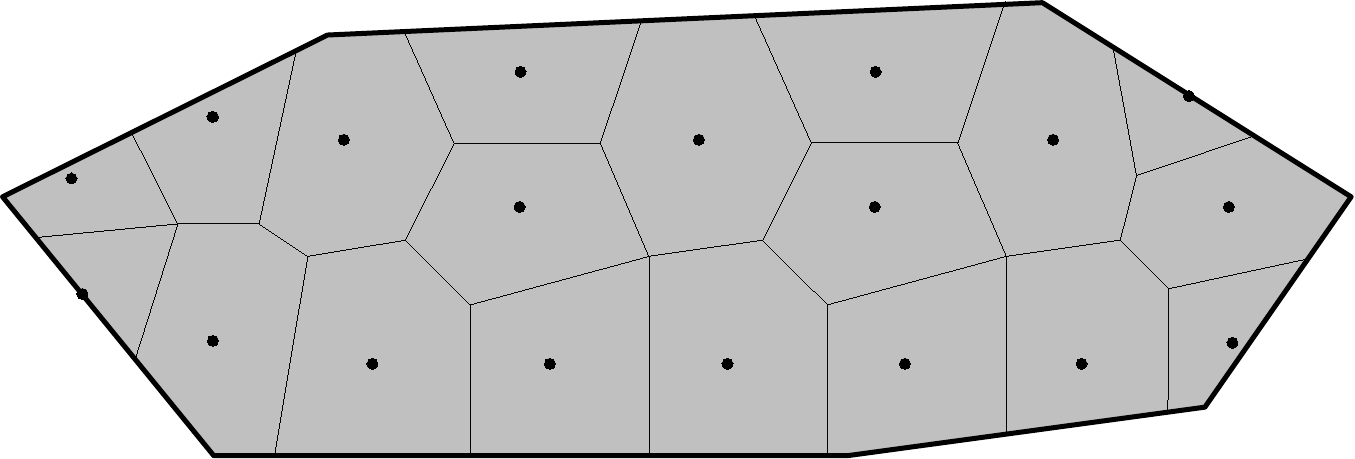}
\par\end{centering}

}
\par\end{centering}

\caption{\label{fig:An-equal-area-partition}We begin in (\ref{fig:alg-a}) with a convex polygon $C$, whose axis-aligned bounding box $\square C$ is computed in (\ref{fig:alg-b}). The bounding box is then partitioned into equal-area pieces in (\ref{fig:alg-c}). Some of the centers of these pieces are then relocated in (\ref{fig:alg-d}), and (\ref{fig:alg-e}) shows the output and Voronoi partition.  The above images are reproduced from Figure 1 of \cite{Carlsson2014kmedian}.}
\end{figure}

\section{Upper and Lower Bounds}
Given a convex object $C$ of area $A$, we want to find upper and lower bounds on $\FW(C)$. The main motivation to find these bounds is that, we want to use them in finding upper and lower bounds for the $k$-medians problem ($\FW(C,k)$). Note that the papers \cite{carmi2005fermat}, \cite{AbuAffash} and \cite{DumitrescuFWBound} almost entirely deal with finding upper and lower bounds on $\FW(C)$ in terms of $A$ and $d=\diam(C)$. Summarized in one line, they prove the following bounds: $0.1666dA \leq \FW(C) \leq 0.3490 d A$. 

Unfortunately, these bounds are not useful in our work for two reasons. First, a lower bound in terms of diameter of $C$ does not help us achieve a lower bound on $\FW(C,k)$. And secondly, the upper bound can be significantly improved when the convex region is ``skinny'' (it is easy to show that the Fermat-Weber of a very long skinny rectangle is approximately $0.25dA$). The authors in \cite{AbuAffash} arrive at their upper bound by making using of the fact that a convex object of diameter $d$ can always be contained in a circle of radius $d/\sqrt{3}$. Rather than using a bounding circle, we find it more useful to derive bounds in terms of a bounding box that contains $C$. This is because for a long and skinny region, the area of the minimum bounding circle can be infinitely larger than the area of the convex region itself. But if one considers the minimum bounding box of $C$, we can prove that the ratio of its area to that of the convex polygon is at most 2 for any convex region (see equation \eqref{eq:BoxAreaRelation}).

In the following sections, we will derive lower and upper bounds for a convex polygon $C$ of area $A$ that is contained within a rectangle of dimensions $w$ and $h$, where $A \in [0,wh]$ (this rectangle need not necessarily be the minimum bounding box of $C$). The Fermat-Weber value of the rectangle itself is denoted using $\FW_{\rectang}(w,h)$. A closed form expression for $\FW_{\rectang}(w,h)$ is given in Section \ref{sec:FWCommonObjects} of the Online Supplement.

\subsection{Lower bounding $\FW(C)$}
\label{sec:OneMedianLB}
Here, we derive a simple lower bound that we will use to place an overall lower bound on $\FW(C,k)$. For a given area $A$, it is a well known fact that disk is the shape that minimizes its Fermat-Weber value. Hence, using the Fermat-Weber of a disk, we can arrive at an easy lower bound on $\FW(C)$: $\Phi_{LB_1}(A) = \frac{2}{3\sqrt{\pi}}A^{3/2}$ (see Remark \ref{rem:disk} in Section \ref{sec:FWCommonObjects} of the Online Supplement). In most cases, this simple lower bound is large enough for our analysis. But, in some cases, we have skinny regions and there might be a big gap between this simple lower bound and the actual Fermat-Weber value. For such cases we consider the region to be the intersection of a disk of radius $r$ and a slab of height $h$ with $r\geq\frac{h}{2}$. As you can also see in Figure \ref{fig:LB}, the area of such region can be written in terms of $r$ and $h$. Calculating the Fermat-Weber of such region we obtain another lower bound on $\FW(C)$: $\Phi_{LB_2}(A_{(r,h)})=\frac{4r^{3}}{3}\sin^{-1}\left(\frac{h}{2r}\right)+\frac{1}{3}rh\sqrt{r^{2}-\frac{h^{2}}{4}}+\frac{1}{12}h^{3}\log\left(\frac{2r+\sqrt{4r^{2}-h^{2}}}{h}\right)$. Combining these with the fact that the region is contained within a box of dimensions $w$ and $h$, we can obtain a comprehensive lower bound as summarized in the following lemma. 

\begin{lem} 
Let $C$ be a convex region with area $A$, contained in a box $B$ of dimensions $w$ and $h$, with $w\geq h$. If $B'$ is a horizontal slab of height $h$ that contains $B$ and $D$ is a disk of radius $r$ centered at the centroid of the box such that $\Area\left(D\cap B'\right)=A$, then:
\end{lem}
\begin{equation}
\FW\left(C\right)\geq\begin{cases}
\frac{2}{3}\pi r^{3} & \mbox{if }r<\frac{h}{2}\\
\frac{4r^{3}}{3}\sin^{-1}\left(\frac{h}{2r}\right)+\frac{1}{3}rh\sqrt{r^{2}-\frac{h^{2}}{4}}+\frac{1}{12}h^{3}\log\left(\frac{2r+\sqrt{4r^{2}-h^{2}}}{h}\right) & \mbox{if }r\geq\frac{h}{2}
\end{cases}
\label{eq:LowerBound}
\end{equation}

\begin{proof}

\begin{figure}[t]
  \centering
  \subfloat[Region $C^{*}$ of minimal Fermat-Weber.]{\label{fig:hexagon1} \includegraphics[width=0.3\textwidth]{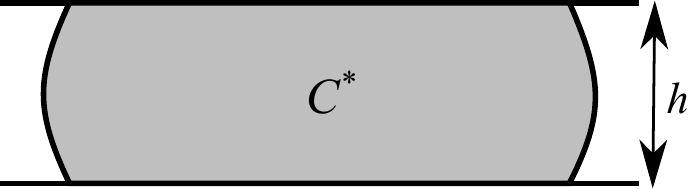}}\,\,\,\,\,\,\,\,\,\,\,\,\,\,\,\,\,\,\,\,\,\,\,
  \subfloat[Breaking $C^{*}$ into triangles and sectors.]{\label{fig:hexagon2} \includegraphics[width=0.3\textwidth]{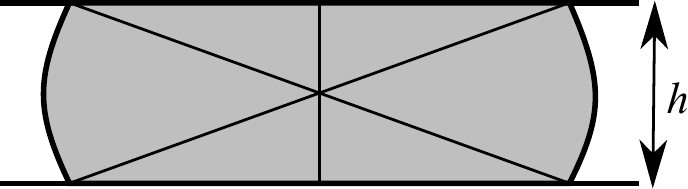}}
  \caption{A narrow convex region of area $A$ formed by intersection of a disk with radius $r$ and a slab of height $h$, that provides a lower bound for the Fermat-Weber value.}
  \label{fig:LB}
\end{figure}

Refer to Figure \ref{fig:LB} for this proof. The shape $C^{*}$ that minimizes $\FW\left(C\right)$ in $B'$ is the intersection of a disk of radius $r$ with a slab of height $h$, and its Fermat-Weber value varies according to $r$ and $h$.

\paragraph{Case~1:} If $r < h/2$, then $C^{*}$ is a disk and hence we have $\FW\left(C^{*}\right)=\frac{2}{3}\pi r^{3}$ (from Lemma \ref{lem:trivial-disk} in Section \ref{sec:FWCommonObjects} of the Online Supplement).

\paragraph{Case~2:} If $r \geq h/2$, then $C^{*}$ is the intersection
of a disk and a slab of height $h$, and the Fermat-Weber of such
a region can be found out by evaluating the integral $\iint_{C^{*}}\sqrt{x_{1}^{2}+x_{2}^{2}}\, dx_{1}\, dx_{2}\,$. An easy way to do this is to break $C^{*}$ into regions consisting
of right-angled triangles and circular sectors (Figure \ref{fig:hexagon2})
and then use Lemmas \ref{lem:trivial-disk} and \ref{lem:FWTriangle} of Section \ref{sec:FWCommonObjects} of the Online Supplement to arrive at the lower bound.
\end{proof} 
We will denote the lower bound in (\ref{eq:LowerBound}) using $\Phi_{LB}\left(A,h\right)=\FW\left(C^{*}\right)$. 
Note that the $\Phi_{LB}$ is independent of the width of the rectangle (since we have assumed that $w\geq h$).

\subsection{Lower Bounds on $\FW(C,k)$}
\label{sec:LB}

\begin{figure}
  \centering
  \subfloat[]{\label{fig:lb_1} \includegraphics[width=0.2\textwidth]{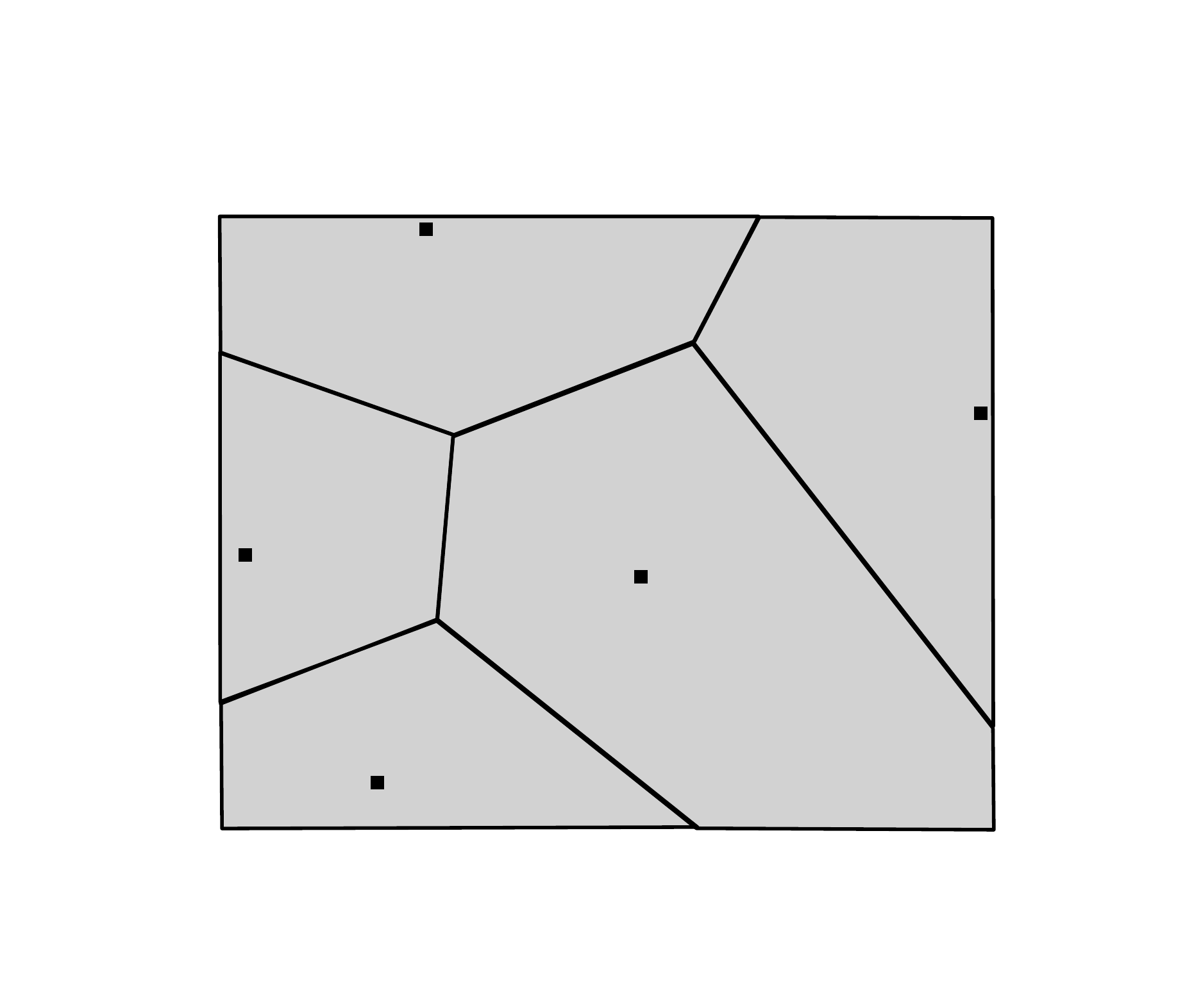}}\,\,\,\,\,\,
  \subfloat[]{\label{fig:lb_2} \includegraphics[width=0.2\textwidth]{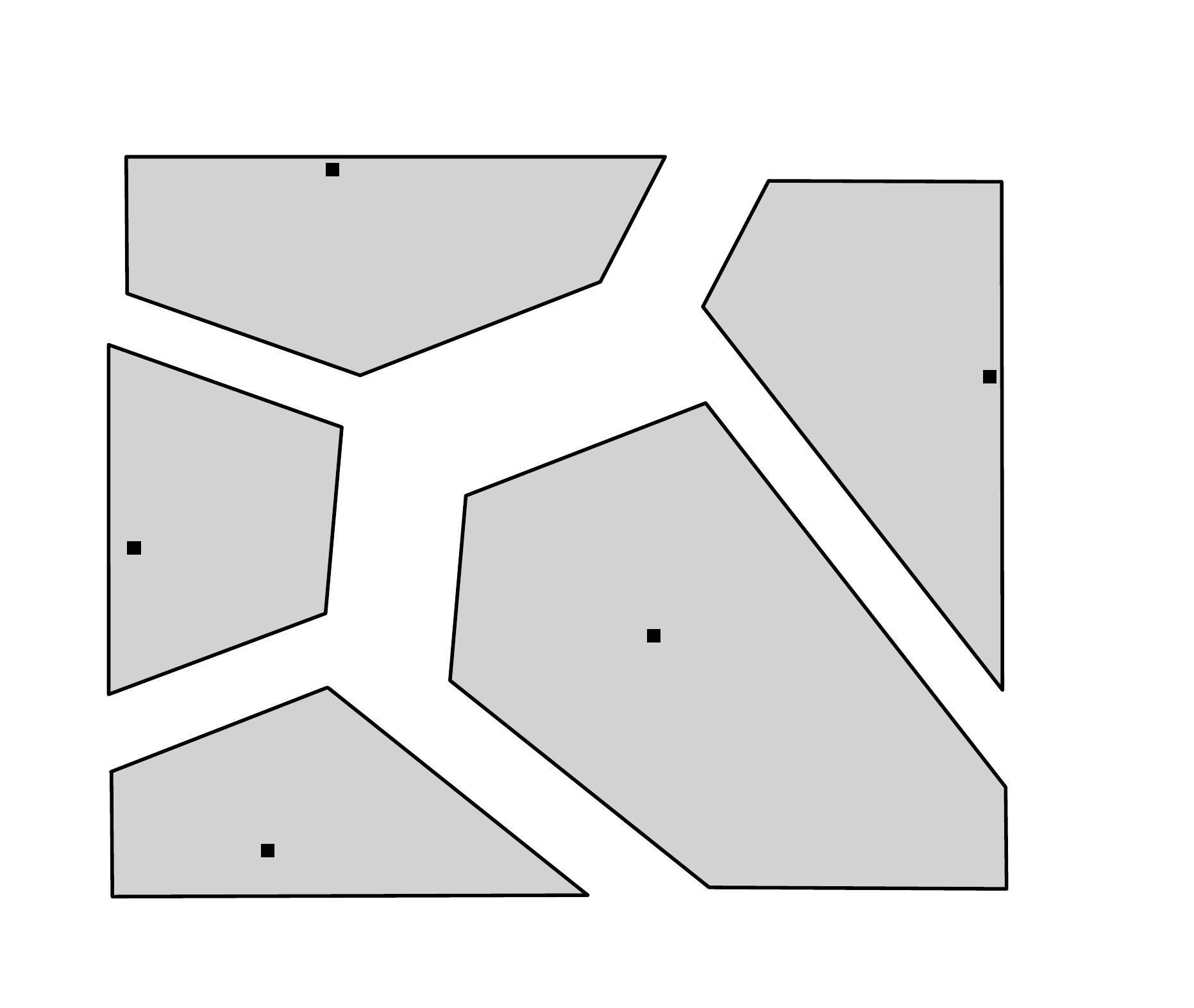}}\,\,\,\,\,\,
  \subfloat[]{\label{fig:lb_3} \includegraphics[width=0.2\textwidth]{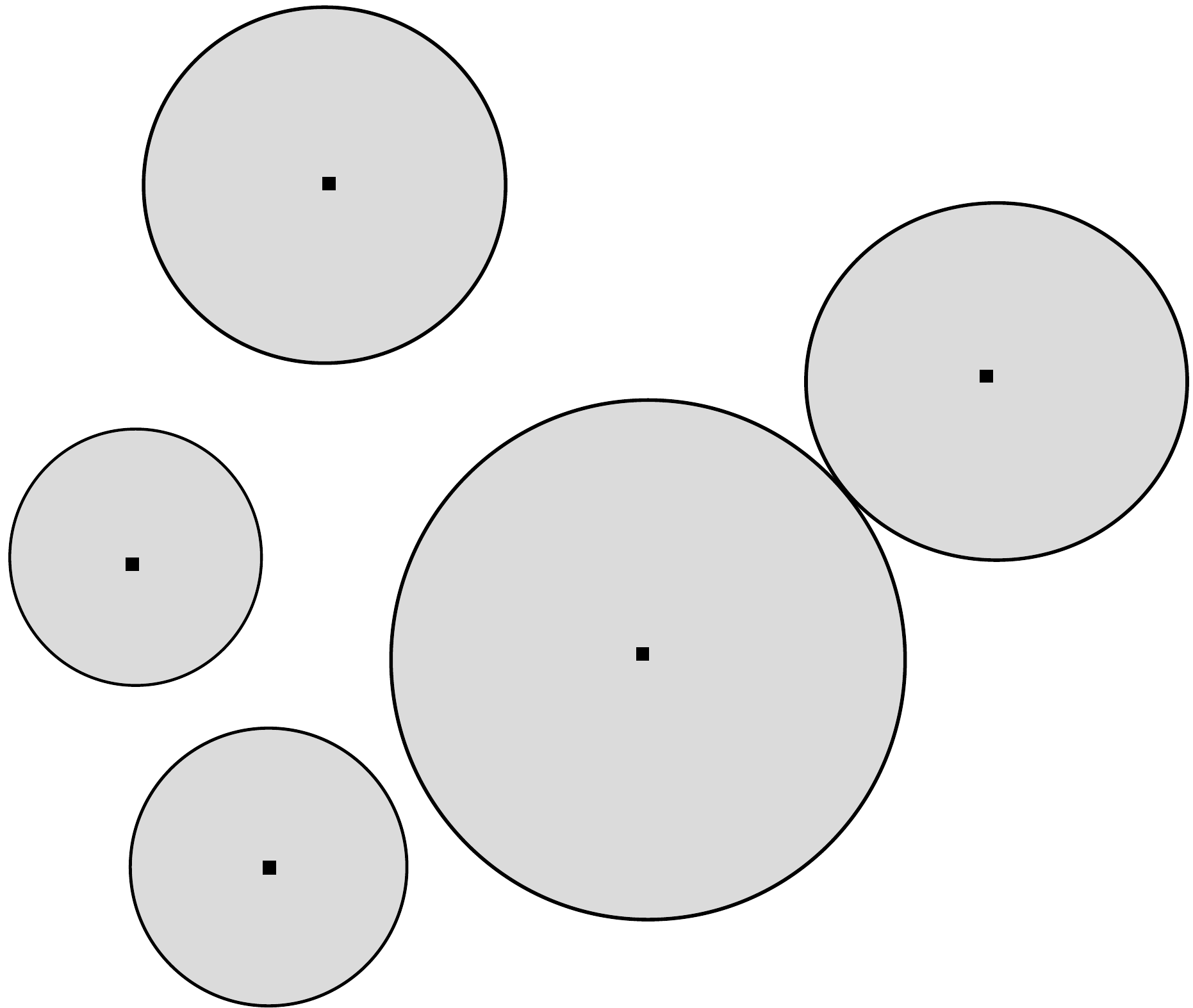}}\,\,\,\,\,\,
  \subfloat[]{\label{fig:lb_4} \includegraphics[width=0.2\textwidth]{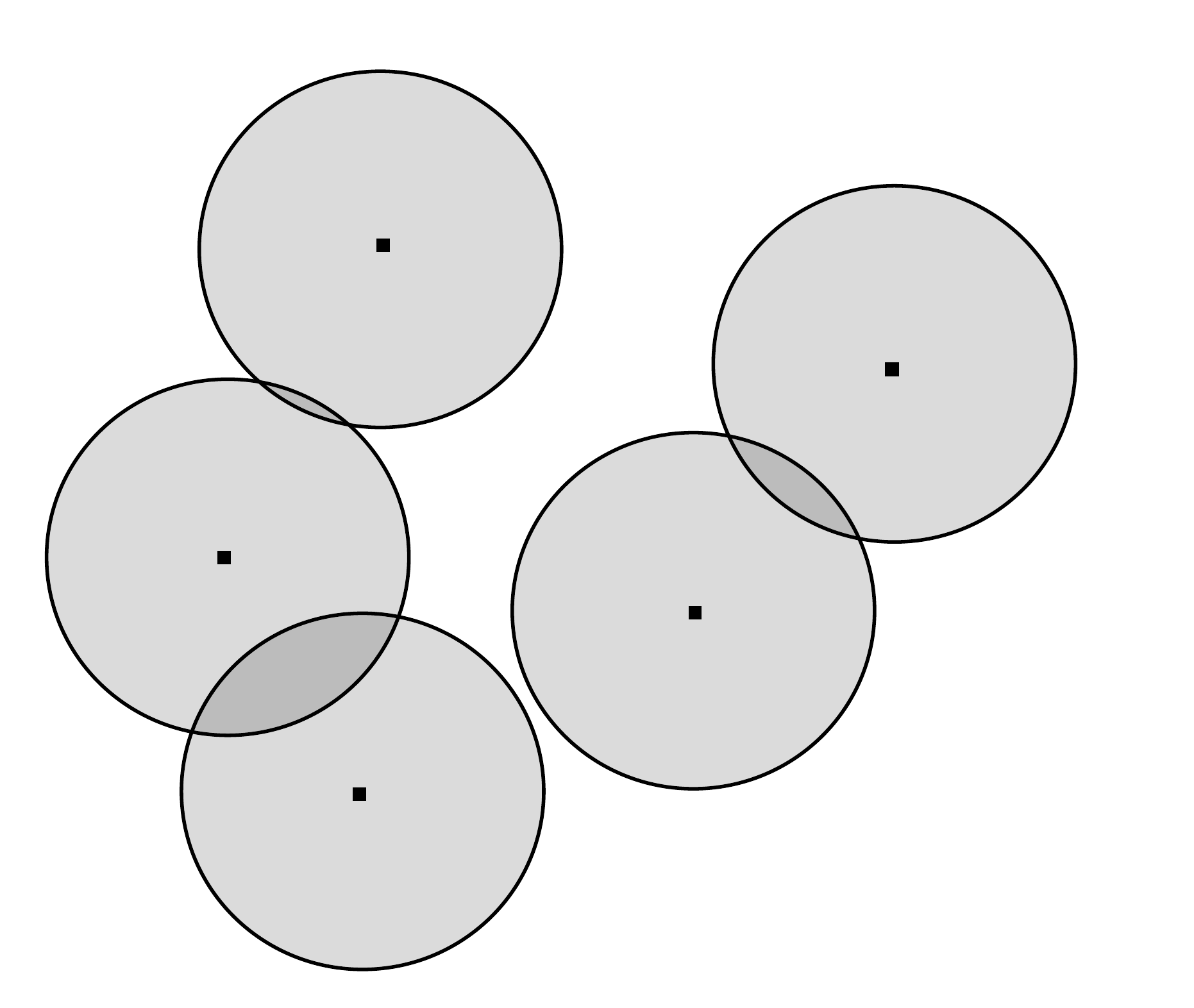}}
  \caption{Lower bounding $\FW(C,k)$: Fermat-Weber of each Voronoi cell can be lower bounded by a disk. The sum of Fermat-Weber of disks is minimized when the areas of all disks are equal, which follows from convexity of $\Phi_{LB_1}$. The same argument holds if we were to replace the disks by slabs ($\Phi_{LB_2}$) shown in Figure \ref{fig:hexagon1}.}
  \label{fig:LowerBoundFWk}
\end{figure}

Let set of points $\{p^*_1,...,p^*_k\}$ be optimum solution for polygon $C$, which is itself embedded in a box of dimensions $w$ and $1$. If each of $k$ points is assigned area $A^*_{i}$ (i.e., if the area of the Voronoi cell associated with the $i^{th}$ point is $A^*_{i}$), then the Fermat-Weber value of $C$ must satisfy (refer to Figure \ref{fig:LowerBoundFWk}):
\[
\FW\left(C,k\right) \geq\sum_{i=1}^{k}\Phi_{LB_1}\left(A^*_{i}\right)\,
\]
Since each of the summands is \emph{convex} as a function of $A^*_{i}$, the right-hand side is \emph{minimized} when each of the points is assigned an equal area $A^*_{i}=A/k$ (Figure \ref{fig:lb_4}). We therefore have: 
\begin{equation}
\label{eq:FirstLB}
\FW\left(C,k\right) \geq k\cdot\Phi_{LB_1}\left(A/k\right)
\end{equation}
Replacing $\Phi_{LB_1}$ with $\Phi_{LB}$ gives us a better lower bound and because $\Phi_{LB}$ is also convex in $A$, all the arguments still hold. 
\begin{equation}
\FW\left(C,k\right)\geq k\cdot\Phi_{LB}\left(A/k,1\right)\,.
\label{eq:FinalLB}
\end{equation}

\subsection{Upper bounding $\FW\left(C\right)$}
\label{sec:UB}
To find an upper bound on $\FW(C)$, we will answer the following question: What is the shape of the convex region of a given area $A$, that is contained within a box $w$ and $h$ and has the maximum Fermat-Weber quantity? Intuitively skinnier regions have larger Fermat-Weber values. But to find the exact shape of such a region is not easy. Interestingly, this becomes tractable if $A = wh/2$. In Lemma \ref{lem:UBHalfRectangle} we will provide an upper bound for the case $A = wh/2$. And in Lemma \ref{lem:UB}, we will extend this result to find an upper bound for all $A \in [0,wh]$. Here on, we will use $\FW_{1/2}(w,h)$ to denote $0.5\FW_{\rectang}(w,h)$, which is half the Fermat-Weber value of the rectangle. 

\begin{lem}
\label{lem:UBHalfRectangle} Let $C$ be a convex region of area $A$ that is contained within a box $B$ of dimensions $w$ and $h$. If $A = \frac{wh}{2}$, then we have:
\[
\FW(C) \leq {\FW}_{1/2}(w,h) 
\]
\end{lem}
\begin{proof}
\begin{figure}[htb]
\begin{centering}
\subfloat[\label{fig:eq-FW-C-1}]{\begin{centering}
\includegraphics[width=0.15\columnwidth]{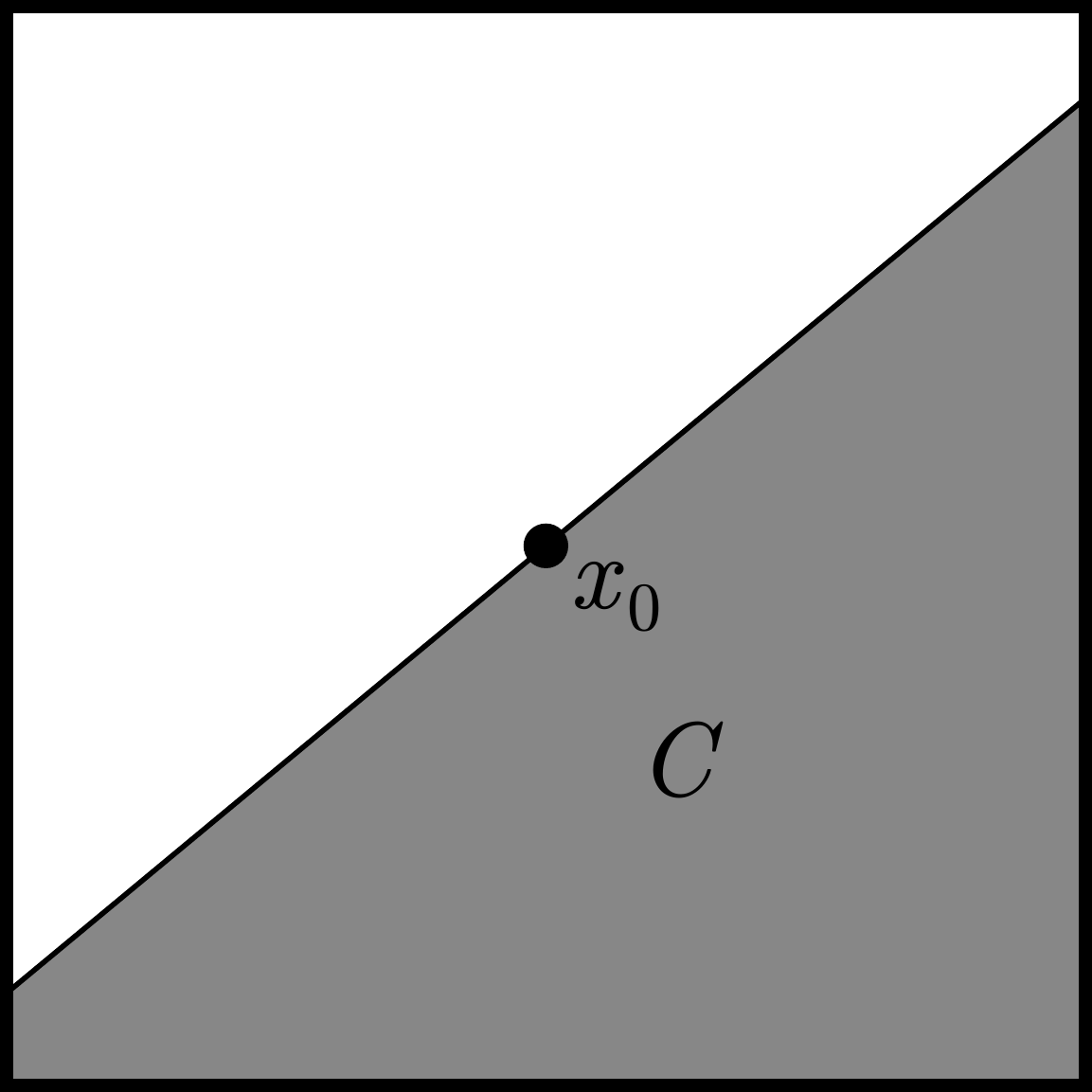}
\par\end{centering}

}\subfloat[\label{fig:eq-FW-C-2}]{\begin{centering}
\includegraphics[width=0.15\columnwidth]{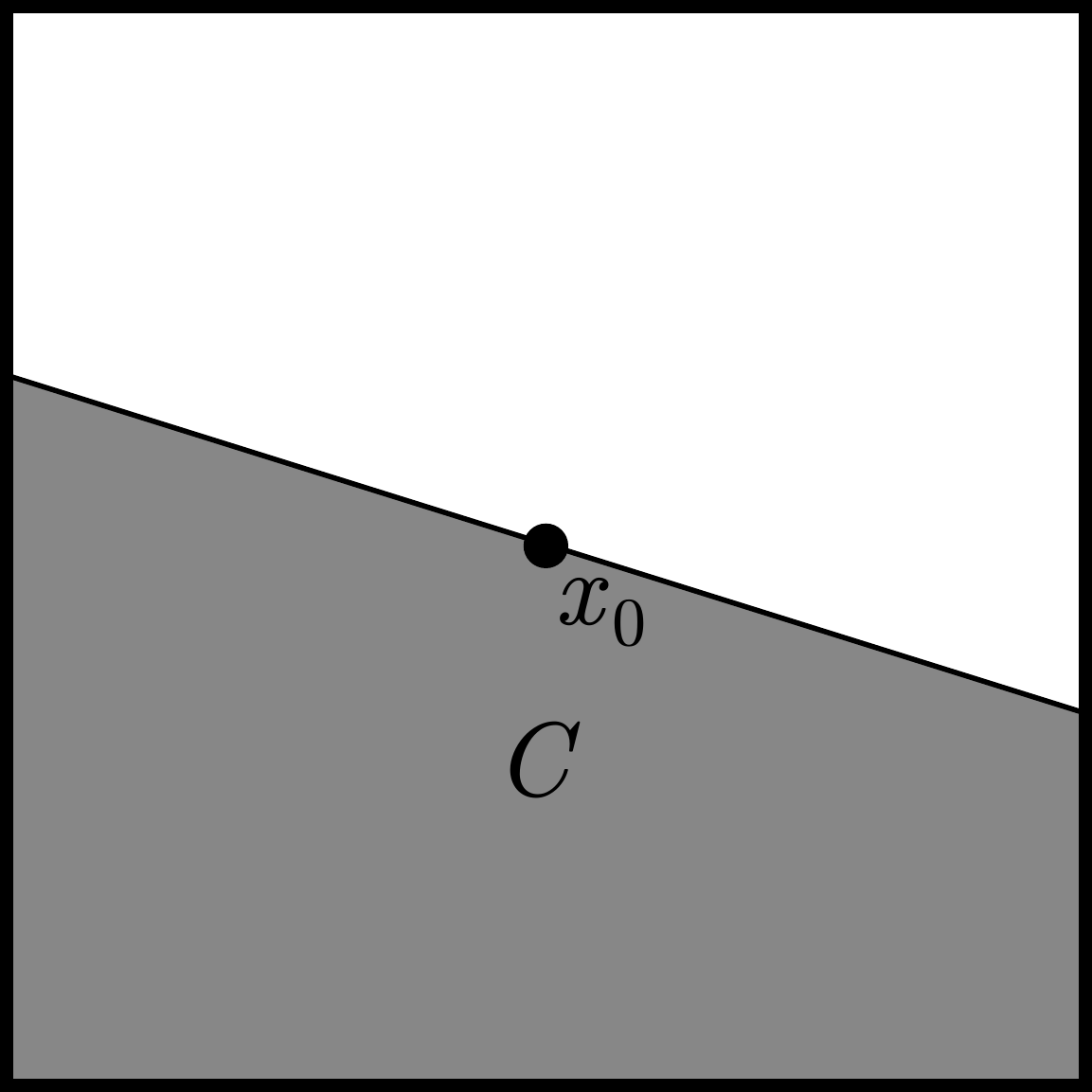}
\par\end{centering}

}\subfloat[\label{fig:eq-FW-C-3}]{\begin{centering}
\includegraphics[width=0.15\columnwidth]{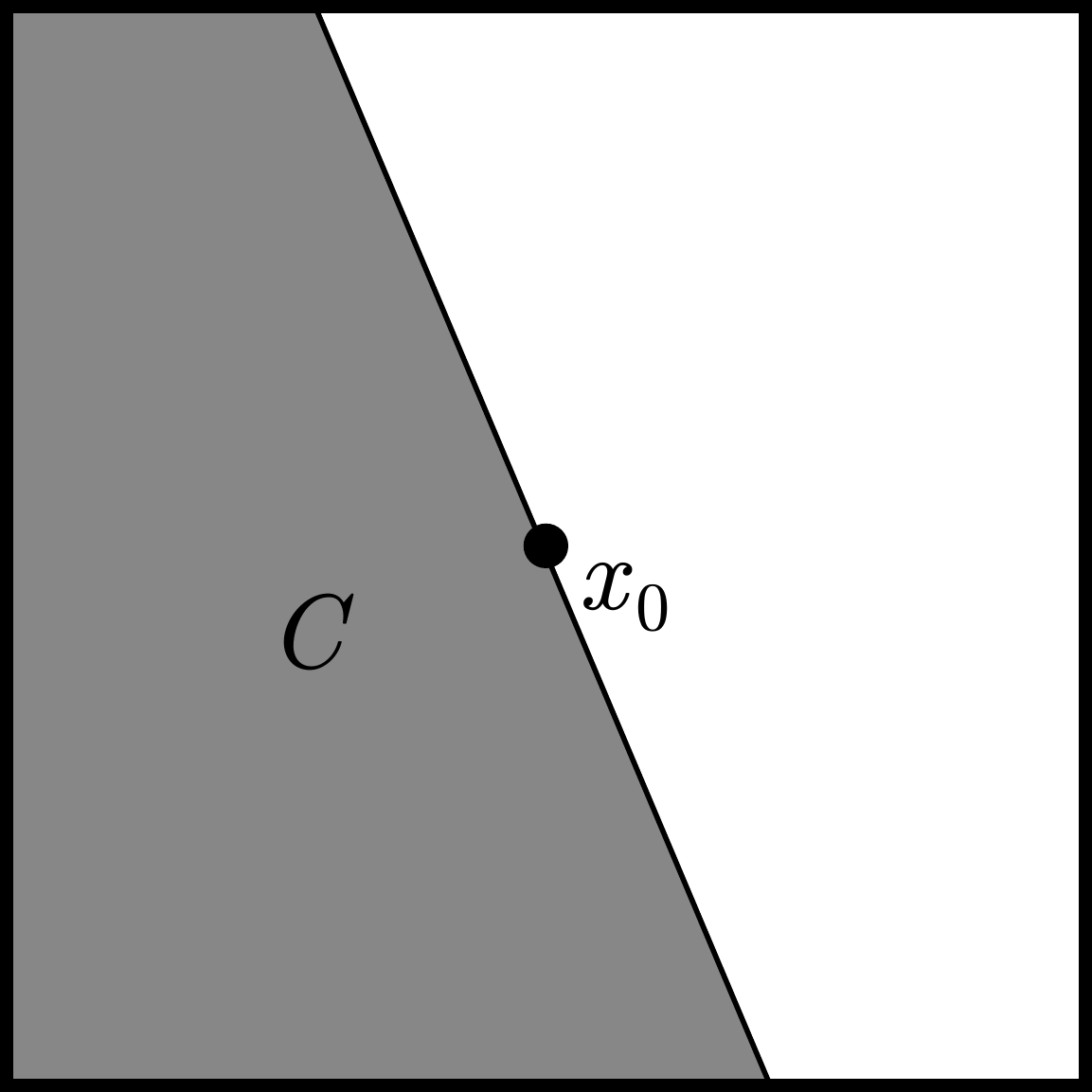}
\par\end{centering}

}
\par\end{centering}

\caption{\label{fig:eq-FW-C}It is obvious that all regions $C$ above have
the same Fermat-Weber value relative to $x_{0}$, namely $\FW_{1/2}(w,h)$. \vspace{5pt}}
\end{figure}
Let $x_{0}$ denote the center of $B$. By definition of Fermat-Weber, for any convex region $C$ we have: $\FW(C) \leq \FW(C,x_{0})$. Consider the family $\mathcal{C}$
consisting of all convex regions $C$ in $B$ which merely consist
of the area lying to one side of a line through $x_{0}$, as shown
in Figure \ref{fig:eq-FW-C}. It is obvious that, for any region $C^{'}\in\mathcal{C}$,
we have $\FW(C^{'},x_{0})=\FW_{1/2}(w,h)$ by a trivial symmetry argument.
Therefore, it will suffice to show that we can take \emph{any} convex
region $C$ in $B$ with area $wh/2$ and transform it into a region
$C^{'}\in\mathcal{C}$ in a manner that does not decrease $\FW(C,x_{0})$
. This will basically prove our result: $\FW(C) \leq \FW(C,x_{0}) \leq \FW(\mathcal{C},x_{0}) = \FW_{1/2}(w,h)$. We will show how to execute this transformation after making two
straightforward observations:
\begin{lem}
\label{lem:line-segment-convex}Let $s$ denote a line segment in
the plane and let $\vec{v}$ be a vector in the plane, so that $s+t\vec{v}$
is a translation operator on $s$ in the direction $\vec{v}$. Then
the function $f(t)=\int_{s+t\vec{v}}||x||\, dx$ is convex in $t$.\end{lem}
\begin{proof}
The function $f(t)$ is simply the Fermat-Weber value of the line
segment $s+t\vec{v}$ relative to the origin; it is an infinite sum
of convex functions and is therefore itself convex.\end{proof}
\begin{lem}
\label{lem:shearing-convex}For any triangle $T$ with vertices $\{x_{1},x_{2},x_{3}\}$,
let $T(t)$ denote the sheared triangle $\{x_{1},x_{2},x_{3}+t(x_{2}-x_{1})\}$.
Then the function $g(t)=\FW(T(t),x_{0})$ is convex in $t$ for any
fixed $x_{0}$.\end{lem}
\begin{proof}
This is a corollary of Lemma \ref{lem:line-segment-convex} because
$g(t)$ is an infinite sum of convex functions $f(t)$.
\end{proof}
We can now proceed to prove Lemma \ref{lem:UBHalfRectangle}. Consider any
convex region $C$ with area $wh/2$, contained inside a box $B$
with dimensions $w\times h$, as in the statement of Lemma \ref{lem:UBHalfRectangle}.

\begin{figure}[t]
\begin{centering}
\subfloat[\label{fig:vertical-segs}]{\begin{centering}
\includegraphics[width=0.15\columnwidth]{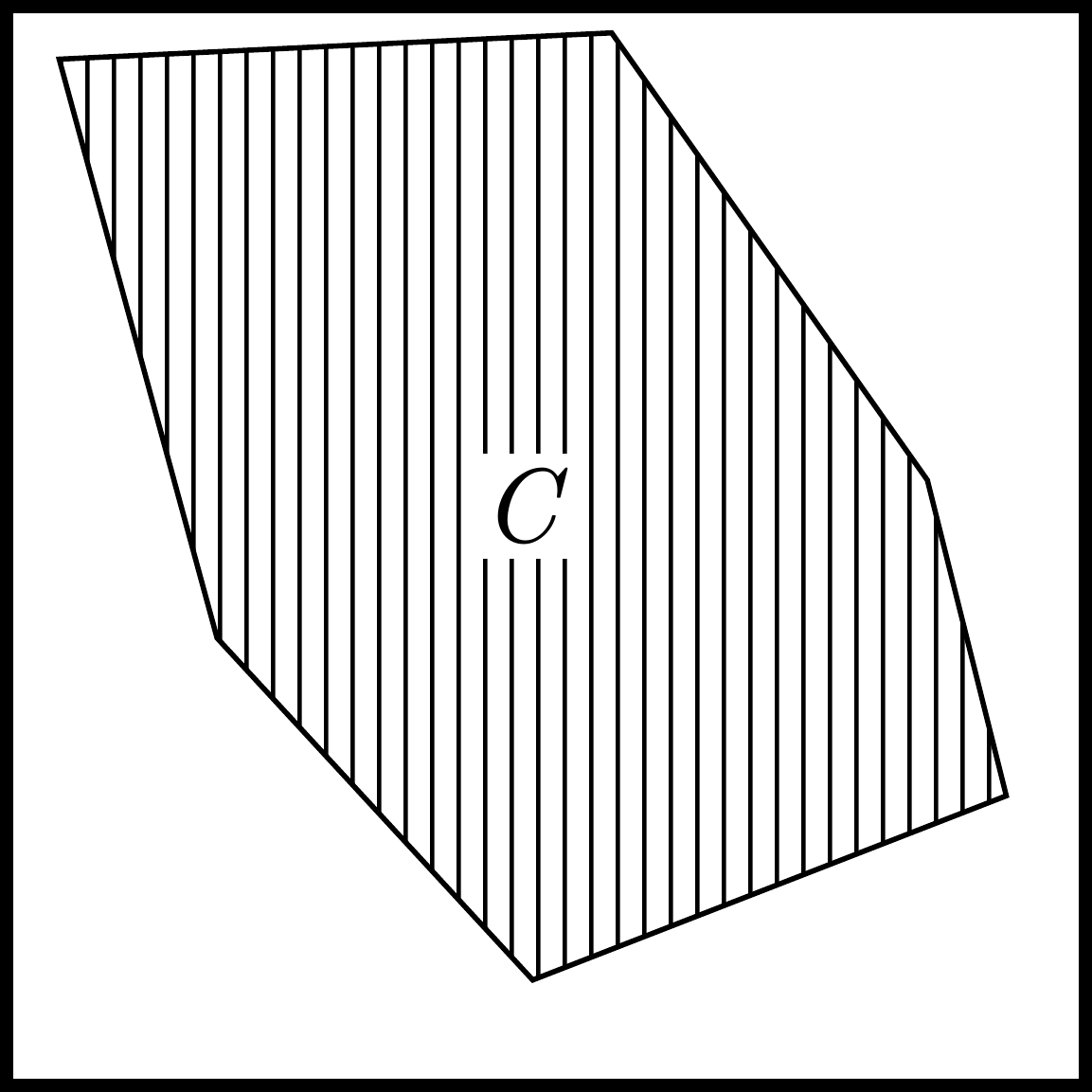}
\par\end{centering}
}\subfloat[\label{fig:shifted-to-bottom}]{\begin{centering}
\includegraphics[width=0.15\columnwidth]{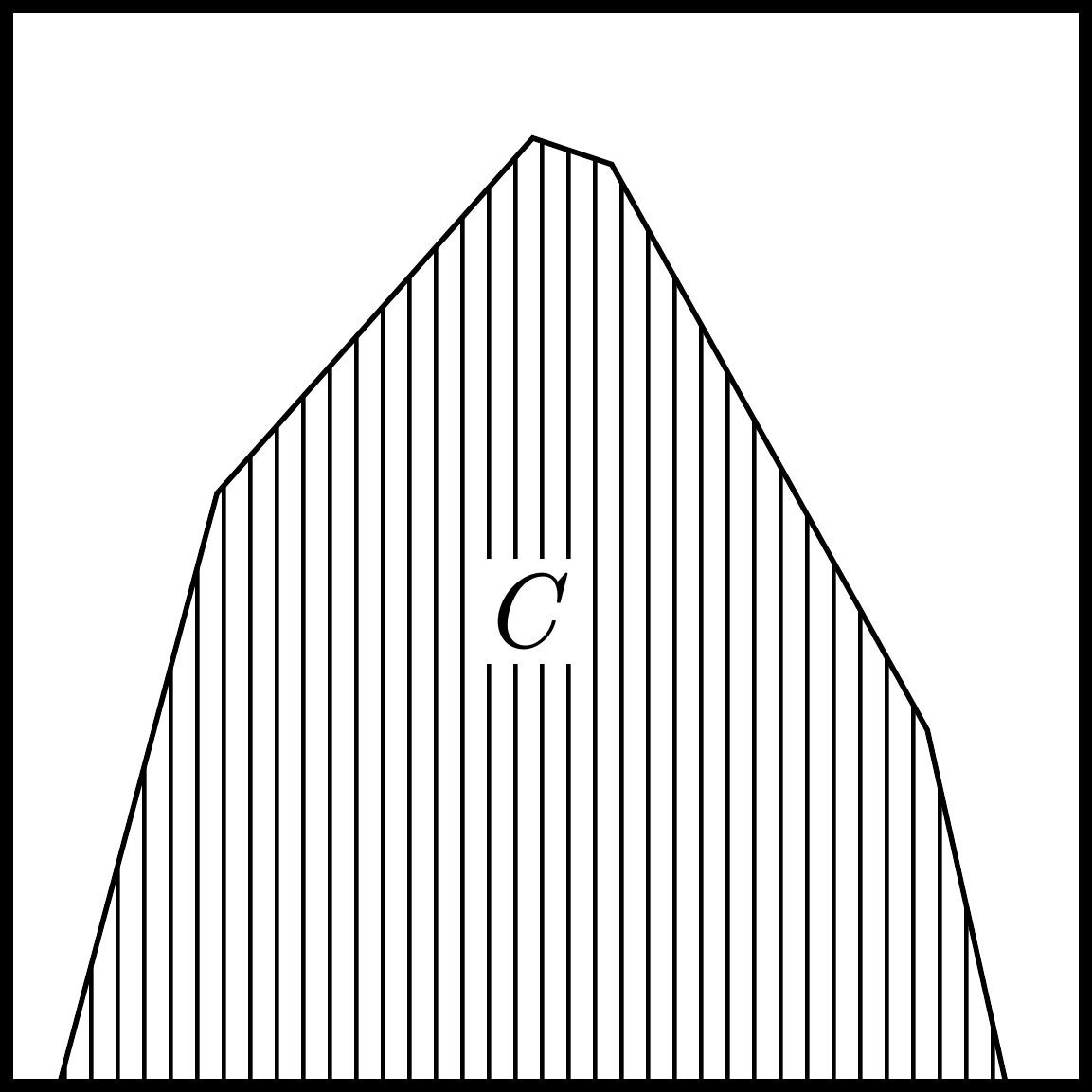}
\par\end{centering}
}
\subfloat[\label{fig:horizontal-segs}]{\begin{centering}
\includegraphics[width=0.15\columnwidth]{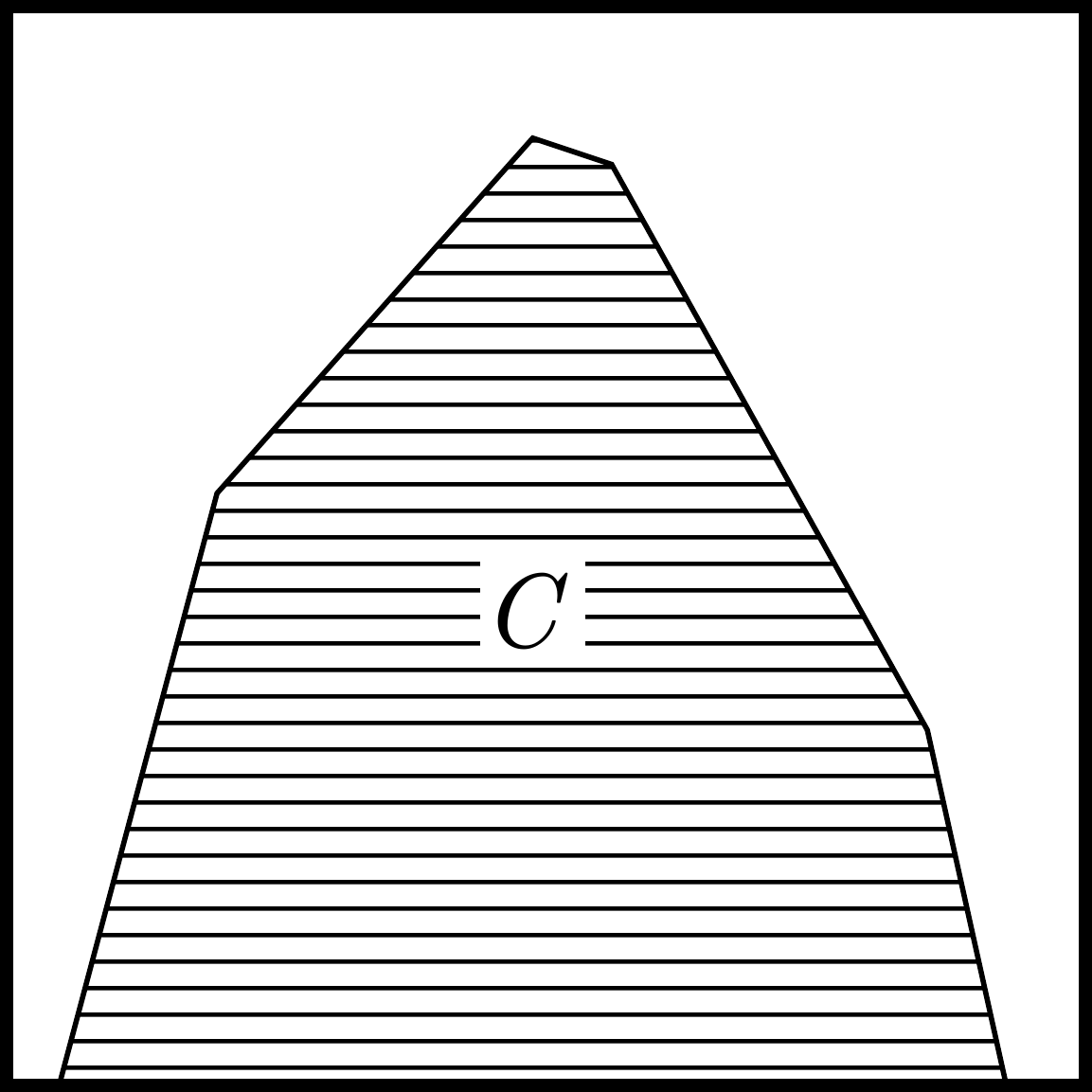}
\par\end{centering}
}\subfloat[\label{fig:shifted-to-right}]{\begin{centering}
\includegraphics[width=0.15\columnwidth]{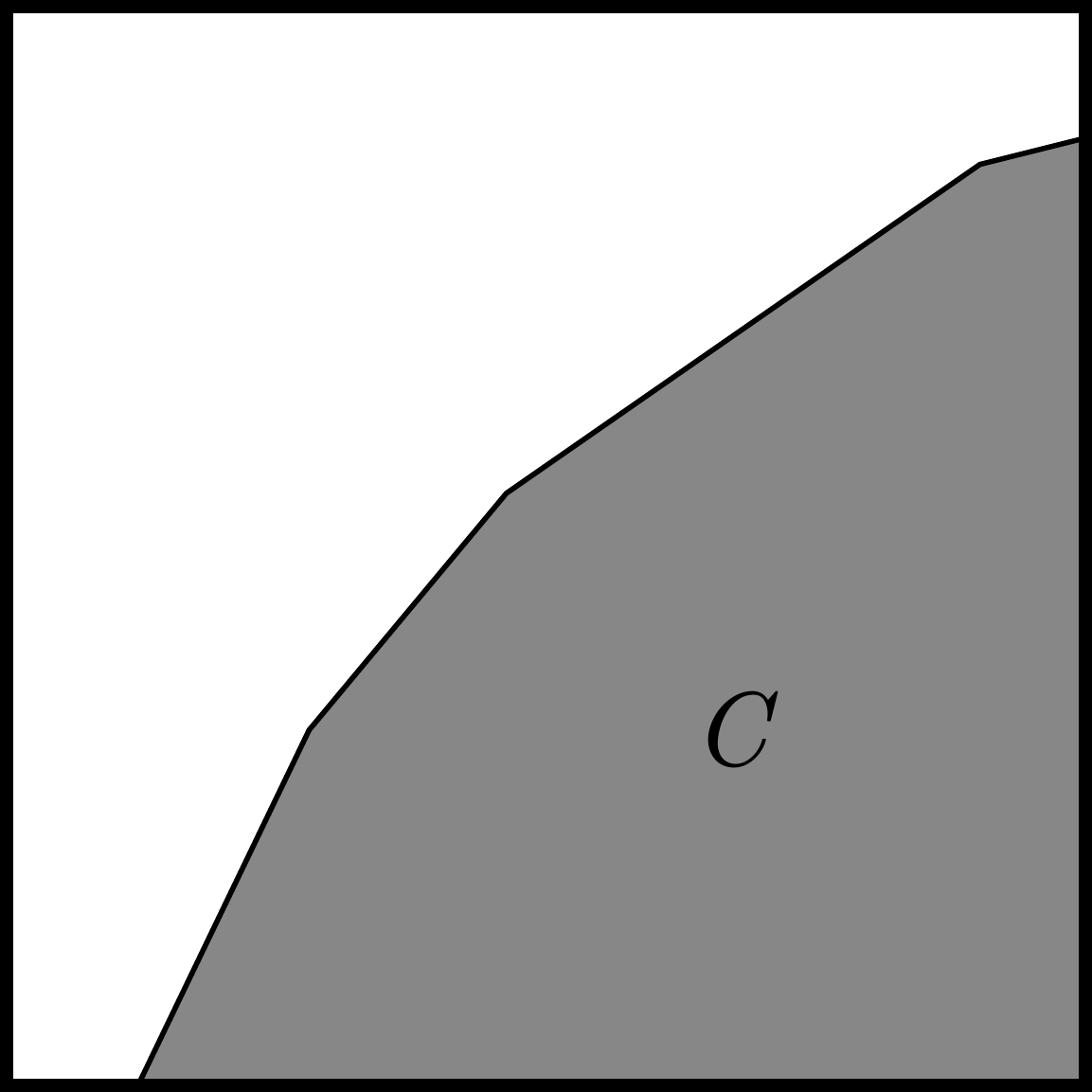}
\par\end{centering}
}
\par\end{centering}
\caption{Two applications of Lemma \ref{lem:line-segment-convex} permit us
to assume that $C$ takes the shape shown in Figure \ref{fig:shifted-to-right}.}
\end{figure}

By visualizing $C$ as a union of vertical line segments (Figure \ref{fig:vertical-segs}),we can apply Lemma \ref{lem:line-segment-convex} to assume without loss
of generality that $C$ takes the shape shown in Figure \ref{fig:shifted-to-bottom}, i.e., that $C$ is the area lying below some concave function (this is because the function defined in Lemma \ref{lem:line-segment-convex} is convex and therefore is maximized when each line segment is as far from the origin, $x_0$, as possible). By visualizing $C$ as a union of \emph{horizontal} line segments (Figure \ref{fig:horizontal-segs}) we can subsequently assume without loss of generality that $C$ takes the shape shown in Figure \ref{fig:shifted-to-right}, i.e., that $C$ is the area lying below some concave and increasing function. By using a series of shear transformations and iteratively applying Lemma \ref{lem:shearing-convex}, we can show that the region $C$ takes the form shown in Figure \ref{fig:eq-FW-C} and in doing so, we would have only increased the Fermat-Weber value of $C$ relative to $x_0$, the center of $B$. Due to space constraints, these transformations are described in detail in section \ref{sec:sheartransformations} of the Online Supplement and we conclude the proof here.

\end{proof}

\begin{lem}
\label{lem:UB}
\label{UpperBound} Let $C$ be a convex region contained in a box $B$ of dimensions $w\times h$ with $w\geq h$ such that $C$ contains a horizontal line segment whose length is equal to that of $\width(C)$. If $A=\Area(C)$, then we have:
\begin{equation}\label{eq:upperb}
\FW(C)\leq  {\FW}_{1/2}\left(w,\frac{2A}{w}\right)
\end{equation}
\end{lem}
\begin{proof}
We break up the proof into two separate cases:

\paragraph{Case~1:~$A<wh/2$}. 
Since $A<\frac{wh}{2}$ and since $C$ contains a line segment whose length is width($C$), we can always construct a smaller box $B'$ with dimensions $w'\times h'$ containing the region $C$, such that $A=\frac{w'h'}{2}$ (see Figure \ref{fig:UB1}). Applying Lemma \ref{lem:UBHalfRectangle} for $B'$, we find that $\FW\left(C\right) \leq \FW_{1/2}\left(w',h'\right)$. By performing some basic algebra, it is not hard to show that, among all values of $w'$ and $h'$
such that $w'\leq w$, $h'\leq h$ and $w'h' = 2A$, the value $FW_{1/2}\left(w',h'\right)$ is maximized when $w' = w$ and $h = 2A/w$ (Fermat-Weber value of a rectangle increases with aspect ratio). Thus, we find that:
\[
\FW(C) \leq {\FW}_{1/2}\left(w,\frac{2A}{w}\right)\,
\]

\begin{figure}[t]
  \centering
  \subfloat[$A < wh/2$]{\label{fig:UB1} \includegraphics[width=0.2\textwidth]{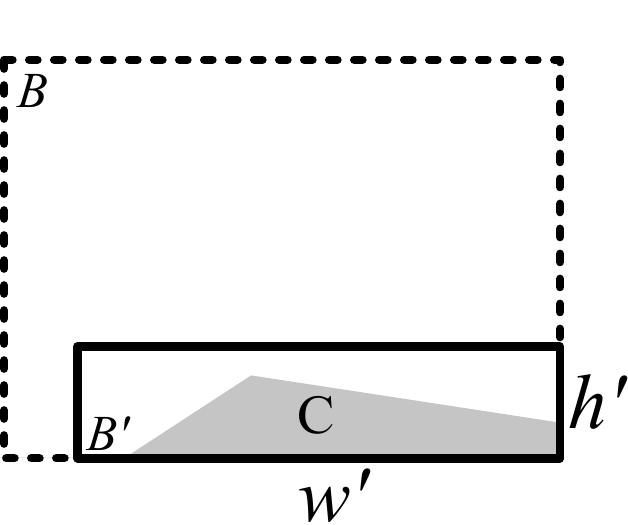}}\,\,\,\,\,\,\,\,\,\,\,\,\,\,\,\,\,\,\,\,\,\,\,\,\,\,
  \subfloat[$A \geq wh/2$]{\label{fig:UB2} \includegraphics[width=0.2\textwidth]{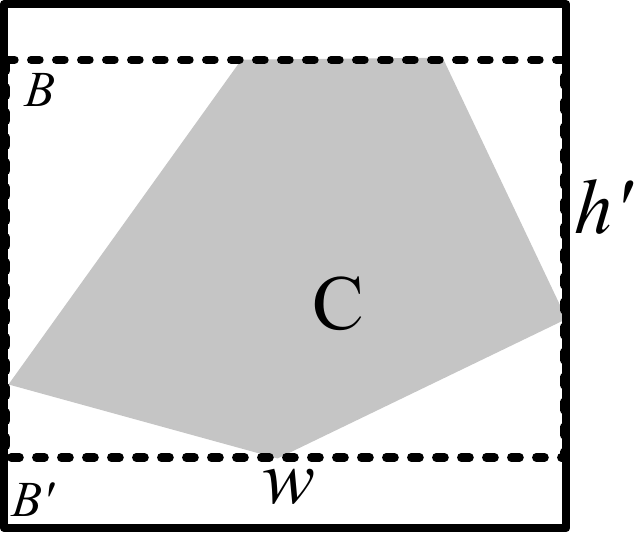}}
  \caption{We can find a box $B'$ that encloses region $C$ such that $w' h' = 2A$, $\forall A \in [0, wh]$.}
  \label{fig:UBFigure1}
\end{figure}
\paragraph{Case~2:~$A\geq wh/2$}. Since $A \geq wh/2$, we can obviously find a larger rectangle $B'$, which contains the region $C$ such that $A=\frac{w'h'}{2}$. One easy way to do that is to
increase the height of $B$ to $h'$ such that $A=\frac{wh'}{2}$ (see Figure \ref{fig:UB2}), which leads to $h'=\frac{2A}{w}\geq h$. Using Lemma \ref{lem:UBHalfRectangle} for $B'$, we arrive at the upper bound which completes the proof.
\end{proof} It is easy to verify that the upper bound of Lemma \ref{UpperBound} is convex and monotonically increasing in $A$ (see Figure \ref{fig:ConvexUB}). It also has an undesirable property that, for some interval: $A\in(A_{c},hw]$, we actually have $\FW_{1/2}\left(w,\frac{2A}{w}\right) > \FW_{\rectang}(w,h)$, which is the Fermat-Weber value of the rectangle itself. Therefore, 
\[
\Phi^2_{UB}(A,w,h) = \min\left\{{ \FW}_{1/2}\left(w,\frac{2A}{w}\right),{\FW}_{\rectang}(w,h)\right\}
\]
will be a more meaningful upper bound. But this function, $\Phi^2_{UB}$, (Figure \ref{fig:ConcaveUB}) is neither convex or concave in the desired region of interest. To help simplify the
analysis in the later section, we choose the simplest \emph{concave envelope} of the above function as our final upper bound, $\Phi_{UB}\,$:
\begin{equation}\label{eq:phiub}
\Phi_{UB}\left(A,w,h\right)=\min\left\{ \frac{A}{A_{c}}{\FW}_{\rectang}(w,h),{\FW}_{\rectang}(w,h)\right\}
\end{equation}
where $A_{c}$ is the solution of the equation: $\FW_{1/2}\left(w,\frac{2A_{c}}{w}\right)={\FW}_{\rectang}\left(w,h\right)$. Fermat-Weber value of a rectangle can be calculated using Equation (\ref{eq:FWRectangleL2}) in Lemma \ref{lem:FWRectangle} in Section \ref{sec:FWCommonObjects} of the Online Supplement.
Note that the ratio $\frac{A_{c}}{hw}$ depends only on the aspect ratio of the rectangle. Also, by construction, for a given $w$ and $h$, $\Phi_{UB}$ is a piecewise linear concave function.

\begin{figure}[htb]
  \centering
  \subfloat[\small{$FW_{1/2}\left(w,\frac{2A}{w}\right)$}]{\label{fig:ConvexUB} \includegraphics[trim = 0mm 60mm 0mm 75mm, clip, width=4.5cm]{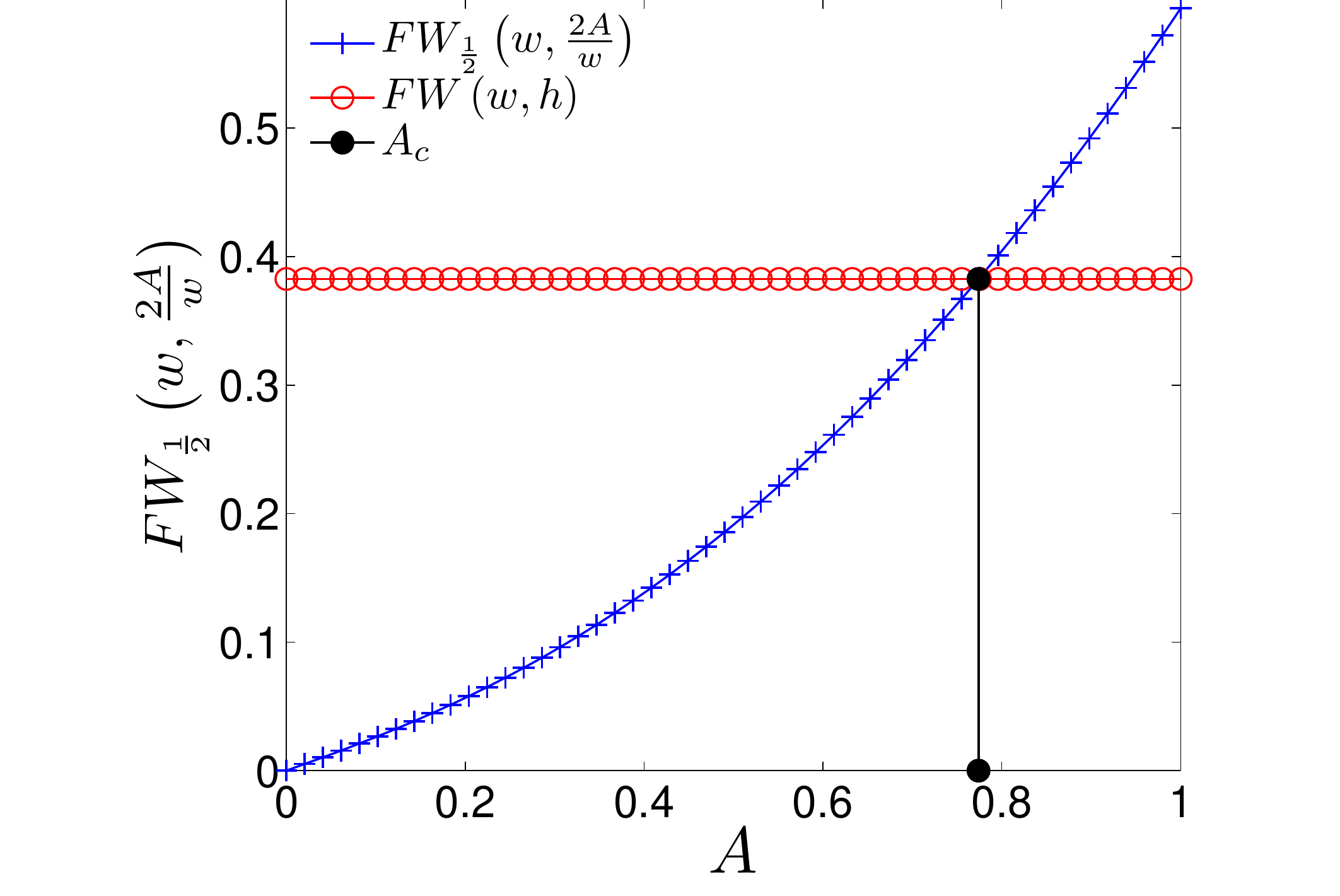}}\,\,\,\,\,
  \subfloat[Concave envelope of $\Phi^2_{UB}$]{\label{fig:ConcaveUB} \includegraphics[trim = 0mm 60mm 0mm 75mm, clip, width=4.5cm]{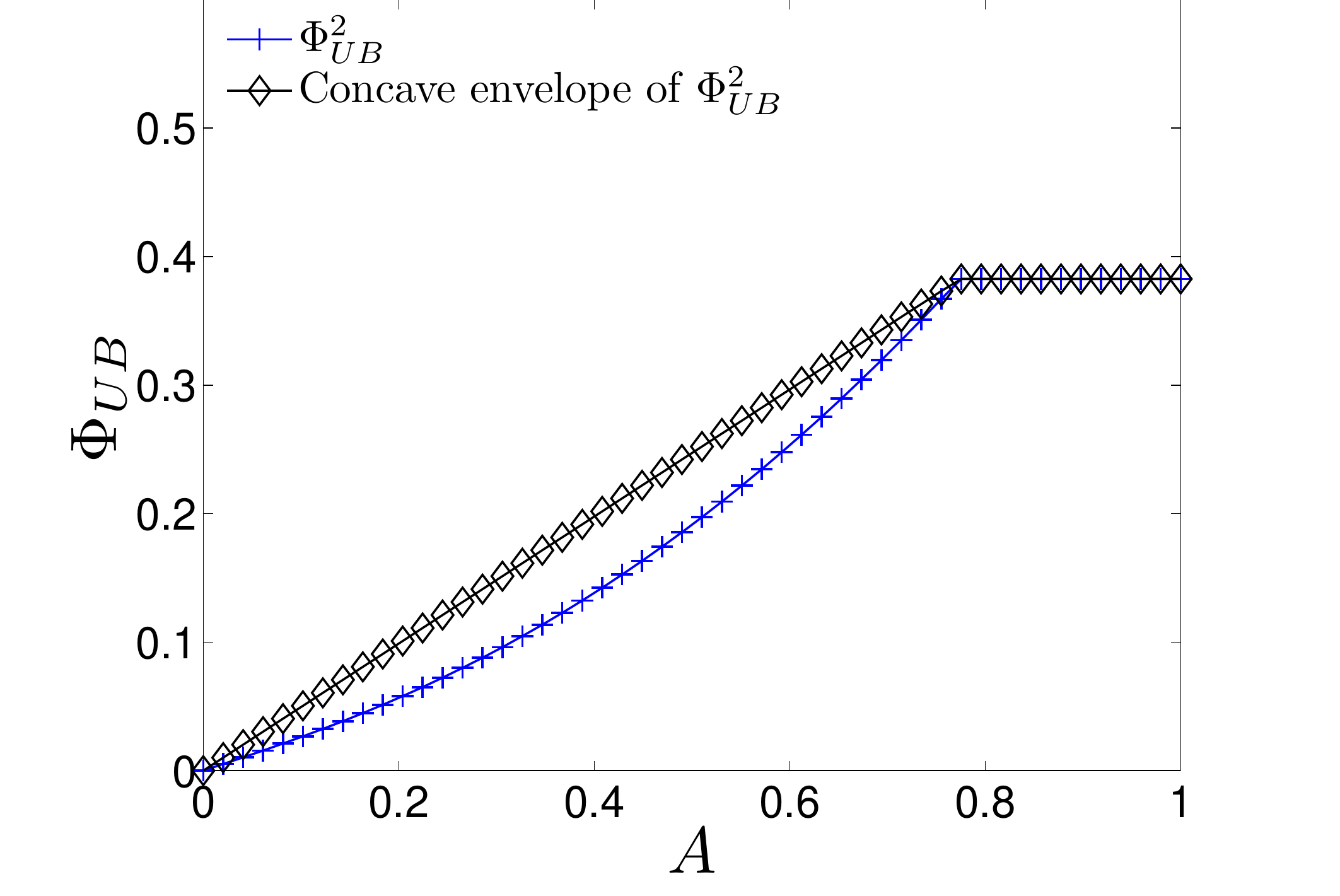}}
  \caption{Upper bound for a convex region of area $A$ within a unit square.}
  \label{fig:ConcaveEnvelope}
\end{figure}

\begin{rem} 
\label{rem:1median}
Few interesting observations on the Fermat-Weber center can be made at this point. Using Lemma 3.1 of \cite{aronov2009minimum} and the lower bound on $\FW(C) \geq 0.16dA$, it can be shown that choosing any point within a convex region is a constant factor-4.17 approximation for the 1-median problem. The center of the smallest enclosing circle of $C$ is a factor 2.41 approximation of Fermat-Weber center, which was proved in \cite{AbuAffash}. Whereas, using bounds $\Phi_{UB}$ and $\Phi_{LB_2}$ we can show that the center of the minimum bounding box of a convex region is a factor-2 approximation algorithm for the 1-median problem. 
\end{rem}

\section{A Constructive Algorithm}
\label{sec:ConstructiveAlg}

As discussed in \cref{sec:GeneralFramework}, the performance of our approach is sensitive to how we divide the bounding box of the convex polygon into $k$ rectangular sub-regions of equal area. More specifically, the aspect ratios of the sub-rectangles will determine the quality of the upper bound.  Here, we modify a well-known constructive algorithm in data visualization field called Squarified treemap algorithm \cite{Squarified} and use it as a partitioning subroutine.
Our algorithm takes a convex polygon $C$ and an integer $k$ and gives a partition of $C$ into $k$ sub-regions and then places a median point at the center of each sub-region. 

\subsection{Description of the Constructive Algorithm}
\label{subsec:SquarifiedAlg}
The algorithm first bounds the polygon with its minimum bounding box $\square C$. Let $R=\square C$ and let $w$ and $h$ denote its dimensions and assume w.l.o.g. that $w\geq h$. The algorithm then rotates $C$ in a way that the bounding box is aligned with the $x$-axis. Note that $w=\diam(C)$. The algorithm then subdivides the bounding box into $k$ sub-rectangles with equal areas. In each iteration, the algorithm makes a vertical (horizontal) division of the remaining (unfilled) part of $\square C$ and fills the left (bottom) part of the division with a subset of the remaining (unplaced) sub-regions in a way that the aspect ratio of the sub-regions that will be placed is minimized. In each iteration, if the width of the remaining part of $\square C$ is greater than its height the division will be vertical and otherwise it will be horizontal.  
Since we assumed $w \geq h$ the first division will be vertical. It creates a strip with a single rectangle at the left side of $R$ which has area equal to $\Area(\square C)/k$ and calculates the aspect ratio of this rectangle. Then it adds another rectangle with same area to this strip on top of the first rectangle. Clearly, the width of the strip will increase. Again it calculates the aspect ratio. Note that the aspect ratio of all rectangles in one strip is the same. Then, it continues adding the rectangles as long as the aspect ratio of our rectangles in the strip keeps decreasing. Once by adding another rectangle to the current strip the aspect ratio increases, the algorithm fixes the current strip and a new strip is created. Let $S$ denote the union of the rectangles in this first strip. After fixing this strip, it will let $R=R\backslash S$ and will continue to the next iteration. In any iteration if $\height(R) > \width(R)$ the division will be horizontal and the rectangles will be added next to each other in a horizontal strip. The algorithm continues until $\square C$ is partitioned into $k$ equal area rectangles. Figure \ref{fig:sqAlg} shows an illustration of applying this algorithm for partitioning a bounding box into 7 equal area rectangles. Algorithm \ref{alg:SquarifiedPartition} shows the detailed steps of this partitioning scheme. 

\begin{figure}[!b]
\centering
  \includegraphics[scale=0.44]{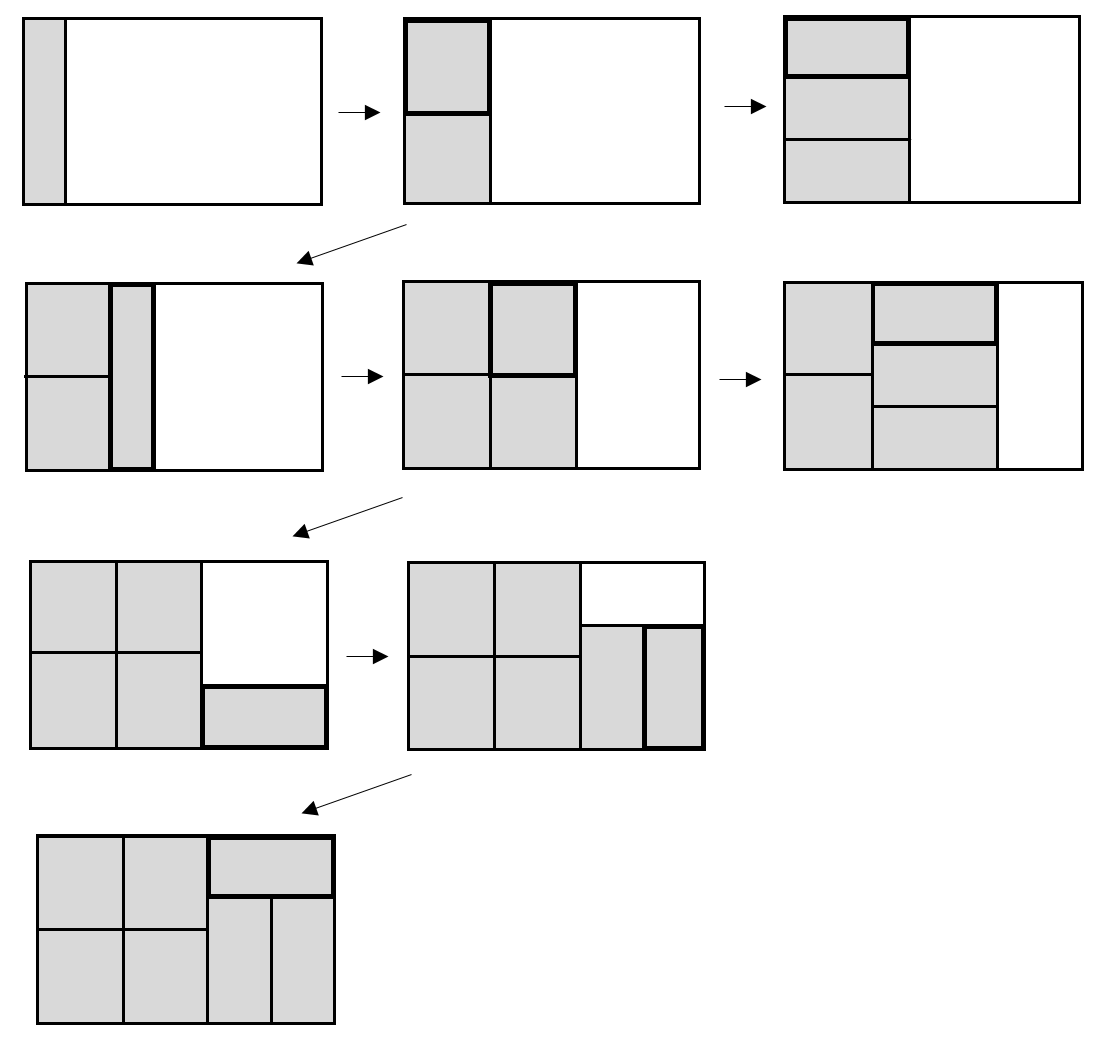}
    \protect\caption{\label{fig:sqAlg} \footnotesize Dividing a box with aspect ratio 1.6 into 7 sub-rectangles using Squarified Algorithm.}
  \end{figure}
  
After completing the partitioning phase, Algorithm \ref{alg:ConstructFW} finds a location for median points. For the rectangles that are completely inside $C$, it places a median point at the center of each such rectangle and for those which are not located completely inside $C$, but have an intersection with $C$, it places a median point at the center of the bounding box of each intersection area. For any rectangle that has no intersection with $C$, it randomly places one center inside $C$. This random placement has minimal impact on the performance of the algorithm which will be discusses in \cref{sec:modification}. 
\begin{algorithm}[t]
\SetAlgoLined
\BlankLine

\KwIn{An axis-aligned rectangle $R$ with dimensions $w$ and $h$ and an integer $k$.}
\KwOut{A partition of $R$ into $k$ rectangles, each having area $\Area(R) /k$.}
\BlankLine
\tcc{**************************************************************************************}
\eIf{k=1}{
   \KwRet{$R$}\;
   }{
   Let $w$ and $h$ denote the width and height of $R$, respectively\;
   \eIf{$w \geq h$}{
   Fill the left part of rectangle with $i\leq k$ sub-rectangles on top of each other each having area $\Area(R) /k$  in a way that the aspect ratio of these sub-rectangles is smaller than the aspect ratio of sub-rectangles when the left part is filled by $i-1$ sub-rectangles\;
   Let $R_j,\; j=1,...,i$ denote the placed sub-rectangles\;
   $S=\cup_{j=1}^{i} R_j$\;
   $R'=R\backslash S$\;
   $w'=w-\frac{i}{kh}\Area(R)$\;
   $h'=h$\;
   }{Fill the bottom part of rectangle with $i\leq k$ sub-rectangles next to each other in a way that the aspect ratio of these sub-rectangles is smaller than the aspect ratio of sub-rectangles when the bottom part is filled by $i-1$ sub-rectangles\;
   Let $R_j,\; j=1,...,i$ denote the placed sub-rectangles\;
   $S=\cup_{j=1}^{i} R_j$\;
   $R'=R\backslash S$\;
   $w'=w$\;
   $h'=h-\frac{i}{kw}\Area(R)$\;
     }
   $k'=k-i$\;
   \KwRet{${\sf SquarifiedPartition}\:(R',w',h',k')$}\;
  } 
\protect\caption{\label{alg:SquarifiedPartition}Algorithm ${\sf SquarifiedPartition}\:(R,w,h,k)$ takes as input an axis-aligned rectangle $R$ and its dimensions $w$ and $h$ and a positive integer $k$ and returns a partition of $R$ into $k$ sub-rectangles with equal areas.}
\end{algorithm}

\begin{algorithm}[t]
\SetAlgoLined
\BlankLine

\KwIn{A convex polygon $C$ and an integer $k$.}
\KwOut{The locations of $k$ median points $p_i,\; i=1,...,k$ in $C$ that approximately minimize $\FW(C, k)$ within a factor of 2.002.}
\BlankLine
\tcc{**************************************************************************************}
Let $R=\square C$ denote a minimal bounding box of $C$. Rotate $C$ so that $R$ is aligned with the coordinate axes\;
Let $w=\width(R)$\;
Let $h=\height(R)$\;
Let $R_1, ... , R_k = {\sf SquarifiedPartition}\:(R,w,h,k)$\;
\For {$i \in \{1,...,k\}$}{
Let $c_i$ denote the center of $R_i$\;
\eIf{$c_i \in C$}{
Set $p_i=c_i$\;
}{
\eIf{$R_i \cap C$ is nonempty}{
Let $R_i^{'}$ be the minimum axis-aligned bounding box of $R_i \cap C$ and let $c_i^{'}$ denote its center\; 
Set $p_i = c_i^{'}$\;
}{
Place $p_i$ anywhere in $C$\;
}
}
}
\KwRet{$p_1,...,p_k$}\;
\protect\caption{\label{alg:ConstructFW}Algorithm ${\sf ConstructFW}\:(C, k)$ takes as input a convex polygon $C$ and an integer $k$ and places $k$ median points in $C$.}
\end{algorithm}

\subsection{Analysis of Algorithm \ref{alg:ConstructFW}}
\label{subsec:AnalysisSqAlg}
\subsubsection{Running Time of Algorithm \ref{alg:ConstructFW}}
\label{subsubsec:RunTimeSqAlg}
This algorithm can be performed with running time $\mathcal{O}(n + k + k \log n)$. This is because Algorithm \ref{alg:SquarifiedPartition} takes $\mathcal{O}(k)$ operations to form the placement of the sub-rectangles and Algorithm \ref{alg:ConstructFW} requires $\mathcal{O}(n)$ operations to find a minimum bounding box of $C$. The last step of Algorithm \ref{alg:ConstructFW} consists of moving the center points to $C$ when necessary, which takes $\mathcal{O}(k \log n)$ operations using a point-in-polygon algorithm \cite{PreparataShamos1985}.

\subsubsection{Aspect Ratios in Algorithm \ref{alg:SquarifiedPartition}}
\label{subsubsec:AspectRatiosSqAlg}
To investigate the aspect ratios of the rectangles produce by Algorithm \ref{alg:SquarifiedPartition}, first, we make two simplifying assumptions:
\begin{description}
\item [Assumption 1.] We assume the bounding box $R=\square C$, once rotated and aligned with the axes, is such that its width is larger than its height, i.e., $w\geq h$. 
\item [Assumption 2.] Without loss of generality we assume that the height of the bounding box is equal to 1. Hence, $w \geq 1$.
\end{description}
\begin{lem}
\label{lem:SquarifiedSliceCase}
If $w/k \geq 0.5$, Algorithm \ref{alg:SquarifiedPartition} divides $R=\square C$ into $k$ equal area rectangles with dimensions $\frac{w}{k}\times 1$.
\end{lem}
\begin{proof}
At the beginning, since $w \geq h$, the algorithm starts with vertical subdivision. If we put one single rectangle on the left, the aspect ratio will be $\AR_1=\max\{\frac{w}{k},\frac{k}{w}\}$. Then, if we add another rectangle to this single rectangle strip, the aspect ratio of each of the two rectangles becomes $\AR_2=\max\{\frac{4w}{k},\frac{k}{4w}\}$. We can easily see that for $\frac{w}{k} \geq 1$ we have $\AR_1=\frac{w}{k}$ and $\AR_2=\frac{4w}{k}$ and hence $\AR_1$ is smaller and the strip with one rectangle will be fixed. For $0.5 \leq \frac{w}{k} \leq 1$, $\AR_1=\frac{k}{w} \leq 2$ and $\AR_2=\frac{4w}{k} \geq 2$ and again $\AR_1$ has smaller
value and the strip with one rectangle will be fixed.

While the width of the undivided part of the main rectangle is greater than its height, the algorithm will continue likewise for new strips. Therefore, it suffices to consider the case when the height of the undivided rectangle becomes greater than its width. Assume that $i$ rectangles have been already located in the main rectangle next to each other. The width and the height of the undivided rectangle will be $w-i\frac{w}{k}$ and 1, respectively. Thus we have
\[
w-i\frac{w}{k} < 1 \Longrightarrow i > k-\frac{k}{w} > k-2 \geq k-1
\]
So, the case when the height becomes greater than the width of the undivided rectangle only happens when one rectangle is remained and it means the subdivision is complete. 
\end{proof}

\begin{lem}
\label{lem:Squarified1stStripAR}
If we have a rectangle with width $w$ and height $h$ and we divide it into $k$ rectangles using Algorithm \ref{alg:SquarifiedPartition} and $w/h \leq k$, the aspect ratio of all rectangles in the first created strip is less than or equal to 2. 
\end{lem}
\begin{proof}
First, note that by construction and assumption of $w\geq h$ the first strip is a vertical strip where we will have rectangles that their height is greater than or equal to their width and they all have the same aspect ratios. Let $\strip_i$ denote a strip that contains $i$ rectangles. Now, assume that Algorithm \ref{alg:SquarifiedPartition} has fixed the first strip and the aspect ratio of rectangles in this strip is greater than 2. So,
\[
{\AR}_{\,\strip_{i}}=\max\{\frac{2i/kh}{h/i},\frac{h/i}{2i/kh}\}=\frac{h/i}{2i/kh}>2
\]
Now, we add another rectangle to this strip and calculate the new aspect ratio.
\[
{\AR}_{\,\strip_{i+1}}=\max\{\frac{h/(i+1)}{2(i+1)/kh} , \frac{2(i+1)/kh}{h/(i+1)} \}
\]
If $\max\{\frac{h/(i+1)}{2(i+1)/kh} , \frac{2(i+1)/kh}{h/(i+1)}\}=\frac{h/(i+1)}{2(i+1)/kh}$, clearly the aspect ratio is decreased by adding the $(i+1)$th rectangle. This contradicts the assumption that the algorithm has fixed the strip with $i$ rectangles. So when the strip is fixed, we must have $\max\{\frac{h/(i+1)}{2(i+1)/kh} , \frac{2(i+1)/kh}{h/(i+1)}\}=\frac{2(i+1)^2}{kh^2}$. Then for $i=1,...,k-1$, we will have
\[
\frac{h/i}{2i/kh}>2 \implies \frac{kh^2}{2i^2}>2  \implies \frac{2i^2}{kh^2}<0.5 \implies \frac{2(i+1)^2}{kh^2}<2\,,
\]
which again shows adding the $(i+1)$th rectangle could improve the aspect ratio and make it less than 2 that contradicts our initial assumption. This completes the proof.
\end{proof}
\begin{lem}
\label{lem:SquarifiedUndividedAR}
If we use Algorithm \ref{alg:SquarifiedPartition} for dividing a rectangle $R$ with dimensions $w$ and height $h$ and $ w/h \leq 2$, in every step of the algorithm, as long as we have more than one sub-rectangle remaining, the undivided part of the rectangle has aspect ratio less than or equal to 2. 
\end{lem}
\begin{proof}
Let $w'$ and $h'$ to be, respectively, the length of the longer and shorter sides of the undivided part of $R$ that is remained after the first strip is fixed by the algorithm. Using this definition, $w'/h'$ represents the aspect ratio of the undivided part in every step of the algorithm. If the direction does not change, we still have $w'/h' = w'/h < w/h \leq 2$. The algorithm will continue, while maintaining this condition, until after fixing a strip the direction changes. Therefore, it suffices to consider the case where the direction changes after fixing the first strip (see Figure \ref{fig:fig1}). Without loss of generality, assume $w\geq h$, 
let the number of sub-rectangles in the fixed strip be $i$ and the remaining sub-rectangles for the undivided part be $r=k-i$. \\
Now, let $z=w'/h' \in\mathbb{R}$ and assume $z > 2$.  Hence, $w > w'=zh'$. Since the area of each sub-rectangle is $\frac{1}{i+r}$, we have
\[
\frac{h(w-h')}{i}=\frac{hh'}{r} \implies \frac{w-h'}{i}=\frac{h'}{r} \implies i=\frac{(w-h')r}{h'} \implies i \geq zr-r \implies i \geq r(\ceil{z}-1)
\]
Let $\rectang_{*}$ denote one of the rectangles in the fixed strip. By increasing $i$, the area of the undivided part decreases and since the height of the rectangle does not change, the aspect ratio of $\rectang_{*}$ increases. Since our algorithm seeks to minimize the aspect ratio in each strip, $i$ should be minimized. Thus, $i=\lceil rz \rceil-r$. We have 
\[
\AR(\rectang_{*})=\frac{w-h'}{h/i}=\frac{i^2 h'/r}{h} = \frac{i^2}{rh/h'}= \frac{i^2}{rw'/h'}= \frac{(\lceil rz\rceil-r)^2}{rz}
\]
For $r \geq 2$ and $z>2$, this term is minimized when $r=2, z=2.5$, which gives $\AR(\rectang_{*}) \geq 1.8$.
Suppose that we reduce one rectangle from the strip and add it to the undivided part of our rectangle. Let $\rectang_{**}$ denote one of the rectangles in the modified strip. We have
\[
\AR(\rectang_{**})=\max\{\frac{w-h'}{h/(i-1)},\frac{h/(i-1)}{w-h'}\}=\max\{\frac{(\lceil rz\rceil-r-1)^2}{rz}, \frac{rz}{(\lceil rz\rceil-r-1)^2}\}
\]
Clearly, if the maximum is the first statement, it is less than $\AR(\rectang_{*})$ and if the maximum is the second statement,  
it is maximized when $r=2, z=2.5$ and results in $\AR(\rectang_{**}) \leq 1.25$.
This is a contradiction since by construction the aspect ratio of a fixed strip cannot be reduced by removing a rectangle. Hence, the assumption that $z>2$ is rejected.
\begin{figure}[t]
\centering
  \includegraphics[scale=0.45]{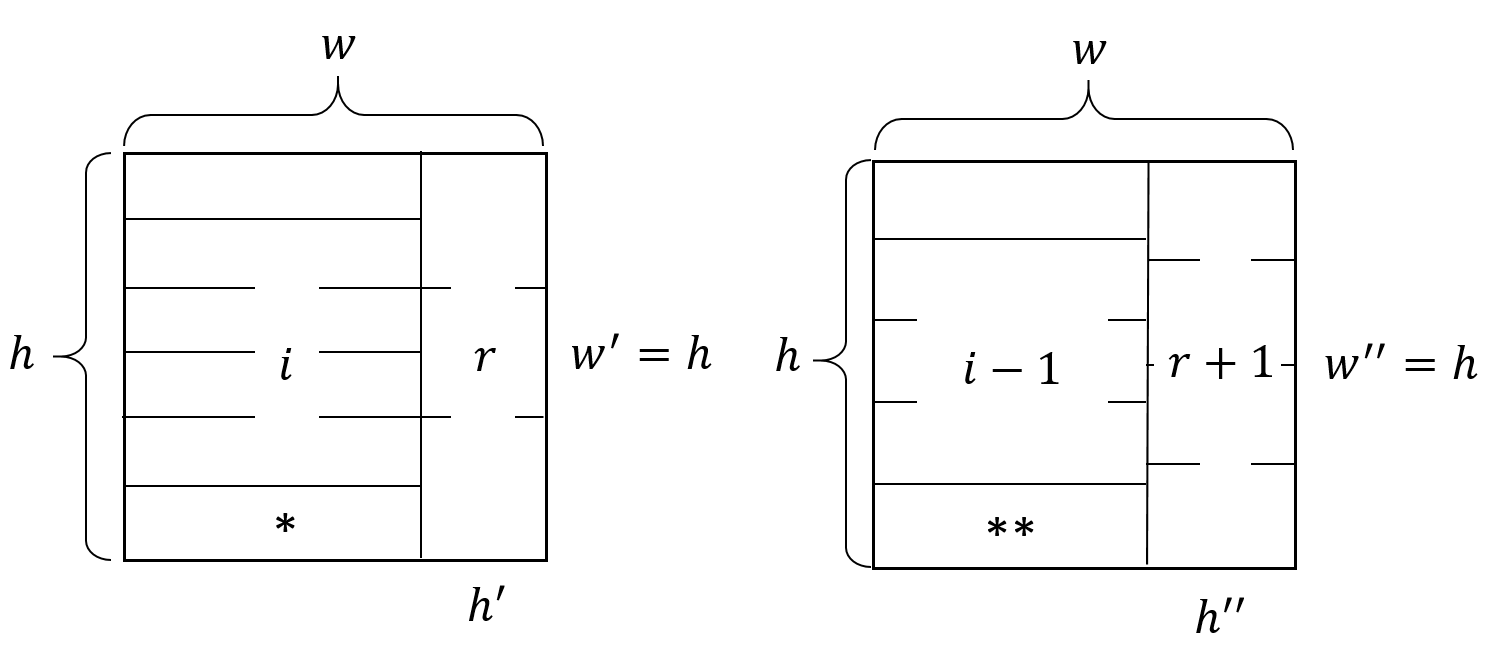}
    \caption{\label{fig:fig1} A rectangle having a fixed strip mentioned in the proof of Lemma \ref{lem:SquarifiedUndividedAR}. The left figure shows a rectangle with width $w$ and height $h$ when a strip with $i$ sub-rectangles is fixed and the direction has switched. The right figure shows this rectangle when one sub-rectangle is reduced from that strip and added to the undivided part on the right.}
\end{figure}
\end{proof}  

\begin{cor}
\label{cor:SquarifiedARbeforeDirectionChange}
If we have a rectangle with width $w$ and height $h$ and we divide it into $k$ rectangles using Algorithm \ref{alg:SquarifiedPartition} and $w/h \leq k$, the aspect ratio of all rectangles in all created strips before a change of direction is at most 2.
\end{cor}
\begin{proof}
This is by iteratively applying the result of Lemma \ref{lem:Squarified1stStripAR} in the undivided part which by Lemma \ref{lem:SquarifiedUndividedAR} satisfies the assumptions of Lemma \ref{lem:Squarified1stStripAR}.
\end{proof}

\begin{lem}
\label{lem:SquarifiedLastAR} 
If $\AR(R) \leq 2$, the aspect ratio of the last sub-rectangle created by Algorithm \ref{alg:SquarifiedPartition} is less than or equal to 3. 
\end{lem}
\begin{proof}
We know from \cref{lem:SquarifiedUndividedAR} that for dividing a rectangle with aspect ratio less than or equal to 2 using \cref{alg:SquarifiedPartition}, while the number of remaining sub-rectangles is more than 1, the undivided part of the rectangle always has aspect ratio less than or equal to 2. The last sub-rectangle will be generated either immediately after a change in direction or width of the last undivided rectangle which has aspect ratio less than or equal to two, should be divided into two equal parts. For the latter case, the aspect ratio is maximized when the width and the height of the undivided part are equal in which case the aspect ratio of the last two sub-rectangles will be equal to 2, which is less than 3.\\
For the former case, let $\rectang_{*}$ be the last sub-rectangle and suppose we have $i$ sub-rectangles in the strip immediately before the last sub-rectangle (see Figure \ref{fig:fig2}). Let $w$ and $h$ be the width and height of the union of this strip and the last rectangle and assume, w.l.o.g., that $w \geq h$ and that the area of this rectangle is equal to $wh=1$. Therefore, the area of each sub-rectangle is $1/(i+1)$. We obtain 
\[
\AR(\rectang_{*})=\frac{h}{h'}=\frac{h^2}{1/(i+1)}=\frac{i+1}{w^2}
\]
We want to see when the aspect ratio of $\rectang_{*}$ is maximized. Since in this case the last sub-rectangle is happening after a change in direction, we must have $i\geq 2$. Note that when we have $i$ sub-rectangles in the strip, $\AR(\rectang_{*})$ is better than the case if we had $i-1$ and $i+1$ sub-rectangles in the strip. This is obvious for $i+1$, since we assumed $w \geq h$. Since we know the aspect ratio of sub-rectangles when the strip has $i$ sub-rectangles is better than the case with $i-1$, we have:
\[
\frac{(w-h')i}{h} < \frac{h}{(w-h'')(i-1)}\,,
\]
where $h''$ is the height of $\rectang_{*}$ if we have $i-1$ sub-rectangles in the strip. 
In the given inequality, the left-hand side represents the aspect ratio of sub-rectangles when there are $i$ sub-rectangles in the strip, while the right-hand side represents the aspect ratio of sub-rectangles when there are $i-1$ sub-rectangles in the strip. It is important to note that in the left-hand side, the sub-rectangles are arranged horizontally since there is only one sub-rectangle left, and this is the only orientation that allows the strip to be positioned before the last sub-rectangle. Conversely, in the right-hand side, the sub-rectangles are arranged vertically. This orientation is chosen because if they were arranged horizontally, the resulting aspect ratio would not be worse than the case with $i$ sub-rectangles.
Then replacing $h'$ with $1/h(i+1)$, $h''$ with $2/h(i+1)$ and $h$ with $1/w$, we obtain:
\[
w^4 < \frac{(i+1)^2}{i^2(i-1)^2}
\]
Since $w\geq h$ and $wh=1$, we have $w\geq 1$, which gives us $i\leq 2$. Hence, $i=2$ and $\AR(\rectang_{*})=\frac{3}{w^2} \leq 3$. 
\begin{figure}[t]
\centering
  \includegraphics[scale=0.45]{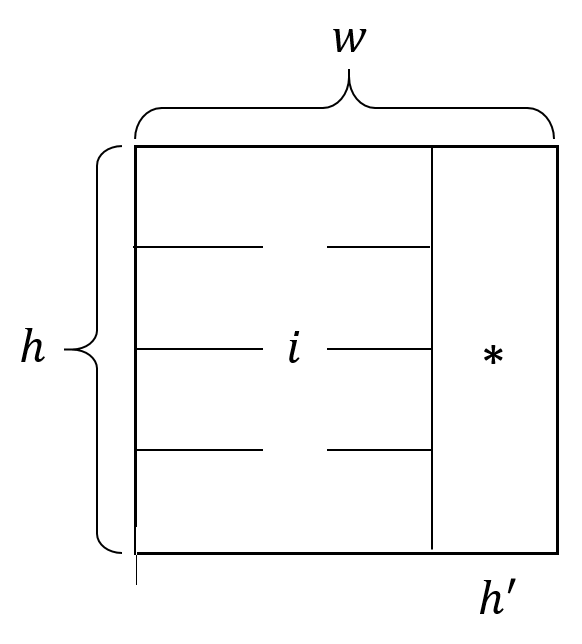}
    \caption{\label{fig:fig2} The strip immediately before the last sub-rectangle mentioned in \cref{lem:SquarifiedLastAR}.}
\end{figure}
\end{proof}

\begin{thm} 
\label{thm:SquarifiedAspectRatioUB}
If we have a rectangle with width $w$ and height 1 (with $w \geq 1$) and we want to divide it into $k$ sub-rectangles such that $w/k \leq 0.5$ using Algorithm \ref{alg:SquarifiedPartition}, the aspect ratio of all rectangles except the last one is less than or equal to 2 and for the last one it is less than or equal to 3.
\end{thm}
\begin{proof}
Algorithm \ref{alg:SquarifiedPartition} starts with creating one strip at the left part of the rectangle (since we have $w \geq h$). The first strip will be repeated until the width of the undivided part of the rectangle becomes smaller than the height of the rectangle, at which point the direction changes. Before this change, somewhere through the process, the aspect ratio of the undivided part will become less than 2, if at the beginning it was not the case. By Corollary \ref{cor:SquarifiedARbeforeDirectionChange}, all of the sub-rectangles created so far have aspect ratios not greater than 2.

We claim that when the direction changes, we must have the ratio of the width of the undivided part to its height less than or equal to the number of remaining sub-rectangles. To prove this, let $w_r$ be the width of the undivided part of the rectangle and let $r=k-i$ be the number of remaining sub-rectangles for this part (see Figure \ref{fig:fig3}). We must show $w_r/h \leq r$ or $w_r/(k-i) \leq 1$. This can be proved by
\[
\frac{w_rh}{k-i}=\frac{wh}{k} \implies \frac{w_r}{k-i}=\frac{w}{k} \leq 0.5 < 1
\]
\begin{figure}[htb]
\centering
\vspace{-5pt}
  \includegraphics[scale=0.4]{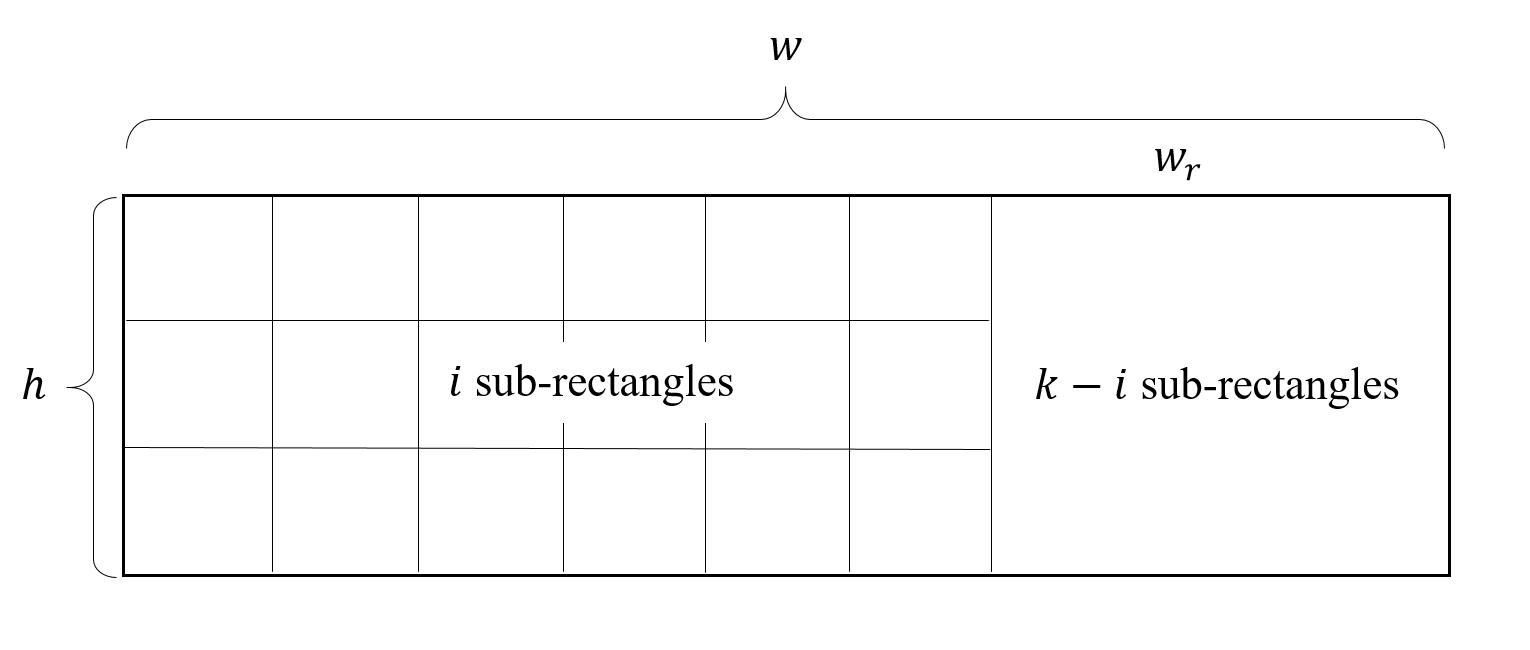}
    \protect\caption{\label{fig:fig3} \footnotesize An intermediate step of Algorithm \ref{alg:SquarifiedPartition} while $i$ sub-rectangles are fixed and $r=k-i$ sub-rectangles should be placed in the remaining part of the box with width $w_r$ and height $h$.}
  \end{figure}
  
Therefore, the condition of Lemma \ref{lem:Squarified1stStripAR} is satisfied for the undivided part after the change in direction. Also, after the change of direction, if $r\geq 2$, we have an undivided rectangle with aspect ratio not greater than 2, satisfying the condition $w/h\leq 2$ in Lemma \ref{lem:SquarifiedUndividedAR}. If the number of remaining sub-rectangles $r$ is more than one, i.e., $w/h \leq 2 \leq r$, we can conclude, by Corollary \ref{cor:SquarifiedARbeforeDirectionChange}, that the aspect ratio of sub-rectangles in the next strip is less than or equal to 2. While we have more than one sub-rectangle, we continue the algorithm by iteratively generating sub-rectangle (in strips) and undivided rectangles with aspect ratios not greater than 2. Therefore, all sub-rectangles that have been generated so far have aspect ratio less than or equal to 2. Once we get into the situation that only one sub-rectangle is remained, by Lemma \ref{lem:SquarifiedLastAR}, the aspect ratio of the last sub-rectangle is less than or equal to 3.
\end{proof}

\section{A Subdivision Algorithm}
\label{sec:SubdivisionAlg}
As mentioned in Sections \ref{sec:GeneralFramework} and \ref{sec:ConstructiveAlg}, the performance of our approach is mainly driven by how we divide the bounding box of the convex polygon into $k$ rectangular sub-regions of equal area and what would be the aspect ratios of those sub-rectangles.  Here, we modify an algorithm presented in \cite{kcenter} for solving the $k$-Centers problem. The difference between $k$-centers and the $k$-medians problems is that the former minimizes the maximum distance between each point and its closest center. This algorithm divides a rectangle into $k$ equal area sub-rectangles and we use it as our partitioning subroutine. Similar to our constructive algorithm, our subdivision algorithm takes a convex polygon $C$ and an integer $k$ and gives a partition of $C$ into $k$ sub-regions and then places a median point at the center of each sub-region.

\subsection{Description of the Subdivision Algorithm}
\label{subsec:KCentersAlg}
Similar to the constructive algorithm, our subdivision algorithm first creates the bounding box aligned with the coordinate axes. Then it breaks the bounding box into $k$ area-equal sub-rectangles with up to 6 different configurations. In each configuration, our algorithm first divides the bounding box into two parts, horizontally or vertically. Based on the number of sub-rectangles in each part, the algorithm finds the dividing point in a way that the area of sub-rectangles in the first part be equal to the area of sub-rectangles in the other part. Beside equal areas, all sub-rectangles in the same part have the same aspect ratios. This partitioning scheme is sketched in Figure \ref{fig:GridPartition} and described in Algorithm \ref{alg:GridPartitionAlg}. Then our algorithm investigates these configurations and the configuration with the smallest maximum aspect ratio of sub-rectangles is selected as the final configuration. If the maximum aspect ratio in two or more configurations are the same, the algorithm compares the minimum aspect ratio of sub-rectangles in different configurations and chooses the one with the lower value as the final configuration. Figure \ref{fig:kCenterAlgSimpleExample} shows a sample of the output of the subdivision algorithm described in Algorithm \ref{alg:SubdivideFW}.\\
 
\begin{figure}[t]
\begin{centering}
\subfloat[$\text{``flag''}={\tt VERTICAL}$]{\begin{centering}
\includegraphics[bb=0bp 0bp 618bp 244bp,width=0.4\columnwidth]{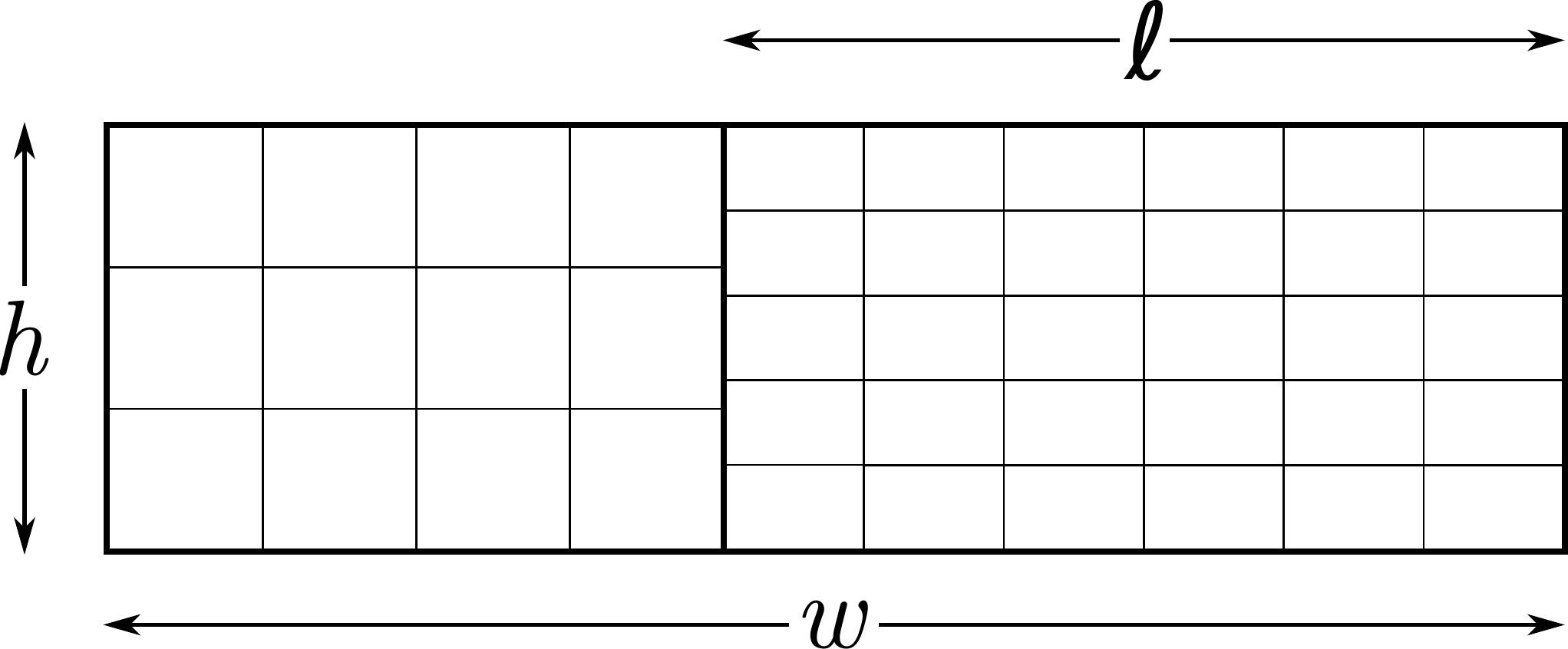}
\par\end{centering}

}\subfloat[$\text{``flag''}={\tt HORIZONTAL}$]{\begin{centering}
\includegraphics[width=0.4\columnwidth]{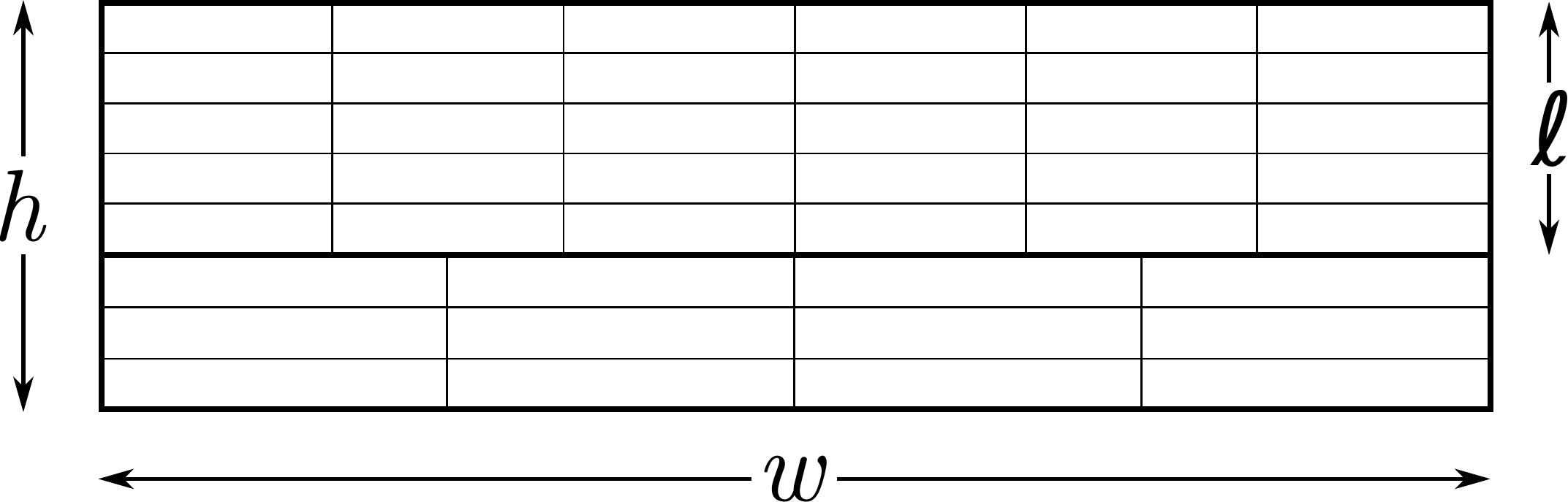}
\par\end{centering}

}
\par\end{centering}

\protect\caption{\label{fig:GridPartition}The output of Algorithm \ref{alg:GridPartitionAlg} where $(p_{1},q_{1},p_{2},q_{2})=(4,3,6,5)$ and $\ell$ is as indicated. We subdivide the rectangle into two grids, one of which has $4\times3$ rectangles and one of which has $6\times5$ rectangles; the width and height of these two grids is determined by $\ell$. The above images are reproduced from Figure 6 of \cite{kcenter}.}
\end{figure}

\begin{algorithm}[H]
\SetAlgoLined
\BlankLine
\KwIn{An axis-aligned rectangle $R$, having dimensions $w \times h$, integers $p_1$, $q_1$, $p_2$, $q_2$, a positive number $\ell$, and a ``flag'' equal to VERTICAL or HORIZONTAL.}
\KwOut{A partition of $R$ into $p_1q_1 + p_2q_2$ equal area rectangles.}
\BlankLine
\tcc{*************************************************************************************************}
\eIf{"flag" = VERTICAL}{
Let $R_1$ be the left half of $R$, having dimensions $(w-\ell) \times h$\;
Let $R_2$ be the right half of $R$, having dimensions  $\ell \times h$\;
Break $R_1$ into a $p_1 \times q_1$ rectangular grid, and call the rectangular cells $\boxdot_1,..., \boxdot_{p_1q_1}$\;
Break $R_2$ into a $p_2 \times q_2$ rectangular grid, and call the rectangular cells $\boxdot_{p_1q_1+1},..., \boxdot_{p_1q_1+p_2q_2}$\;
   }{
   Let $R_1$ be the bottom half of $R$, having dimensions $w \times (h-\ell)$\;
Let $R_2$ be the top half of $R$, having dimensions  $w \times \ell$\;
Break $R_1$ into a $p_1 \times q_1$ rectangular grid, and call the rectangular cells $\boxdot_1,..., \boxdot_{p_1q_1}$\;
Break $R_2$ into a $p_2 \times q_2$ rectangular grid, and call the rectangular cells $\boxdot_{p_1q_1+1},..., \boxdot_{p_1q_1+p_2q_2}$\;
  }
\KwRet{$\boxdot_1,...,\boxdot_{p_1q_1+p_2q_2}$}\;
\protect\caption{\label{alg:GridPartitionAlg}Algorithm ${\sf GridPartition}\:(R,p_1,q_1,p_2,q_2,\ell,\mbox{``flag''})$ takes as input an axis-aligned rectangle $R$ and positive integers $p_1,q_1,p_2$ and $q_2$ and divides $R$ into $p_1q_1+p_2q_2$ sub-rectangles with equal area.}
\end{algorithm}

 \scalebox{.9}{
\begin{algorithm}[H]
\SetAlgoLined
\BlankLine

\scriptsize
\KwIn{A convex polygon $C$ and an integer $k$.}
\KwOut{The locations of $k$ median points $p_i,\; i=1,...,k$ in $C$ that approximately minimize $\FW(C, k)$ within a factor of 2.002.}
\BlankLine
\tcc{*************************************************************************************************}
Rotate $C$ so that $\square C$ is aligned with the coordinate axes\;
Let $\square C$ denote the minimum bounding box of $C$\; 
Let $w=\width(R)$\;
Let $h=\height(R)$\;
Set $p_0 =\lfloor{\sqrt{wk/h}}\rfloor$, $q_0 =\lfloor{\sqrt{hk/w}}\rfloor$, $\mbox{minAR} =\infty$, and $\mbox{minmaxAR} =\infty$\;
\For{$p\in \lbrace p_0-1,p_0,p_0+1\rbrace$}{
Set $q=\lfloor{k/p}\rfloor$\;
\If{$p,q \geq 1 $}{
Set $s=k-pq$, so that $k=pq+s=(p-s)q+s(q+1)$\; 
Let $\ell$ be the solution to $\Big(\frac{w-\ell}{p-s}\Big)\Big(\frac{h}{q}\Big)=\Big(\frac{\ell}{s}\Big)\Big(\frac{h}{q+1}\Big)$ that satisfies $0 \leq \ell \leq w$. If no such $\ell$ exists, let $\ell = 0$\;
Let $\boxdot_1, . . . ,\boxdot_k = {\sf GridPartition}\:(\square C, p-s, q, s, q+1, \ell,$ VERTICAL)\;
Let $(x^{\prime}_{1}, . . . , x^{\prime}_{k})$ be the centers of the boxes  $\boxdot_i$\;
\eIf{$\ell \neq 0$}{
Let $\AR_1 = \max(\frac{wq^2}{hk},\frac{hk}{wq^2})$ and $\AR_2 = \max(\frac{w(q+1)^2}{hk},\frac{hk}{w(q+1)^2})$ and $\AR = \max(AR_1,AR_2)$\;
}{Let $\AR = \max(\frac{wq^2}{hk},\frac{hk}{wq^2})$\;}
\eIf{ $\AR<\mbox{minAR}$}{
Set $\mbox{minAR}=\AR$ and $x_1, . . . , x_k = x^{\prime}_{1}, . . . , x^{\prime}_{k}$\;}
{\If{$\AR=\mbox{minAR}$ \AND $\min(\AR_1,\AR_2)<\mbox{minmaxAR}$}{
Set $\mbox{minAR}=\AR$ and $\mbox{minmaxAR}=\min(\AR_1,\AR_2)$\;
Set $x_1, . . . , x_k = x^{\prime}_{1}, . . . , x^{\prime}_{k}$\;}}
}
}
\For{$q\in \lbrace q_0-1,q_0,q_0+1\rbrace$}{
Set $p=\lfloor{k/q}\rfloor$\;
\If{$p,q \geq 1 $}{
Set $s=k-pq$, so that $k=pq+s=(q-s)p+s(p+1)$\;
Let $\ell$ be the solution to $\Big(\frac{w}{p}\Big)\Big(\frac{h-\ell}{q-s}\Big)=\Big(\frac{w}{p+1}\Big)\Big(\frac{\ell}{s}\Big)$ that satisfies $0 \leq \ell \leq h$. If no such $\ell$ exists, let $\ell = 0$\;
Let $\boxdot_1, . . . ,\boxdot_k = {\sf GridPartition}\:(\square C, p, q-s, p+1, s, \ell,$ HORIZONTAL)\;
Let $(x^{\prime}_{1}, . . . , x^{\prime}_{k})$ be the centers of the boxes  $\boxdot_i$\;
\eIf{$\ell \neq 0$}{
Let $\AR_1 = \max(\frac{hp^2}{wk},\frac{wk}{hp^2})$ and $\AR_2 = \max(\frac{h(p+1)^2}{wk},\frac{wk}{h(p+1)^2})$ and $\AR = \max(AR_1,AR_2)$\;
}{Let $\AR = \max(\frac{hp^2}{wk},\frac{wk}{hp^2})$\;
}
\eIf{ $\AR<\mbox{minAR}$}{
Set $\mbox{minAR}=\AR$ and $x_1, . . . , x_k = x^\prime_{1}, . . . , x^\prime_{k}$\;}
{\If{$\AR=\mbox{minAR}$ \AND $\min(\AR_1,\AR_2)<\mbox{minmaxAR}$}{
Set $\mbox{minAR}=\AR$ and $\mbox{minmaxAR}=\min(\AR_1,\AR_2)$\;
Set $x_1, . . . , x_k = x^\prime_{1}, . . . , x^\prime_{k}$\;}}
}
}
\For {$i \in \{1,...,k\}$}{
\eIf{$x_i \in C$}{
Set $p_i=x_i$\;
}{
\eIf{$R_i \cap C$ is nonempty}{
Let $R_i^{'}$ be the minimum axis-aligned bounding box of $R_i \cap C$ and let $c_i^{'}$ denote its center\; 
Set $p_i = c_i^{'}$\;
}{
Place $p_i$ anywhere in $C$\;
}
}
}
\KwRet{$p_1,...,p_k$}\;
\protect\caption{\label{alg:SubdivideFW}Algorithm ${\sf SubdivideFW}$ takes a polygon $C$ and integer $k$ as input, partitions $C$ into $k$ sub-regions and finds locations for $k$ median points.}
\end{algorithm}
} 
~\\

\begin{figure}[htb]
\captionsetup{farskip=0pt}
\begin{centering}
  \includegraphics[scale=0.42]{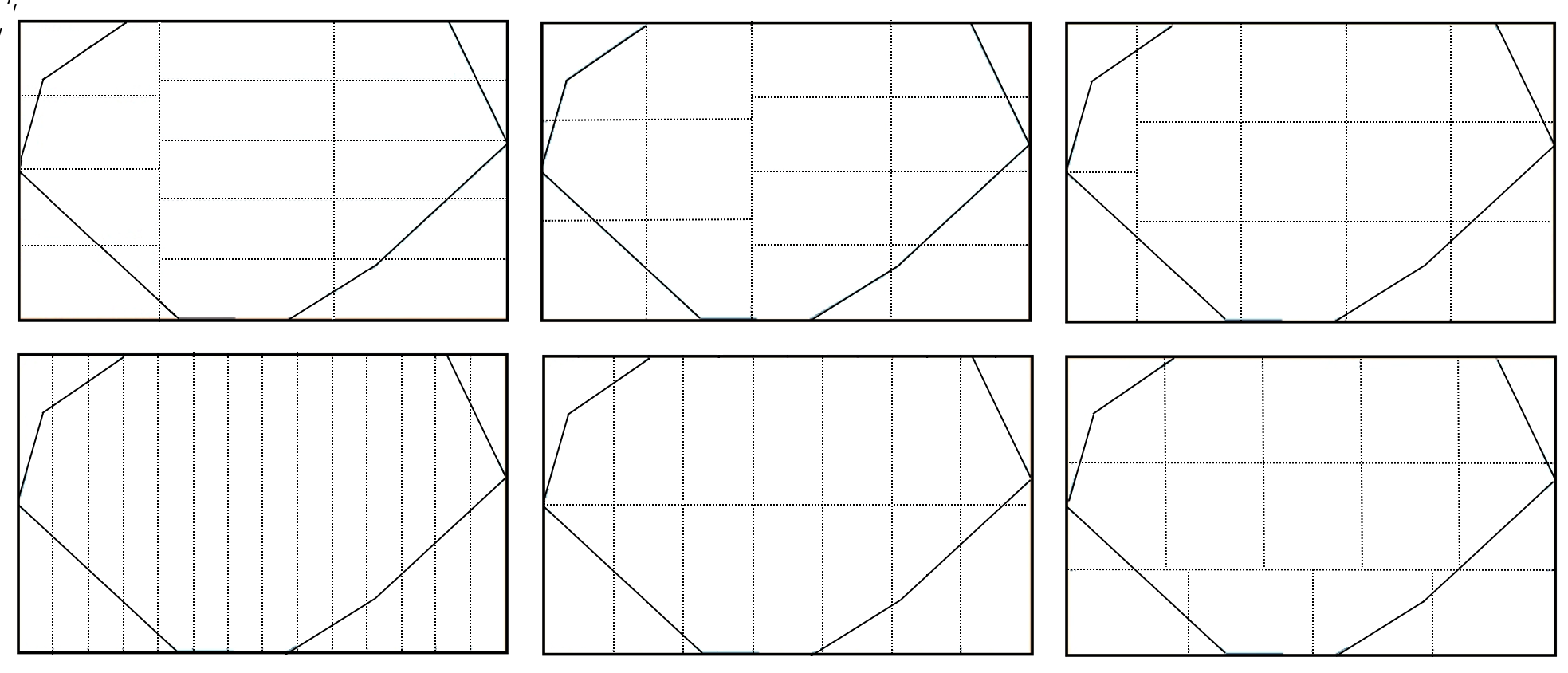}
  
    \protect\caption{\label{fig:kCenterAlgSimpleExample} Six different configurations of dividing a bounding box of a polygon into 14 sub-rectangles using Algorithm \ref{alg:SubdivideFW}. The width of bounding box is 1.4917 and the height is 0.9085. In each configuration, the box is divided into two parts. All 14 sub-rectangles have same areas. The bottom right figure is chosen among these six configurations because of having the best maximum aspect ratio of sub-rectangles. The aspect ratio of sub-rectangles in the top part of this configuration is 1.0876 and for the sub-rectangles in the bottom part is 1.4367.}
\end{centering}  
\end{figure}

\subsection{Analysis of Algorithm \ref{alg:SubdivideFW}}
\label{subsec:AnalysisKCenterAlg}

\subsubsection{Running Time of Algorithm \ref{alg:SubdivideFW}}
\label{subsubsec:RunTimeKCentersAlg}
This algorithm can be performed with running time $\mathcal{O}(n + k + k \log n)$. This is because Algorithm \ref{alg:GridPartitionAlg} takes $\mathcal{O}(k)$ operations to partition the rectangle and Algorithm \ref{alg:SubdivideFW} requires $\mathcal{O}(n)$ operations to find a minimum bounding box of $C$. The last step of Algorithm \ref{alg:SubdivideFW} consists of moving the center points to $C$ when necessary, which takes $\mathcal{O}(k \log n)$ operations using a point-in-polygon algorithm \cite{PreparataShamos1985}.

\subsubsection{Aspect Ratios in Algorithm \ref{alg:SubdivideFW}}
\label{subsubsec:AspectRatiosKCentersAlg}
Similar to \cref{subsubsec:AspectRatiosSqAlg}, we first adopt Assumptions 1 and 2 to simplify the analysis of aspect ratios.
\begin{lem}
\label{lem:SubdivisionAlgSliceCase}
If $w/k \geq 0.5$, one of the configurations of Algorithm \ref{alg:SubdivideFW} divides the bounding box into $k$ rectangles each with dimensions $\frac{w}{k} \times 1$.
\end{lem}
\begin{proof}
We know that $q_0=\floor{\sqrt{hk/w}}=\floor{\sqrt{k/w}}$. It is easy to show that if $0.5 \leq w/k \leq 1$, we will have $q_0=1$ and if $w/k>1$, we have $q_0=0$. 
In all cases, we have a configuration with $q=1$ and hence $p=\floor{k/q}=k$. In this case, $s=\ell=0$ and the bounding box would be partitioned into $k$ rectangles with dimensions $\frac{w}{k} \times 1$.
\end{proof}
\begin{claim} 
\label{claim:claim1}
For $p=p_0+1$, we always have $q=q_0-1$ or $q=q_0$.
\end{claim}
\begin{proof}
When $p=p_0+1$, we have $q = \Bigl\lfloor\frac{k}{1+ p_0}\Bigr\rfloor=\Bigl\lfloor\frac{k}{1+\floor{\sqrt{kw}}}\Bigr\rfloor$.
We know that $\frac{\sqrt{kw}}{1+\floor{\sqrt{kw}}}<1$. We have, 
\[
\frac{\sqrt{kw}}{1+\floor{\sqrt{kw}}}<1 \; \Longrightarrow \; \frac{k}{1+\floor{\sqrt{kw}}}<\frac{\sqrt{k}}{\sqrt{w}} \; \Longrightarrow \; q=\Bigl\lfloor\frac{k}{1+\floor{\sqrt{kw}}}\Bigr\rfloor \leq \Bigl\lfloor\sqrt\frac{k}{w}\Bigr\rfloor = q_0
\]
For the lower bound we start with the fact that $w\geq 1$.
\[
w\geq 1\Rightarrow w\sqrt{k}+\sqrt{w} \geq \sqrt{k} \;\Rightarrow \; 
 w\sqrt{k}+\sqrt{w}+k\sqrt{w} \geq \sqrt{k}+k\sqrt{w} \;\Rightarrow \;
 \sqrt{w}(k+1+\sqrt{kw}) \geq \sqrt{k}(1+\sqrt{wk}) 
 \;\Rightarrow \; \]
 \[
 \dfrac{\sqrt{k}}{\sqrt{w}} \leq \dfrac{k+1+\sqrt{kw}}{1+\sqrt{kw}}
 \;\Rightarrow \; \dfrac{\sqrt{k}}{\sqrt{w}}-1 \leq \dfrac{k}{1+\sqrt{kw}} \;\Rightarrow \;
\Bigl\lfloor{\sqrt{\dfrac{k}{w}}} \Bigr\rfloor -1 \leq \Bigl\lfloor \dfrac{k}{1+\sqrt{kw}}\Bigr\rfloor \leq \Bigl\lfloor \dfrac{k}{1+\lfloor \sqrt{kw} \rfloor}\Bigr\rfloor \;\Rightarrow \; q_0-1 \leq q \]
 
Hence, for $p=p_0+1$, we have $q \in \{q_0-1,q_0\}$.
\end{proof}
\begin{lem}
\label{lem:claimi}
At least one of the configurations made with $p=p_0$ or $p=p_0+1$, has $q=q_0$
\end{lem}
\begin{proof}
By \cref{claim:claim1}, if we set $p=p_0+1$, we would have $q=q_0$ or $q=q_0-1$. If in this case $q=q_0$, we are done. Therefore, it suffices to show that if $q=q_0-1$ when $p=p_0+1$, in the configuration that we set $p=p_0$, we will have $q=q_0$.
\\We prove this in two parts. First, we show that if $p=p_0$, we will have $q=\floor{k/p}=\floor{k/p_0} \geq q_0$ and then we show that $q < q_0+1$.
First,
\[
p_0=\floor{\sqrt{wk}} \; \Rightarrow \; \frac{k}{p_0}=\frac{k}{\floor{\sqrt{wk}}} \geq \frac{k}{\sqrt{wk}}=\sqrt{\dfrac{k}{w}} \geq \floor{\sqrt{\dfrac{k}{w}}} =q_0 \;\Rightarrow\; q=\floor{\frac{k}{p_0}} \geq q_0
\]
Second, suppose that when $p=p_0+1$, we have $q=\floor{\frac{k}{p_0+1}}=q_0-1$, then
\[
\frac{k}{p_0+1}=q_0-\alpha, \quad 0<\alpha \leq 1 \; \Longrightarrow \; k=p_0q_0+q_0-\alpha p_0-\alpha
\; \Longrightarrow \; \frac{k}{p_0}=q_0+\frac{q_0}{p_0}-\alpha-\frac{\alpha}{p_0}
\]
Therefore, for the case of $p=p_0$ we have
\[
q=\lfloor\frac{k}{p_0}\rfloor=q_0+\lfloor\frac{q_0}{p_0}-\alpha-\frac{\alpha}{p_0}\rfloor
\]
It is enough to show that $\frac{q_0}{p_0}-\alpha-\frac{\alpha}{p_0} < 1$. Since $w\geq h$, we always have $p_0 \geq q_0$ and thus, $\frac{q_0}{p_0}\leq 1$. Since $\alpha>0$, we can conclude that $\frac{q_0}{p_0}-\alpha-\frac{\alpha}{p_0} < 1$.
\end{proof}

Using Lemmas \ref{lem:SubdivisionAlgSliceCase} and \ref{lem:claimi} we can analyze the aspect ratios of the output rectangles of Algorithm \ref{alg:SubdivideFW} by considering several cases. To avoid repetition we combine this section with Section \ref{subsubsec:ApxFactorSubdivideAlg}.

\section{Approximation Factors}
\label{sec:ApxFactors}

\subsection{Upper Bound on $\FW(C,k)$ for Algorithms \ref{alg:ConstructFW} and \ref{alg:SubdivideFW}}
\label{sec:UB-Algs}

In order to derive an upper bound for $\FW(C, k)$ under Algorithms \ref{alg:ConstructFW} and \ref{alg:SubdivideFW}, we first make a straightforward observation that validates our use of the upper bounding function $\Phi_{UB}$ from Equation (\ref{eq:phiub}). Recall
that Algorithm \ref{alg:ConstructFW} fits $k$ rectangles of equal area in the bounding box of $C$, i.e., $R=\square C$, Algorithm \ref{alg:SubdivideFW} partitions $\square C$ into $k$ rectangles of equal area, and that in both algorithms $C$ is initially oriented so that its diameter is aligned with the coordinate $x$-axis. It is easy to see, then, that for any rectangle $R_i$ produced by either of these algorithms, it must be the case that the region $R_i \cap C$ contains a horizontal line segment whose length is equal to width($R_i \cap C$) (a similar fact was previously proven in Claim 4.1 of \cite{minCostLoadBalancing}). Thus, we can safely conclude that, if Algorithms \ref{alg:ConstructFW} and \ref{alg:SubdivideFW} produce $k$ rectangles with equal areas and dimensions $w_i \times h_i$, where $w_i$ is the length of the longer edge, a valid upper bound for $\FW(C, k)$ is indeed  $ \sum_i \Phi_{UB}(A_i,w_i, h_i)$, where $A_i = \text{Area}(R_i \cap C)$. By (\ref{eq:phiub}) we have
\[
\FW(C,k) = \sum_{i=1}^{k} \FW(C\cap R_i) \leq \sum_{i=1}^{k} \Phi_{UB}(A_i,w_i,h_i) 
\] 
It is not hard to verify algebraically that, for any fixed $A_i$, the function $\Phi_{UB}\left(A_{i},w_{i},h_{i}\right)$ is maximized when the aspect ratio of $R_i$ is as large as possible. Let $\beta=\max_{i=1,...,k} \AR(R_i)$. Thus, for an upper bound, and using Lemma \ref{lem:UB}, we can replace $w_i$ with $\sqrt{2A_i\beta}$ and $h_i$ with $\sqrt{2A_i/\beta}$. By observing that $\Phi_{UB}$ is concave in $A_i$, we ultimately conclude that:
\begin{equation}
\label{eq:GeneralUpperrBound}
\FW(C,k)\leq k \cdot \Phi_{UB}(\frac{A}{k},\sqrt{\frac{2A\beta}{k}},\sqrt{\frac{2A}{k\beta}}\,),
\end{equation}
We can derive a more explicit upper bound for $\FW(C, k)$ by considering the relationship between $k$ and $w$, as described in the next section.

\subsection{Approximation Factor of Algorithms \ref{alg:ConstructFW} and \ref{alg:SubdivideFW}}
\label{sec:ApproxFactors}

We analyze each algorithm by breaking all possibilities into several cases in Sections \ref{subsubsec:ApxFactorConstructAlg} and \ref{subsubsec:ApxFactorSubdivideAlg}. The following theorem summarizes the results of this analysis.  
\begin{thm}
The solution to the $k$-medians problem provided by Algorithm \ref{alg:ConstructFW} and Algorithm \ref{alg:SubdivideFW} are within factor 2.002 of optimality.
\end{thm}

\subsubsection{Approximation Factor of Algorithm \ref{alg:ConstructFW}}
\label{subsubsec:ApxFactorConstructAlg}

We break our analysis into three possible cases: $0.5 \leq w/k \leq 1000$, $w/k > 1000$, and $w/k < 0.5$. For simplicity in the rest of the analysis, let $\alpha = A/k$, and $z=w/k$. Approximation factors for each of these cases are as follows:\\
~\\ 
\textbf{CASE I $\bm{(0.5 \leq z \leq 1000}$):} 
By Lemma \ref{lem:SquarifiedSliceCase}, in this case, our algorithm divides the bounding box into $k$ rectangles each with dimensions $\frac{w}{k} \times 1$. 
For the upper bound we have
\[
\FW(C,k) \leq k \cdot \Phi_{UB}(\frac{A}{k},\frac{w}{k},1) = k \cdot \Phi_{UB}(\alpha,z,1)
\]
For the lower bound we have $\FW(C,k) \geq k \cdot \Phi_{LB}(\frac{A}{k},1) = k\cdot \Phi_{LB}(\alpha,1)$. 
To pick the best lower bound in (\ref{eq:LowerBound}), we investigate two cases depending on the value of $\alpha$. \\ 
~\\
\textbf{CASE I.1 ($\bm{\alpha \leq \frac{\pi}{4}}$):}
The smallest bounding box of a convex polygon has an area that is no greater than twice the area of the polygon. Thus, we have the inequality $z/2 \leq \alpha \leq \min(z, \pi/4)$, where $z$ is bounded above by $\pi/2$. The lower bound is derived from  (\ref{eq:FirstLB}), i.e., $\Phi_{LB}(C)$ comes from the first case in (\ref{eq:LowerBound}). Computational analysis demonstrates that in this case the approximation factor $\rho=\Phi_{UB}/\Phi_{LB}$ is bounded above by 1.88, which occurs at $z = 0.5$ and $\alpha = z/2 = 0.25$ (see Figure \ref{fig:lb1_case}). Note that while Figure \ref{fig:lb1_case} shows a spike in the ratio $\Phi_{UB}/\Phi_{LB}$ around $z=1$, the algorithm actually generates the results closer to optimality when $z=1$. This discrepancy arises because the upper bound we use in our proofs amplifies the ratio near $z=1$. 
\begin{figure*}[!ht]
\centering
  \includegraphics[scale=0.65]{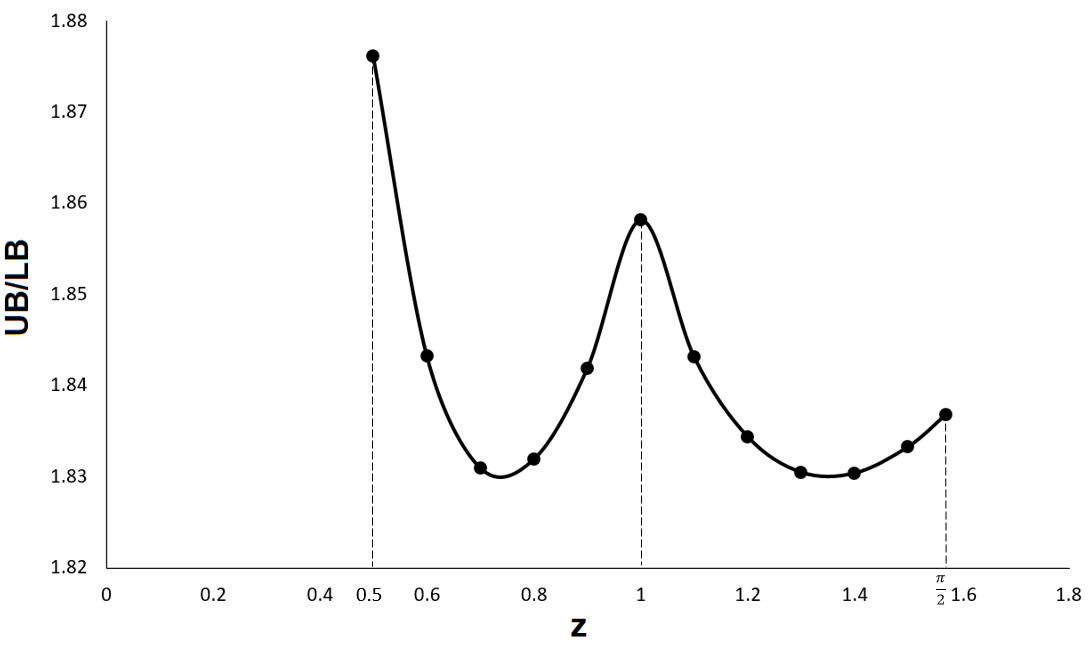}
    \protect\caption{\label{fig:lb1_case} \footnotesize The approximation factor of Algorithm \ref{alg:ConstructFW} when $\alpha \leq \dfrac{\pi}{4}$, $z \geq 0.5$ and $\alpha = \dfrac{z}{2}$.}
\end{figure*}

\noindent\textbf{CASE I.2 $\bm{(\alpha > \frac{\pi}{4}})$:}
For this case, we can establish that $z \geq \frac{\pi}{4}$ due to the given condition that $z/2 \leq \alpha \leq z$. In this scenario, the value of $\Phi_{LB}(C)$ corresponds to the second case in (\ref{eq:LowerBound}). Computational analysis shows, as illustrated in Figure \ref{fig:lb2_case}, that in this case $\rho \leq 2$ and the maximum for each value of $z$ occurs at $\alpha = z/2$.
\begin{figure*}[!ht]
\centering
  \includegraphics[scale=0.75]{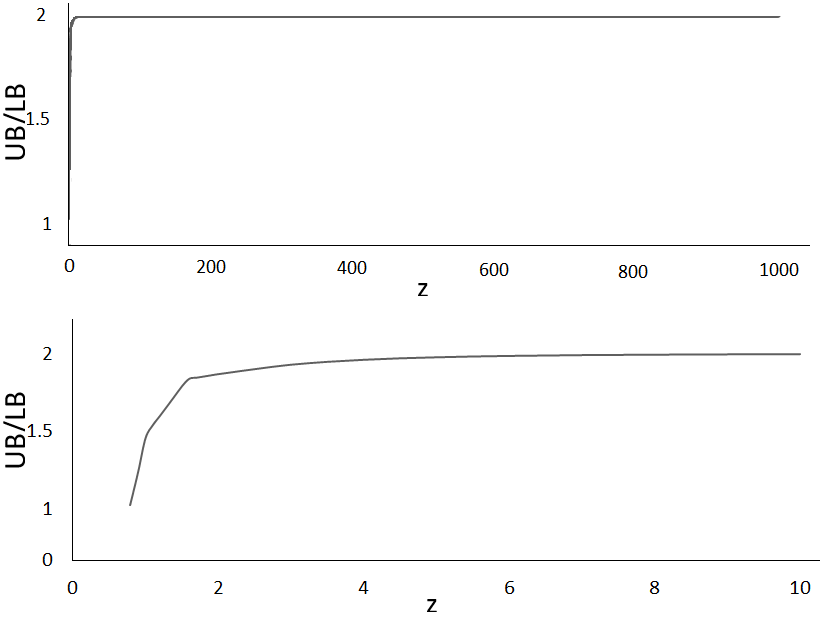}
    \protect\caption{\label{fig:lb2_case} \footnotesize The approximation factor of Algorithm \ref{alg:ConstructFW} when $\alpha > \dfrac{\pi}{4}$, $0.5 \leq z \leq 1000$ and $\alpha = \dfrac{z}{2}$. Bottom figure shows the bounds ratio for $z \leq 10$.}
\end{figure*}
\\
~\\ 
\textbf{CASE II $\bm{( z > 1000)}$:} 
In this case, similar to the last case our algorithm divides the bounding box into $k$ rectangles each with dimensions $\frac{w}{k} \times 1$. However, since the bounding box $\square C$ is very long and skinny we can approximate $\Phi_{UB}$ using the $L_1$ norm and $\Phi_{LB}$ using $L_{\infty}$ norm. As the value of $z$ increases and consequently, the aspect ratio of the sub-rectangles becomes larger, the value of ${\FW}_{\rectang}(z,1)$ approaches ${\FW}_{1/2}(z,2)$. This phenomenon occurs because when the width is significantly larger than the height of the rectangle, the change of the Fermat-Weber due to displacement of Fermat-Weber point along the rectangle's height is minimal. Hence, we can see $\alpha_c$ defined as $A_c$ in (\ref{eq:phiub}), i.e., ${\FW}_{1/2}\left(z,\frac{2\alpha_{c}}{z}\right)={\FW}_{\rectang}\left(z,1\right)$, for a large $z$ will becomes equal to $z$, 
i.e., $\alpha_c = z$. 
Hence, the upper bound is the concave envelope of $\FW_{1/2}\left(z,\frac{2\alpha_{c}}{z}\right)$ or $\sfrac{\alpha}{z}\,{\FW}_{\rectang}(z,1)$ for all values of $\alpha$. Under an $L_1$ norm, we have ${\FW}_{\rectang}(z,1) = 0.25(z^2+z)$ 
(from Equation (\ref{eq:FWRectangleL1}) in Lemma \ref{lem:FWRectangle} in Section \ref{sec:FWCommonObjects} of the Online Supplement). The upper bound, therefore is precisely:
\[
\Phi_{UB}(\alpha,z,1)= \frac{\alpha}{z}0.25(z^2+z) = 0.25 \alpha (z+1)
\]

For the lower bound, it is easy to see that instead of the slab-shape in Figure \ref{fig:hexagon1}, when we use the $L_\infty$ norm, the shape within our slab of height 1 with minimal Fermat-Weber value is simply a rectangle with dimensions $\alpha \times 1$, whose Fermat-Weber value is $0.25\alpha^2$ (from Equation (\ref{eq:FWRectangleLinfty}) in Lemma \ref{lem:FWRectangle} in Section \ref{sec:FWCommonObjects} of the Online Supplement). Therefore, for the lower bound we have
\[
\Phi_{LB}(\alpha,1) = 0.25\alpha^2
\]
The approximation factor is therefore:
\[
\rho = \frac{\Phi_{UB}}{\Phi_{LB}}= \frac{0.25\alpha(z+1)}{0.25\alpha^2} = \frac{z+1}{\alpha} < 2.002\,,
\]
since $\alpha \geq z/2$ and $z > 1000$.

~\\
\textbf{Case III $\bm{(z < 0.5)}$:}
Theorem \ref{thm:SquarifiedAspectRatioUB} shows that Algorithm \ref{alg:SquarifiedPartition} divides the bounding box into $k$ rectangles, while the aspect ratio of all except one rectangle is less than or equal to 2 and the aspect ratio of the last rectangle is at most 3.

\normalsize
  
We know that once the aspect ratio of the sub-rectangles get larger, the Fermat-Weber value and hence the approximation factor of our algorithm increases. Here, we consider the worst case for the aspect ratios of the sub-rectangles and find the approximation factor. According to Theorem \ref{thm:SquarifiedAspectRatioUB}, all sub-rectangles, except one, have an aspect ratio of 2 or less, while the last sub-rectangle has an aspect ratio of at most 3. Therefore, we consider a configuration with $k-1$ sub-rectangles of dimensions $\sqrt{2z}\times \sqrt{z/2}$ (having an aspect ratio of 2) and one sub-rectangle with dimensions $\sqrt{3z}\times \sqrt{z/3}$ (having an aspect ratio of 3). Since the aspect ratio of the sub-rectangles are not identical in this worst-case scenario, we initially fix the area of the polygon contained within the sub-rectangle with an aspect ratio of 3. Then, assuming that the other sub-rectangles are identical, we proceed with the analysis similar to the previous case. We denote $\alpha_1$ as the area of the polygon within the sub-rectangle with an aspect ratio of 3 and $\alpha_2$ as the area of the polygon within the other sub-rectangles. Recall that this area must be equal within those sub-rectangles in order to maximize the approximation factor. 
Hence, $\alpha_1 \leq z$ and $\frac{A-\alpha_1}{k-1} \leq \alpha_2 \leq z$.
The upper bound is calculated by summing the upper bound for the sub-rectangle with aspect ratio 3 and the other sub-rectangles.
\[
\Phi_{UB} = \Phi_{UB_{(\AR_{\rectang}=3)}} + (k-1) \cdot \Phi_{UB_{(\AR_{\rectang}=2)}} 
\]

In the case of a rectangle with an aspect ratio of 3, the value of $\alpha_c$ can be calculated as $0.8844z$, where $\alpha_c$ is defined by the equation $\alpha_c:{\FW}_{1/2}(z, \frac{2\alpha_c}{z})={\FW}_{\rectang}(z,1)$. Consequently, when we set $\alpha_1$ as a fixed value, we compare it with $\alpha_c$. If $\alpha_1$ is less than $\alpha_c$, we consider $\frac{\alpha_1}{0.8844z}{\FW}_{\rectang}(\sqrt{3z},\sqrt{z/3})$ as the upper bound. However, if $\alpha_1$ is greater than or equal to $\alpha_c$, we consider ${\FW}_{\rectang}(\sqrt{3z},\sqrt{z/3})$ as the upper bound. 
In the case of a rectangle with an aspect ratio of 2, the value of $\alpha_c$ can be calculated as $0.8405z$. Then for $\alpha_2 \leq 0.8405z$, we consider $\frac{\alpha_2}{0.8405z}{\FW}_{\rectang}(\sqrt{2z},\sqrt{z/2})$ as the upper bound and for $\alpha_2 > 0.8405z$ we consider ${\FW}_{\rectang}(\sqrt{2z},\sqrt{z/2})$. Therefore, the upper bound is
\[
\resizebox{\textwidth}{!}{%
$\Phi_{UB} = \min\left\{\frac{\alpha_1}{0.8844z}{\FW}_{\rectang}(\sqrt{3z},\sqrt{\dfrac{z}{3}}), {\FW}_{\rectang}(\sqrt{3z},\sqrt{\dfrac{z}{3}})\right\}+ (k-1) \cdot \min\left\{\frac{\alpha_2}{0.8405z}{\FW}_{\rectang}(\sqrt{2z},\sqrt{\dfrac{z}{2}}), {\FW}_{\rectang}(\sqrt{2z},\sqrt{\dfrac{z}{2}})\right\}$
}
\]
For the lower bound we use the first case in (\ref{eq:LowerBound}) that is $\Phi_{LB}=\frac{2}{3}\pi r^3$. 
Recall that in order to determine the lower bound, we must first have knowledge of the polygon's area and divide it to $k$ to find the value of $\alpha$. Then we equate $\alpha$ to $\pi r^2$ to solve for $r$ and finally calculate the lower bound. In this context, 
\[
\alpha = \frac{\alpha_1+(k-1)\alpha_2}{k},
\]
and it is easy to see that if $k$ increases, the effect of the sub-rectangle with aspect ratio of 3 decreases. So, the worst case happens for $k=3$. 
Computational analysis shows that in this situation, the maximum approximation factor happens for $\alpha_1 = 0.8844z$ and $\alpha_2 = \dfrac{A-\alpha_1}{2} = 
 $ and its value is 1.9614 (i.e., $\rho \leq 1.9614$). Figure \ref{fig:k_eq_3} shows the maximum approximation factor $\rho$ for values of $\alpha_1$. \\

\begin{figure}[hbt]
\centering
  \includegraphics[scale=0.3]{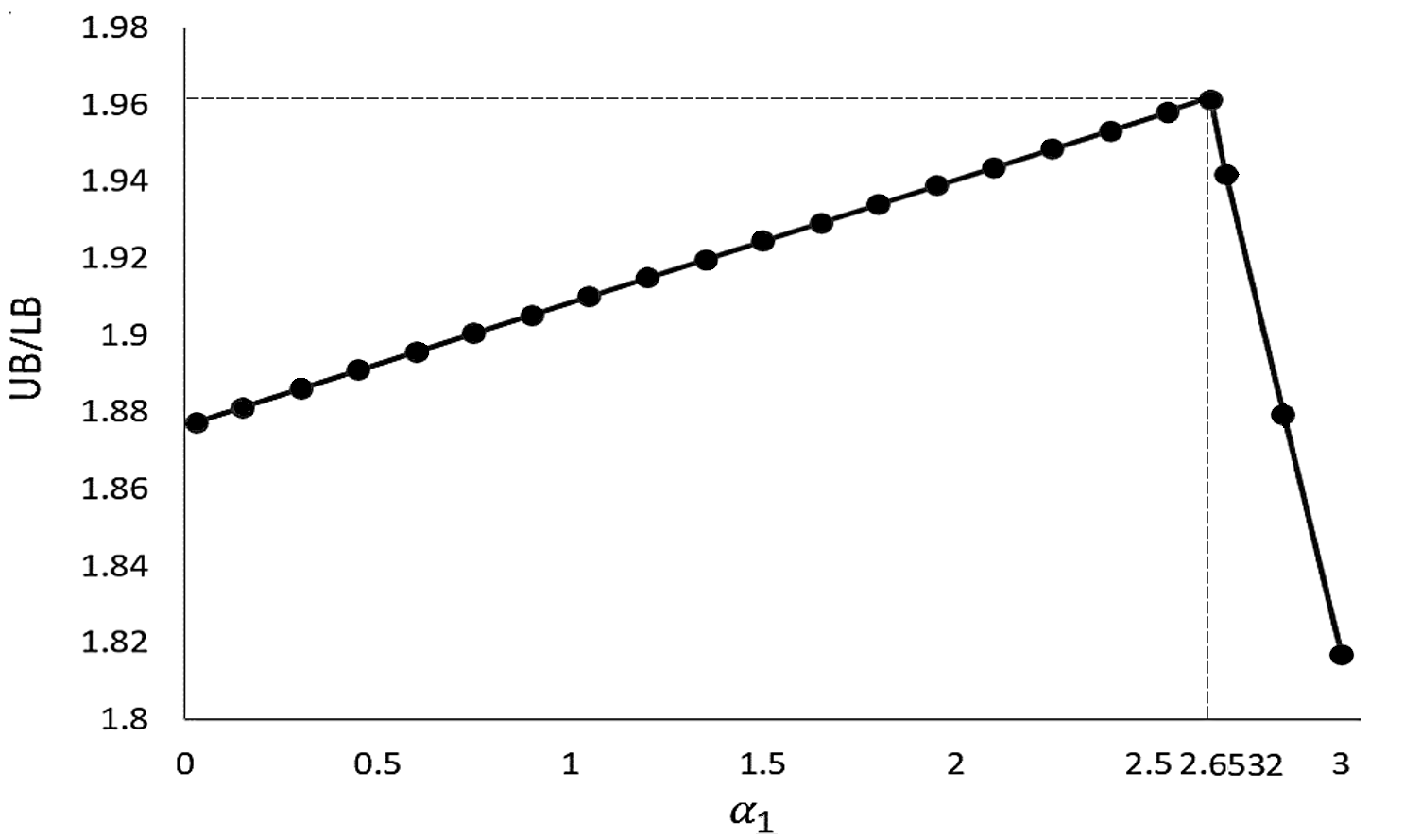}
    \protect\caption{\label{fig:k_eq_3} \footnotesize The approximation factor of Algorithm \ref{alg:ConstructFW} when $k=3$ and $\alpha_2 = \frac{(k/2)z-\alpha_1}{k-1} = \frac{1.5z-\alpha_1}{2}$.}
\end{figure}

\subsubsection{Approximation Factor of Algorithm \ref{alg:SubdivideFW}}
\label{subsubsec:ApxFactorSubdivideAlg}
To find the approximation factor of this algorithm, we investigate several cases that are summarized in \cref{tab:SubdivisionAlgCases}. The last column shows the maximum approximation ratio for each case.

\begin{table}[h]
\vspace{8pt}
\caption{\label{tab:SubdivisionAlgCases} Different investigated cases for finding the approximation factor of Algorithm \ref{alg:SubdivideFW}.}
\begin{center}
\begin{tabular}{ccc}
 \toprule
~~~~ \textbf{Case} ~~~~~~~~&~~~~~~~~ \textbf{Condition} ~~~~~~~~ &~~~~~~~~ $\bm{\rho_{\max}}$ ~~~~~~~~\\ \hline
 CASE 1 &   $w/k\leq 1/9$ & 1.8530 \\    \hline
 CASE 2 &   $1/9\leq w/k < 1/8$ & 1.8446 \\  \hline
 CASE 3 &   $1/8\leq w/k < 2/9$ & 1.8765 \\  \hline
 CASE 4 &   $2/9\leq w/k \leq 1/4$ & 1.8408 \\  \hline
 CASE 5 &   $1/4 < w/k < 0.5$ & 1.9614 \\  \hline 
 CASE 6 &   $w/k \geq 0.5$ & 2.002 \\
\bottomrule 
\vspace{-30pt}
\end{tabular}
\end{center}
\end{table}

~\\
\textbf{CASE 1 $\bm{(w/k\leq 1/9)}$:}
In this case, we have $q_0 =\lfloor{\sqrt{hk/w}}\rfloor = \lfloor{\sqrt{k/w}}\rfloor \geq 3$. By Lemma \ref{lem:claimi}, we know that at least one of the configurations made with  $p=p_0$ or $p=p_0+1$, has $q=q_0$. The maximum aspect ratio of created sub-rectangles is
\[
\AR=\max\{\frac{w(q+1)^2}{k},\frac{k}{wq^2}\}
\]
Suppose that $q=q_0$. We have:
\[
\frac{w(q_0+1)^2}{k} \leq \frac{w(\sqrt{\dfrac{k}{w}}+1)^2}{k} = (1+\sqrt{\frac{w}{k}})^2 \leq \frac{16}{9}
\]
Moreover, we have:
\[
q_0 = \lfloor{\sqrt{k/w}}\rfloor \Longrightarrow
q_0^2 \leq \frac{k}{w}< (q_0+1)^2 \Longrightarrow 1 \leq \frac{k}{wq_0^2}<\dfrac{(q_0+1)^2}{q_0^2}\Longrightarrow \frac{k}{wq_0^2} < \frac{16}{9}
\]
So, using the maximum aspect ratio $\AR=16/9$ for sub-rectangles in this case and the first lower bound in \eqref{eq:LowerBound}, we will have $\alpha_c = 0.8279z$ and the maximum approximation factor is for $\alpha = z/2$ with value 1.8530.

~\\
\textbf{CASE 2 $\bm{(1/9\leq w/k < 1/8)}$:} 
In this case, we have $q_0=2$. One of the configurations that our algorithm generates, is achieved by setting $q=q_0+1=3$ and $p=\lfloor\frac{k}{3}\rfloor$. If this configuration is made of two grids, its worst-case aspect ratio is
\[
\AR=\max\{\frac{wk}{p^2},\frac{(p+1)^2}{wk}\}
\]
We investigate this case depending on different values of $\bmod(k,q=3)$.
\begin{itemize}
    \item $\bm{\bmod(k,3)=0}$
\[
\bmod(k,3)=0 \Longrightarrow p=\frac{k}{3}\xrightarrow{s=0} \AR=\max\{\frac{wk}{k^2/9},\frac{k^2/9}{wk}\}=\max\{\frac{9w}{k},\frac{k}{9w}\} \leq 1.125
\]
Note that in this case we have $s=0$ and $\ell=0$ and Algorithm \ref{alg:GridPartitionAlg} generates only one grid. For this case, $\alpha_c =0.7836z$ and the maximum approximation factor is for $\alpha = z/2$ with value 1.8408.
    
     \item $\bm{\bmod(k,3)=1}$
\[
\bmod(k,3)=1 \Longrightarrow p=\frac{k-1}{3}\Longrightarrow \AR=\max\{\frac{9kw}{(k-1)^2},\frac{(k+2)^2}{9kw}\}
\]
     Since $w\geq 1$ and $w/k<1/8$, we have $k \geq 10$. Using this inequality, we will have $AR\leq 1.44$. Hence, $\alpha_c = 0.8064z$ and the maximum approximation factor is for $\alpha = z/2$ with value 1.8316.

      \item $\bm{\bmod(k,3)=2}$
\[
\bmod(k,3)=2 \Longrightarrow p=\frac{k-2}{3}\Longrightarrow \AR=\max\{\frac{9kw}{(k-2)^2},\frac{(k+1)^2}{9kw}\}
\]
     Since $w\geq 1$ and $w/k<1/8$, we would have $k \geq 11$. Using this inequality, we will have $\AR \leq 1.68$. Hence, $\alpha_c = 0.8220z$ and the maximum approximation factor is for $\alpha = z/2$ with value 1.8446. 

\end{itemize}  
  
~\\ 
\textbf{CASE 3 $\bm{(1/8\leq w/k\leq 2/9)}$:} 
In this case, we have $q_0 = 2$. Again by Lemma \ref{lem:claimi}, we know that at least one of the configurations made with $p=p_0$ or $p=p_0+1$, has $q=q_0$. The maximum aspect ratio of sub-rectangles in that configuration is
\[
\AR=\max\{\frac{w(q_0+1)^2}{k},\frac{k}{wq_0^2}\},
\]
where
\[
9/8 \leq \frac{w(q_0+1)^2}{k}=\frac{9w}{k} \leq 2\,, \qquad \qquad \mbox{ and } \qquad \qquad 9/8 \leq \frac{k}{wq_0^2}=\frac{k}{4w} \leq 2. 
\]
So, using maximum aspect ratio $\AR=2$ for sub-rectangles, we will have $\alpha_c = 0.8406z$ and the maximum approximation factor is for $\alpha = z/2$ with value 1.8765. 

~\\ 
\textbf{CASE 4 $\bm{(2/9 < w/k\leq 1/4)}$:}
In this case, we have $q_0=2$. One of the configurations that our algorithm generates, is achieved by setting $q=q_0$ and $p=\lfloor\frac{k}{2}\rfloor$. We investigate this case based on different values of $\bmod(k,q)$.
\begin{itemize}
\item $\bm{\bmod(k,2)=0}$
\[
p=\frac{k}{2}\xrightarrow{s=0}\AR=\max\{\frac{kw}{p^2},\frac{p^2}{kw}\} = \max\{\frac{4w}{k},\frac{k}{4w}\} \leq 1.125
\]
Hence, $\alpha_c =0.7836z$ and the maximum approximation factor is for $\alpha = z/2$ with value 1.8408.
     
\item $\bm{\bmod(k,2)=1}$ ~\\
\[
\resizebox{0.95\textwidth}{!}{%
$p=\frac{k-1}{2}\Longrightarrow \AR=\max\{\frac{(p+1)^2}{kw},\frac{kw}{p^2}\}=\max\{\frac{(k+1)^2}{4kw},\frac{4kw}{(k-1)^2}\}\leq \max\{\frac{9(k+1)^2}{8k^2},\frac{k^2}{(k-1)^2}\}$
}
\]
Since $w \geq 1$ and $w/k < 1/4$, we have $ k \geq 5$ and as a result the maximum aspect ratio is 1.62.
Having this aspect ratio, $\alpha_c =0.8182z$ and the maximum approximation factor is for $\alpha = z/2$ with value 1.8403.
\end{itemize}

~\\ 
\textbf{CASE 5 $\bm{(1/4 < w/k < 0.5)}$:}
In this case we have $q_0=1$. We investigate this case based on different values of $\bmod(k,q)$.
\begin{itemize}
\item $\bm{\bmod(k,2)=0}$\\
One of the configurations that our algorithm makes in this case, is obtained by setting $q=q_0+1=2$ and $p=\lfloor\frac{k}{2}\rfloor=\frac{k}{2}$. Since $s=0$, we have
\[
\AR=\max\{\frac{kw}{p^2},\frac{p^2}{kw}\}=\max\{\dfrac{4w}{k},\frac{k}{4w}\} \leq 2
\]
In this case, using maximum aspect ratio of 2 for sub-rectangles, we will have $\alpha_c = 0.8406z$ and the maximum
approximation factor is for $\alpha = z/2$ with value 1.8765.

\item $\bm{\bmod(k,2)=1}$\\
For this case, we consider three intervals.
\begin{itemize}
\item $\bm{(1/4 < w/k\leq 8/25)}$\\
In this case, one of the configurations that our algorithm makes, is obtained by setting $q=q_0+1=2$ and $p=\lfloor\frac{k}{2}\rfloor=\frac{k-1}{2}$. The maximum aspect ratio of sub-rectangles in that configuration is
\[
\AR=\max\{\frac{kw}{p^2},\frac{(p+1)^2}{kw}\}=\max\{\frac{4kw}{(k-1)^2},\frac{(k+1)^2}{4kw}\} \leq \max\{\frac{1.28k^2}{(k-1)^2},\frac{(k+1)^2}{2k^2}\} \leq 2\,,
\]
where the upper bound comes from the fact that $w\geq 1$ and $w/k \leq 0.32$,  which gives $k \geq 5$. Therefore, the worst case of approximation factor is the same as previous case which is 1.8765.

\item $\bm{(8/25 < w/k\leq 18/49)}$\\
In this case, using the fact that $w\geq 1$ and $w/k\leq 18/49$, we can observe that $k \geq 3$, which means we have either $k=3$ or $k \geq 5$.

For $k=3$, $w \geq 1$ and $w/3 \leq 18/49$, we have $w \leq 54/49$. As a result, $p_0=\lfloor\sqrt{wk/h}\rfloor=\lfloor\sqrt{3w}\rfloor=1$. So, we have a configuration in which $p=p_0+1=2, \; q=1$ and $s=1$. This configuration makes $(p-s)q=1$ rectangle with $\AR=\max\{\frac{3}{w},\frac{w}{3}\}\leq 3$ and $(q+1)s=2$ rectangles with $\AR=\max\{\frac{4w}{3},\frac{3}{4w}\}\leq 72/49$. Recall the case that we had 2 rectangles with aspect ratio less than or equal to 2 and one rectangle with aspect ratio less than or equal to 3 that we investigated in the analysis of Algorithm \ref{alg:ConstructFW}. In that case, we showed that the maximum approximation factor is 1.9614. Here, the aspect ratio of our rectangles are less than or equal to that case. Hence, we can claim that the approximation factor is less than 1.9614.

For $k\geq 5$, the maximum aspect ratio of the generated sub-rectangles is
\[
\AR=\max\{\frac{kw}{p^2},\frac{(p+1)^2}{kw}\}=\max\{\frac{4kw}{(k-1)^2},\frac{(k+1)^2}{4kw}\} < 2.296
\]
This gives $\alpha_c = 0.8555z$ and hence $\rho \leq 1.9145$.

\item $\bm{(18/49 < w/k < 0.5)}$\\
In this case, we know that for one of the configurations where we set $p=p_0$ or $p=p_0+1$, we get $q=q_0$. Then, the aspect ratio of rectangles is less than $AR \leq \max \{\frac{4w}{k}, \frac{k}{w}\}$ which gives us the maximum of 49/18. In this case we get $\alpha_c= 0.8741z$ and $\rho \leq 1.9770$.

\end{itemize}

\end{itemize}

~\\ 
\textbf{CASE 6 $\bm{(w/k \geq 0.5)}$:}
In this case, by Lemma \ref{lem:SubdivisionAlgSliceCase}, our algorithm divides the bounding box into $k$ rectangles each with dimensions $\frac{w}{k} \times 1$. Hence, this case is identical to cases I and II of analysis of Algorithm \ref{alg:ConstructFW}, for which we showed $\rho \leq 2.002$.

\section{Computational Results}
As discussed above, both of our algorithms have very good worst-case performance. However, it is still interesting to investigate their average performance. In order to compare our algorithms and see their average performance, we present in this section a simple computational experiment using the convex hull of maps of Brookline, MA (Figure \ref{fig:Brookline-rand}) and the state of Massachusetts (Figure \ref{fig:MA-rand}) as our initial polygon. We have partitioned the polygon for a number of medians from 3 to 150 and for the cases when we have several medians out of the polygon and we randomly place them inside the polygon. We have run each case three times and use the average of these runs in our computational results. Figure \ref{fig:random-perf-Brookline_MA} shows the ratio of Fermat-Weber value of our solution using ${\sf ConstructFW}$ and ${\sf SubdivideFW}$ Algorithms to the lower bound.
\begin{figure}[htb]
    \centering
    \includegraphics[scale=0.9]{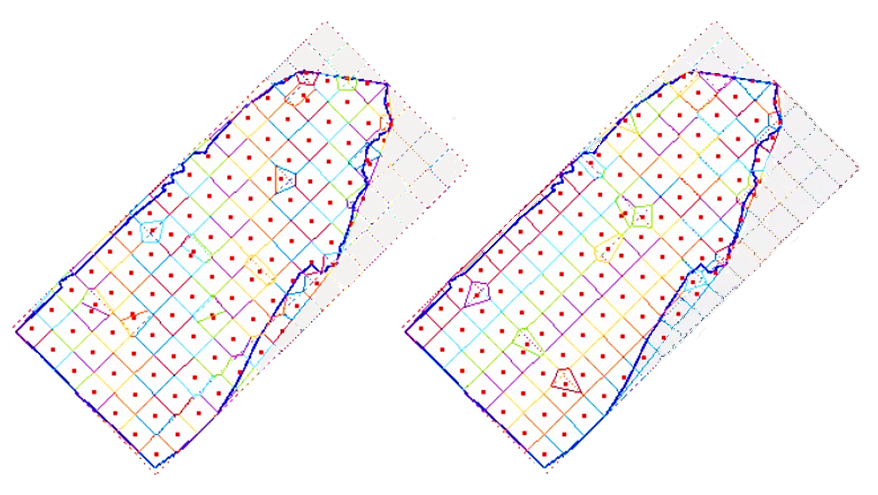}
    \protect\caption{\label{fig:Brookline-rand} The map of Brookline partitioned into 111 sub-regions using ${\sf SubdivideFW}$ Algorithm on the left and ${\sf ConstructFW}$ Algorithm on the right. The median points corresponding to the center of the smaller rectangles which are completely outside of the polygon are placed randomly in the polygon.}
\end{figure}
\begin{figure}[htb]
\captionsetup{farskip=0pt}
    \begin{centering}
    \includegraphics[scale=0.8]{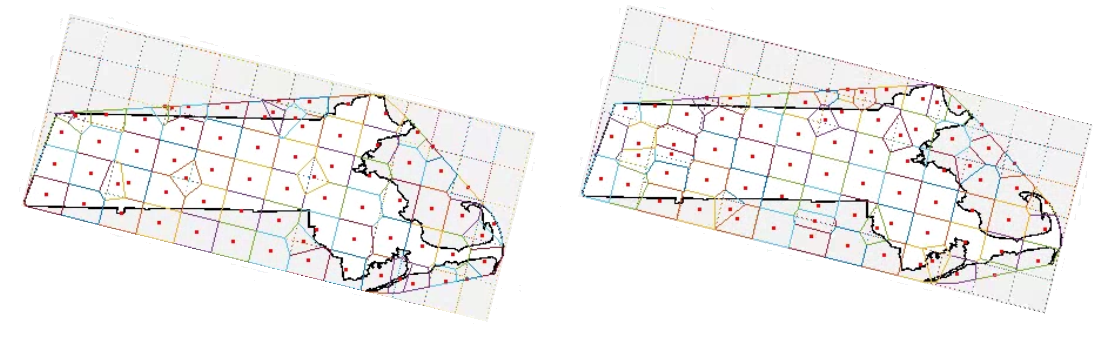} 
    \protect\caption{\label{fig:MA-rand} The map of Massachusetts partitioned into 76 parts using ${\sf SubdivideFW}$ Algorithm on the left and ${\sf ConstructFW}$ Algorithm on the right. The median points corresponding to the center of the smaller rectangles which are completely outside of the polygon are placed randomly in the polygon.}
    \end{centering}
\end{figure}
\begin{figure}[htb]
    \centering
    \includegraphics[scale=0.8]{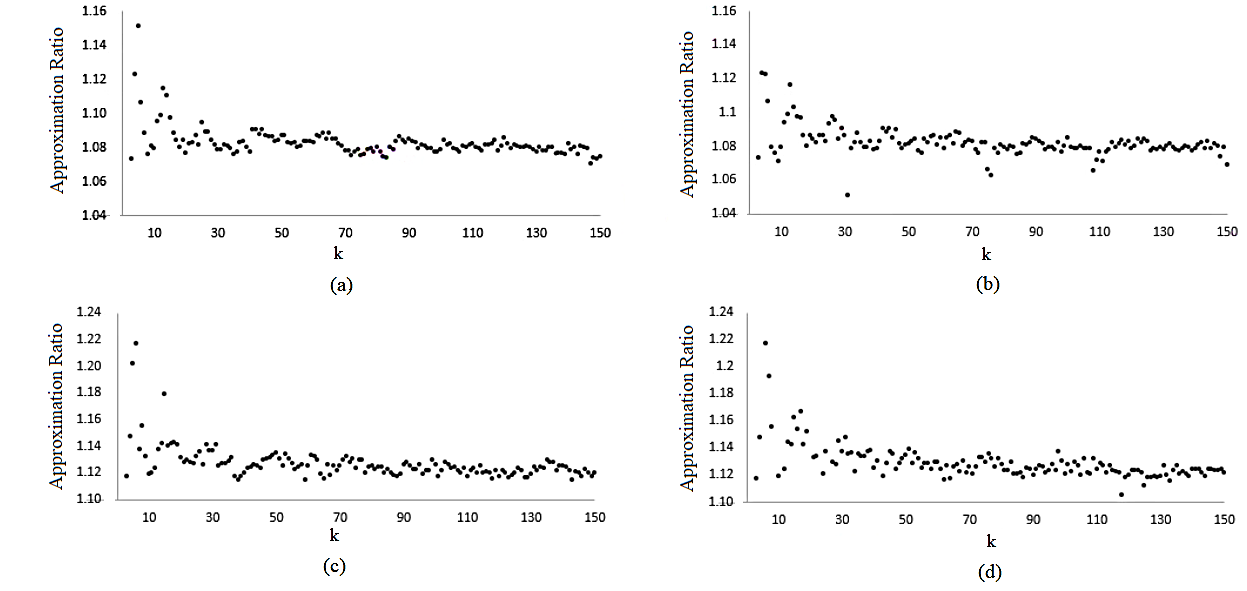}
    \caption{\footnotesize \label{fig:random-perf-Brookline_MA} The approximation factor trend with respect to the number of median points ($k$). Here, in approximation factor calculation, we use actual Fermat-Weber value of our solutions over the lower bound. Figures (a) and (b) show this ratio for ${\sf SubdivideFW}$ and ${\sf ConstructFW}$ Algorithms, respectively, on the map of Brookline, and Figures (c) and (d) show the same for the map of Massachusetts.}  
\end{figure}

\subsection{Modification of the Algorithm}
\label{sec:modification}
The computational results show a very satisfactory average performance. However, we could still do better by applying a small modification in our algorithms. In our algorithms we place the median points corresponding to those sub-rectangles that are completely outside of the polygon, randomly inside the polygon. Doing so we will only make the Fermat-Weber value better. In \cref{sec:GeneralFramework} we mentioned that the way we do the placement for such points has a minor impact on the performance of the algorithm. In this section we will try to do a bit better than random.

For the ${\sf ConstructFW}$ Algorithm, we find the sub-rectangles made by the algorithm that are fully in the polygon. For each sub-rectangle that is completely outside of the polygon, we choose a strip and then add a new median point to the rectangles fully in that strip. Assume that there are $i$ sub-rectangles in the strip that are completely inside the polygon. So, the length or width of those $i$ sub-rectangles decreases in a way that $i+1$ rectangles fit. Then the $i+1$ median points will be the center of these new sub-rectangles. The area of all new sub-rectangles would be $i/k(i+1)$ and the center of rectangles have equal distance. To select a strip amongst all strips we have for placing one new median, we use a measure for each strip and then select the strip with maximum measure value. The measure is 
\[
\frac{(CAR-NAR)\times CNR}{ONR},
\]
where CAR is the current aspect ratio of the sub-rectangles in the strip, NAR is the new aspect ratio of the sub-rectangles by adding one new median point, CNR is the current number of sub-rectangles in the strip, and ONR is the original number of sub-rectangles in the strip before adding any median. After adding one median, CAR, NAR and CNR for selected strip will be updated and the process will be repeated until all outside median points are placed in a strip inside the polygon.

We used the same measure and method for the ${\sf SubdivideFW}$ Algorithm. The only difference was that based on the direction of dividing the bounding box of the polygon into two sections, we use rows or columns that are generated, instead of strips. For example if we have divided the box with a vertical (horizontal) line into two parts, we consider columns (rows) instead of the strips.
Figures \ref{fig:Brookline-kcrow-sqstrip} and \ref{fig:MA-kcrowstrip} show the results of using this modifcation on maps of Brookline and Massachusetts.

Applying this modification and running the algorithms for Brookline and Massachusetts maps, the ratio of Fermat-Weber value over the lower bound decreases. Figure~\ref{fig:Ratios-random-striprow} shows the ratio of the Fermat-Weber value over the lower bound  when the outer medians points are placed randomly ($x$-axis) and when they are placed using the modified point placement scheme ($y$-axis) on both maps.
One important parameter affecting the performance improvement is the number of medians that are initially outside of the polygon. The more we have the outer median points, the more potential there is for improvement by this modification. The weighted average improvement of applying this modification on the ${\sf SubdivideFW}$ Algorithm on Brookline and Massachusetts map is 2.14\% and 3.45\%, respectively. This value is 2.06\% and 3.64\% for the ${\sf ConstructFW}$ Algorithm.
\begin{figure}[!ht]
    \centering
    \includegraphics[scale=1]{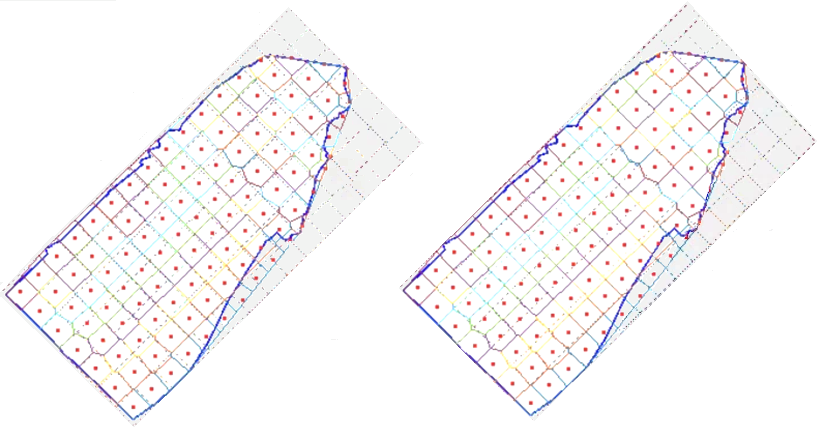}    
    \protect\caption{\label{fig:Brookline-kcrow-sqstrip} The map of Brookline partitioned into 111 parts using ${\sf SubdivideFW}$ Algorithm on the left and ${\sf ConstructFW}$ Algorithm on the right, where the median points corresponding to the center of the sub-rectangles which are completely outside of the polygon are placed in the polygon using the modified point placement scheme.}  
\end{figure}
\begin{figure}[!ht]
\captionsetup{farskip=0pt}
    \begin{centering}
    \includegraphics[scale=.8]{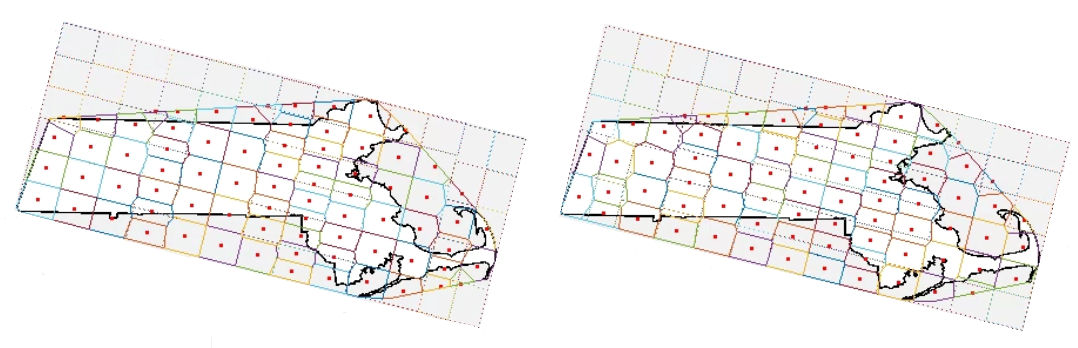}    
    \protect\caption{\label{fig:MA-kcrowstrip}The map of Massachusetts partitioned into 76 parts using ${\sf SubdivideFW}$ Algorithm on the left and ${\sf ConstructFW}$ Algorithm on the right, where the median points corresponding to the center of the sub-rectangles which are completely outside of the polygon are placed in the polygon using the modified point placement scheme.}
    \end{centering}
\end{figure}
\begin{figure}[htb]
    \centering
    \includegraphics[scale=.8]{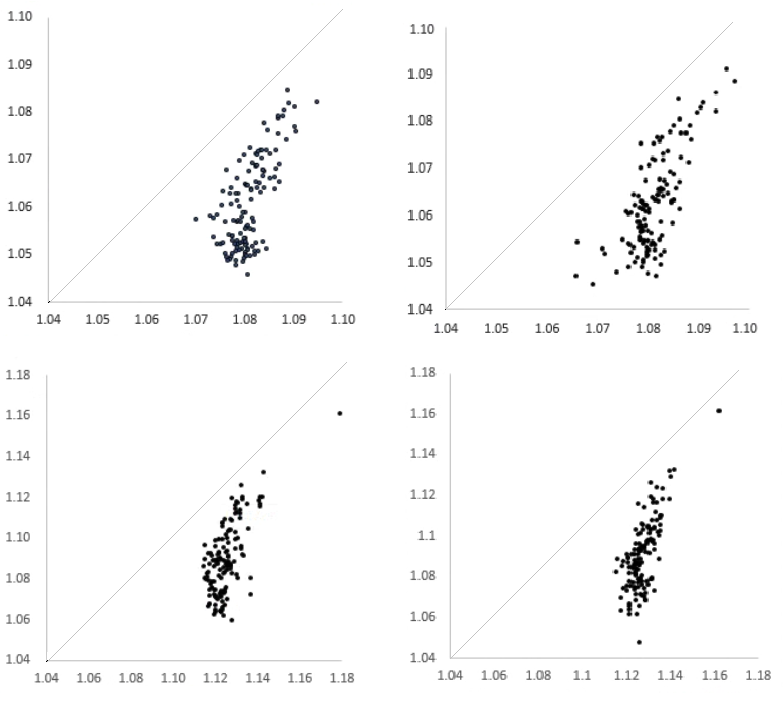}
    \protect\caption{\label{fig:Ratios-random-striprow} The ratio of Fermat-Weber value over the lower bounds when the outer median points are placed randomly ($x$-axis) and when they are placed using the the modified point placement scheme ($y$-axis). The upper figures refer to the Brookline's map and the lower refer to Massachusetts' map. The left figures show the ratios on ${\sf SubdivideFW}$ Algorithm and the right figures show it for the ${\sf ConstructFW}$ algorithm.}
    
\end{figure}

\section{Conclusion}
\label{sec:conclusion}
We have presented two approximation algorithms for finding the location of $k$ median points in a convex polygon $C$ with $n$ vertices. Our algorithms partition $C$ into $k$ sub-regions and then take the center of each sub-region as a median point. Both algorithms have a computational complexity of $\mathcal{O}(n + k + k \log n)$. This fact together with the worst-case performance analysis and the average performance results show that our algorithms are capable of finding high quality solutions efficiently. Potential directions for future research could be considering $C$ to be a polygon with obstacles (holes) or a general convex region and application of the presented algorithms to the similar location optimization problems. 

\bibliographystyle{unsrt}
\bibliography{references}

\pagebreak{}

\appendix

\part*{Online supplement to ``Approximating Median Points in a Convex Polygon''}

\section{Fermat Weber Value of Disk, Rectangle and Triangle}
\label{sec:FWCommonObjects}
\begin{lem}
\label{lem:trivial-disk}For a disk $D$ with radius $r$ and using $L_2$ norm
\[
\FW\left(D\right)=\frac{2\pi r^{3}}{3}\,.
\]
Also, for an angle $\theta\in\left[0,2\pi\right]$, if $S$ denotes
the circular sector of $D$ subtended by $\theta$, then
\[
\FW\left(S,c\right)=\frac{\theta}{3}r^{3}=\frac{2\cdot\Area\left(S\right)}{3}r
\]
where $c$ denotes the center of the disk of which $S$ is a sector.\end{lem}
\begin{proof}
Trivial.\end{proof}

\begin{rem} \label{rem:disk}
It is well-known that, for a fixed area, the disk is the region with minimal Fermat-Weber value $\FW\left(C\right)$. Using \ref{lem:trivial-disk}, this gives us an
easy lower bound:

\[
\FW\left(C\right)\geq\frac{2}{3\sqrt{\pi}}A^{3/2}
\]
where $A$ is the area of $C$.\end{rem}

\begin{lem}
\label{lem:FWRectangle} 
For a rectangle $R$ of height $h$ and width
$w$ using $L_2$ norm,
\begin{equation}
\label{eq:FWRectangleL2}
\FW(R)=\frac{1}{6}hw\sqrt{h^{2}+w^{2}}+\frac{1}{12}w^{3}\log\left(\frac{\sqrt{h^{2}+w^{2}}+h}{w}\right)+\frac{1}{12}h^{3}\log\left(\frac{\sqrt{h^{2}+w^{2}}+w}{h}\right)
\end{equation}
Also, using $L_1$ norm this is
\begin{equation}
\label{eq:FWRectangleL1}
\FW(R)=\frac{1}{4}(w^2h+wh^2),
\end{equation}
Finally, using $L_\infty$ norm and assuming $w\geq h$ this is
\begin{equation}
\label{eq:FWRectangleLinfty}
\FW(R)=\frac{1}{4}w^2
\end{equation}
\end{lem}
\begin{proof}
The above result can be obtained by evaluating the following integral
analytically:
\[
\FW(R)=\int_{\frac{-h}{2}}^{\frac{h}{2}}\int_{\frac{-w}{2}}^{\frac{w}{2}} \norm{\bm{x}}_p \, dx_{1}\, dx_{2},
\]
for $p=2,1,\infty$, respectively.
\end{proof}
\begin{lem}
\label{lem:FWTriangle} For a right triangle $T=\triangle ABC$ whose
right angle is located at the point $C$, with sides $AB=c$, $BC=a$
and $CA=b$, the Fermat-Weber value of the triangle with respect to
the point $B$ is given by
\begin{equation}
\FW(ABC)=\frac{1}{6}abc+\frac{1}{6}a^{3}\log\left(\frac{c+b}{a}\right)\label{eq:FWTriangle}
\end{equation}
\end{lem}
\begin{proof}
This is again merely analytic integration.\end{proof}

\newpage
\section{Completing proof of Lemma \ref{lem:UBHalfRectangle}}
\label{sec:sheartransformations}

We begin with Figure \ref{fig:shifted-to-right} and let $B^{'}$ denote the minimum bounding box of $C$. We iteratively
apply Lemma \ref{lem:shearing-convex} to shear the triangular components
of $C$ (all the while increasing the Fermat-Weber value of $C$ relative
to $x_{0}$, the center of $B$) until it takes the form shown in
Figure \ref{fig:two-triangles}, to the point where it remains merely
to shear the two remaining triangular components of $C$ appropriately.
We select one such triangular component arbitrarily and shear it in
a direction that increases the Fermat-Weber value of $C$ relative
to the center $x_{0}$ of $B$ (say, the point marked $x$ in Figure
\ref{fig:introduce-x-shear}), until $x$ either touches the bottom
edge of $B$ or the leftmost edge of $B$. Suppose that $x$ touches
the leftmost edge of $B$ as shown in Figure \ref{fig:x-far-left}
(the case where $x$ touches the bottom edge of $B$ can be addressed
in a similar fashion and we omit it for brevity). It is clear that
we can then shear $x$ \emph{downward} until it touches the corner
of $B$ as shown in Figure \ref{fig:introduce-xp}. This downward
shearing introduces a new vertex, which we call $x^{'}$ (also shown
in Figure \ref{fig:introduce-xp}), which we can shear in a direction
that increases the Fermat-Weber value, until it either touches the
top edge of $B^{'}$ or the left edge of $B$. Suppose that $x^{'}$
touches the left edge of $B$ (the case where $x^{'}$ touches the
top edge of $B^{'}$ is similar and we omit it for brevity), shown
in Figure \ref{fig:introduce-xpp}. This leaves us with one final
vertex, $x^{''}$, which we can shear in a direction that increases
the Fermat-Weber value of $C$ relative to $x_{0}$; note that we
are free to shear $x^{''}$ until it touches the boundary of $B$,
at which point our shape $C$ now has precisely the shape desired
as in Lemma \ref{lem:UBHalfRectangle}, as shown in Figure \ref{fig:shearing-done}.
This completes the proof.

\begin{figure}[htb]
\begin{centering}
\subfloat[\label{fig:two-triangles}]{\begin{centering}
\includegraphics[height=0.12\textheight]{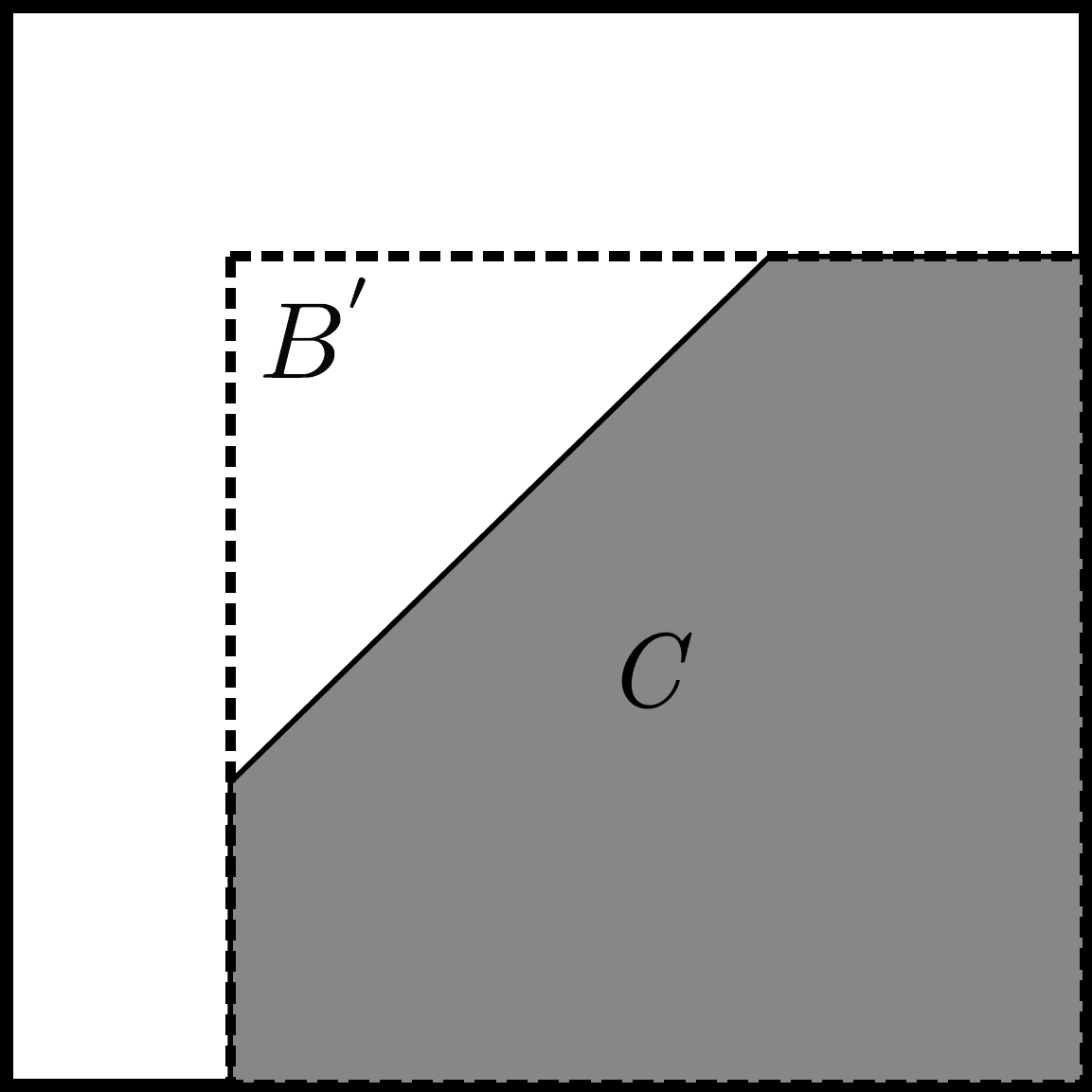}
\par\end{centering} }
\quad{}
\subfloat[\label{fig:introduce-x-shear}]{\begin{centering}
\includegraphics[height=0.12\textheight]{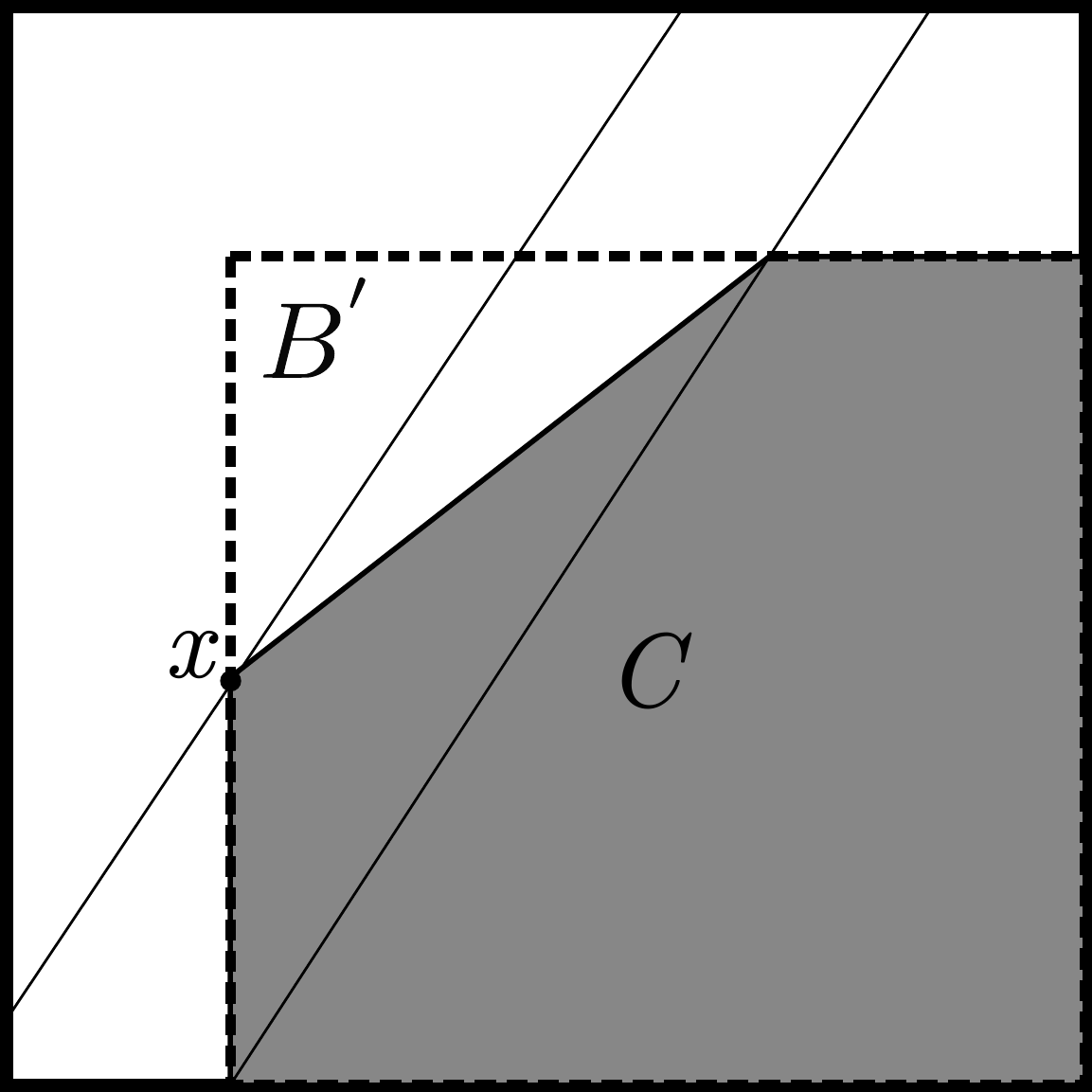}
\par\end{centering} }
\quad{}
\subfloat[\label{fig:x-far-left}]{\begin{centering}
\includegraphics[bb=16bp 0bp 348bp 332bp,height=0.12\textheight]{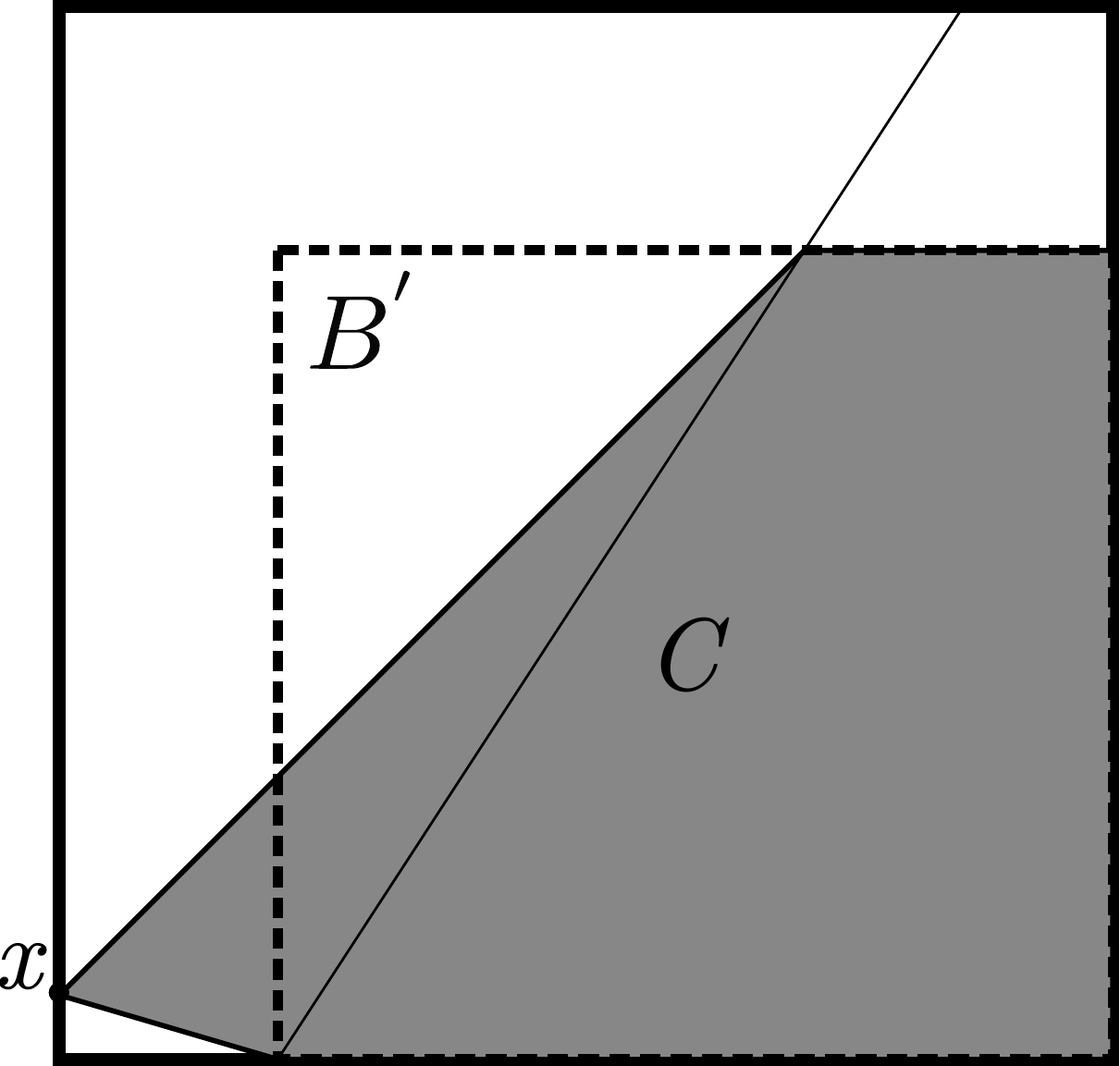}
\par\end{centering} }
\par\end{centering}

\begin{centering}
\subfloat[\label{fig:introduce-xp}]{\begin{centering}
\includegraphics[bb=16bp 0bp 349bp 333bp,height=0.12\textheight]{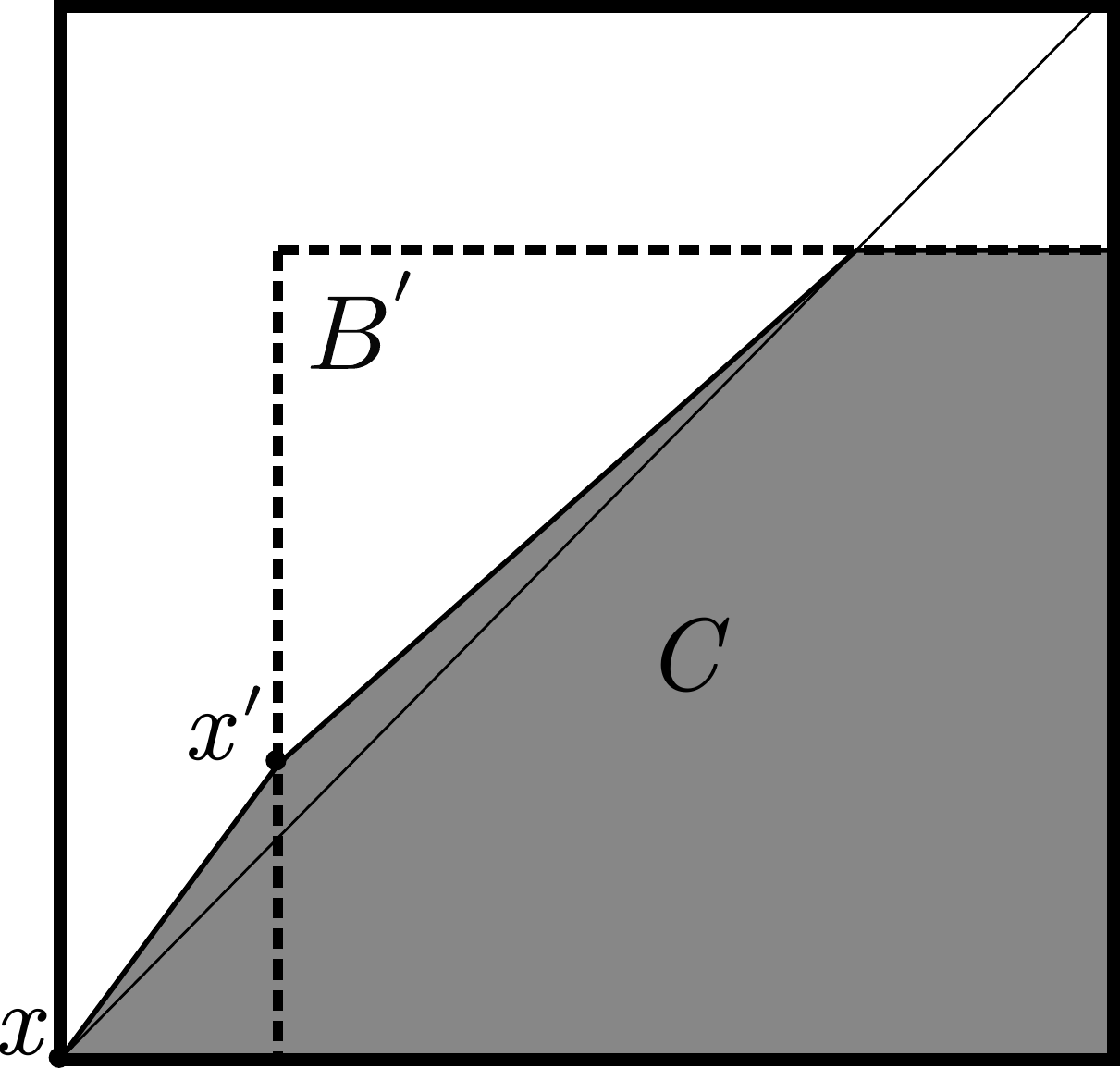}
\par\end{centering} }
\quad{}
\subfloat[\label{fig:introduce-xpp}]{\begin{centering}
\includegraphics[bb=25bp 0bp 332bp 332bp,height=0.12\textheight]{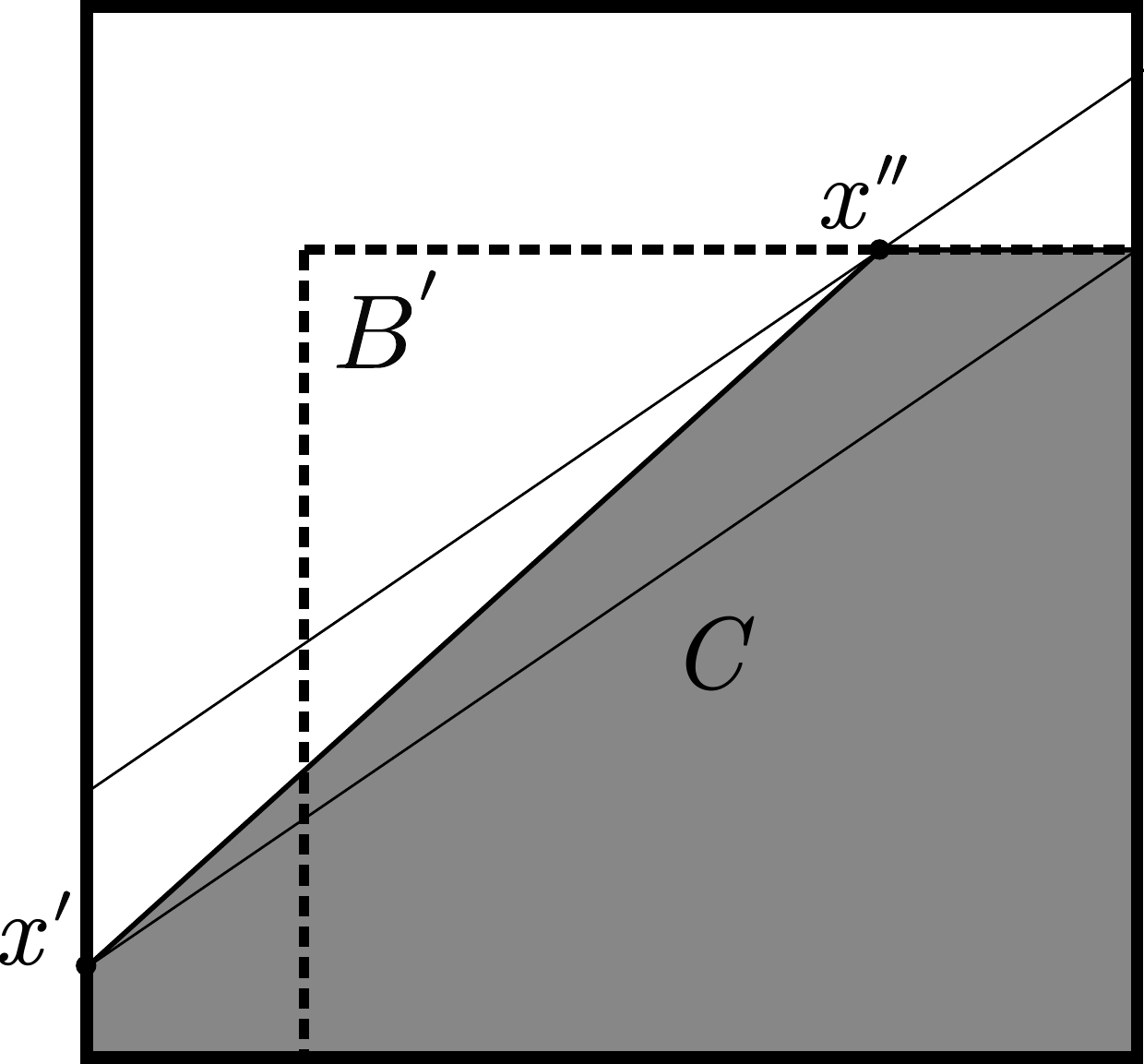}
\par\end{centering} }
\quad{}
\subfloat[\label{fig:shearing-done}]{\begin{centering}
\includegraphics[height=0.12\textheight]{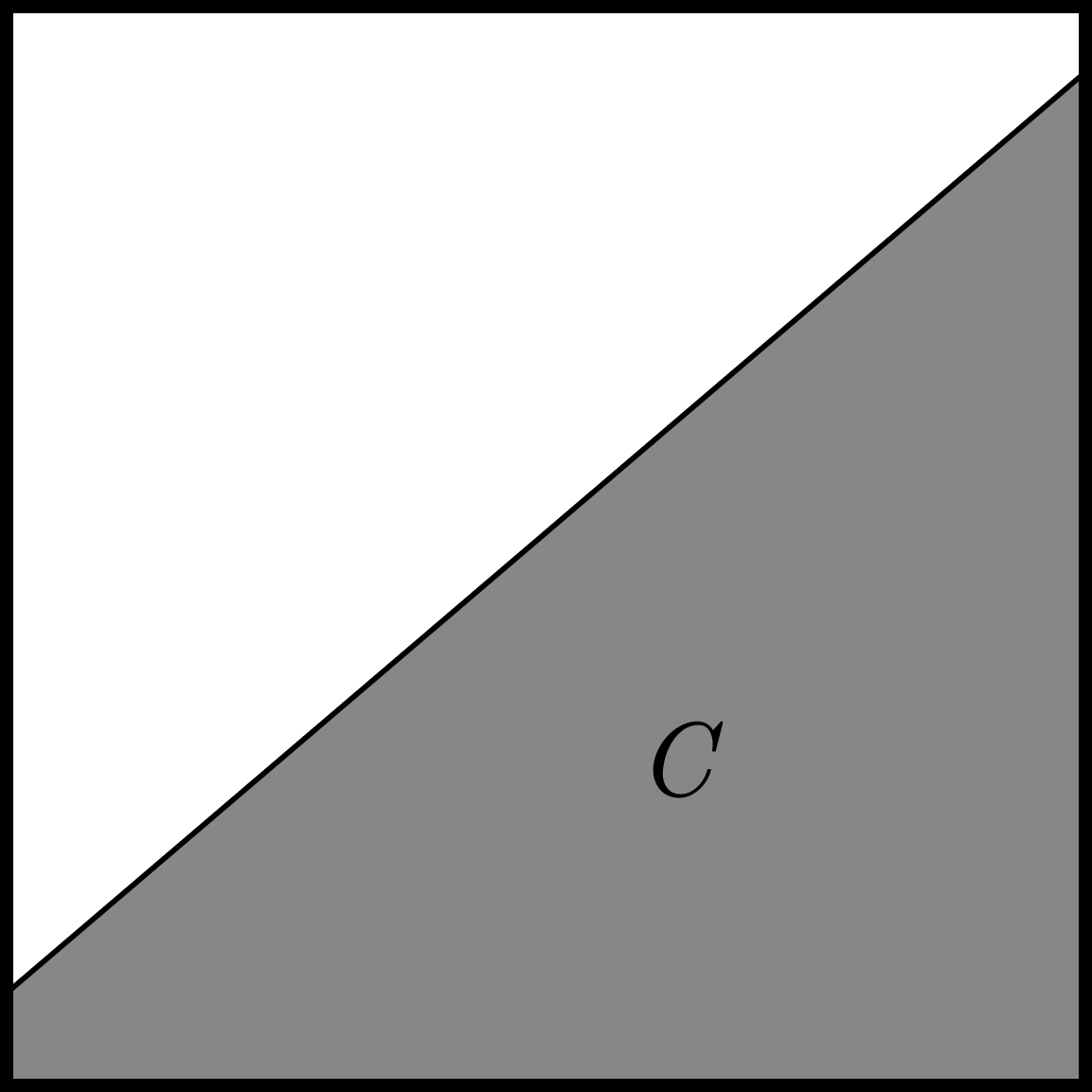}
\par\end{centering} }
\par\end{centering}

\caption{Final adjustments to $C$.}
\end{figure}

\end{document}